\documentclass[reqno,12pt]{amsart}

\usepackage{fullpage}
\usepackage[french]{babel}
\usepackage[T1]{fontenc}
\usepackage{graphicx}
\usepackage{amsmath}
\usepackage{amsfonts}
\usepackage{amssymb}
\usepackage{a4wide}

\newtheorem{theo}{Théorème}

\newtheorem*{theo2}{Théorème}
\newtheorem*{theobis}{Théorème bis}
\newtheorem*{propo2}{Critère de Landau}
\newtheorem{cor}{Corollaire}

\newtheorem{propo}{Proposition}
\newtheorem{lemme}{Lemme}

\theoremstyle{remark}
\newtheorem*{Remarque}{Remarque}
\newtheorem*{Remarques}{Remarques}
\newtheorem*{Définition}{Définition}

\numberwithin{equation}{section}

\author{E. DELAYGUE}
\title{Critère pour l'intégralité des coefficients de Taylor des applications miroir}
\date{}

\begin{document}
\maketitle

\begin{abstract}
Nous donnons une condition nécessaire et suffisante pour que les coefficients de Taylor de séries de la forme $q(z):=z\exp(G(z)/F(z))$ soient entiers, où $F(z)$ et $G(z)+\log(z)F(z)$ sont des solutions particulières de certaines équations différentielles hypergéométriques généralisées. Ce critère est basé sur les propriétés analytiques de l'application de Landau (classiquement associée aux suites de quotients de factorielles) et il généralise les résultats de Krattenthaler-Rivoal dans \textit{On the integrality of the Taylor coefficients of mirror maps}, Duke Math. J. à paraître. Pour démontrer ce critère, nous généralisons entre autres un théorème de Dwork concernant les congruences formelles entre séries formelles dans \textit{On $p$-adic differential equations IV : generalized hypergeometric functions as $p$-adic functions in one variable}, Annales scientifiques de l'E.N.S., ce dernier ne suffisant pas dans notre cas.
\end{abstract}

\section{Introduction}

\subsection{Applications miroir}\label{section mirror map}

Les applications miroir sont des séries entières $z(q)$ inverses pour la composition de séries entières $q(z):=z\exp(G(z)/F(z))$, où $F(z)$ et $G(z)+\log(z)F(z)$ sont des solutions particulières de certaines équations différentielles hypergéométriques généralisées. Elles dépendent de nombreux param\`etres et, dans les cas les plus simples, elles s'identifient à des formes modulaires classiques définies sur différents sous-groupes de congruence de $SL_2(\mathbb{Z})$. Ces dernières apparaissant naturellement dans la théorie de Schwarz des fonctions hypergéométriques de Gauss (voir \cite{Schwarz}). Les applications miroir ont connu un regain d'intérêt à la fin des années 1980, à la suite de travaux en théorie des cordes qui ont amené des physiciens à étudier des \textit{Calabi-Yau threefolds} et à construire leur variété miroir. En particulier, cette construction est associée à une application miroir dont les coefficients de Taylor permettent dans certains cas de compter les courbes rationnelles sur les \textit{Calabi-Yau threefolds} (voir \cite{Batyrev} par exemple).

Sur de nombreux exemples en symétrie miroir, il a été observé que les coefficients de Taylor des applications miroir associées sont entiers. Cette observation a été démontrée dans de nombreux cas (voir un peu plus bas dans l'introduction) et le sujet de cet article est la démonstration d'un critère d'intégralité des coefficients de Taylor des applications miroir issues de certaines équations différentielles hypergéométriques généralisées. Ces équations s'interprètent souvent comme des équations de Picard-Fuchs de familles à un paramètre d'intersections complètes de Calabi-Yau dans des espaces projectifs à poids (voir \cite{Batyrev}). Dans les cas où les applications miroir sont des formes modulaires classiques, l'intégralité de leurs coefficients de Taylor est conséquence des théorèmes de structure classiques des algèbres de formes modulaires. Néanmoins, dans le cadre plus général de cet article, la modularité au sens usuel disparaît et on doit faire appel à des techniques différentes pour aborder ces problèmes d'intégralité.

Dans la suite, si $\textbf{e}:=(e_1,\dots,e_{q_1})$ et $\textbf{f}:=(f_1,\dots,f_{q_2})$ sont deux suites d'entiers positifs, on note $|\textbf{e}|:=\sum_{i=1}^{q_1}e_i$, $|\textbf{f}\,|:=\sum_{j=1}^{q_2}f_j$ et $\mathcal{Q}_{(\textbf{e},\textbf{f}\,)}(n):=\frac{(e_1n)!\cdots(e_{q_1}n)!}{(f_1n)!\cdots(f_{q_2}n)!}$, où $n\geq 0$. On définit les séries entières
$$
F_{(\textbf{e},\textbf{f}\,)}(z):=\sum_{n=0}^{\infty}\frac{(e_1n)!\cdots(e_{q_1}n)!}{(f_1n)!\cdots(f_{q_2}n)!}z^n
$$
et
\begin{equation}\label{definition de G facto}
G_{(\textbf{e},\textbf{f}\,)}(z):=\sum_{n=1}^{\infty}\frac{(e_1n)!\cdots(e_{q_1}n)!}{(f_1n)!\cdots(f_{q_2}n)!}\left(\sum_{i=1}^{q_1}e_iH_{e_in}-\sum_{j=1}^{q_2}f_jH_{f_jn}\right)z^n,
\end{equation}
où $H_n:=\sum_{i=1}^n\frac{1}{i}$ est le $n$-ième nombre harmonique. La série $F_{(\textbf{e},\textbf{f}\,)}(z)$ est une série hypergéométrique généralisée (\footnote{On caractérisera dans la proposition \ref{propo equiv} de la partie \ref{préliminaires} les fonctions hypergéométriques généralisées dont les coefficients peuvent se mettre sous forme de factorielles.}) et est donc solution d'une équation différentielle fuchsienne. Dans certains cas que l'on étudiera (voir la fin du paragraphe \ref{section delta facto}), on obtient, \textit{via} la méthode de Frobenius (voir \cite{Schwarz}), une base des solutions de cette équation différentielle avec au plus des singularités logarithmiques à l'origine, dont $F_{(\textbf{e},\textbf{f}\,)}(z)$ et $G_{(\textbf{e},\textbf{f}\,)}(z)+\log(z)F_{(\textbf{e},\textbf{f}\,)}(z)$.
 
Dans le contexte de la symétrie miroir, la fonction $q_{(\textbf{e},\textbf{f}\,)}(z):=z\exp(G_{(\textbf{e},\textbf{f}\,)}(z)/F_{(\textbf{e},\textbf{f}\,)}(z))$ est une \textit{coordonnée canonique} et son inverse pour la composition $z(q)$ est une \textit{application miroir}. Le but de cet article est d'établir une condition nécessaire et suffisante pour l'intégralité des coefficients des applications miroir $z(q)$, c'est-à-dire de déterminer sous quelles conditions on a $z(q)\in\mathbb{Z}[[q]]$. Dans le contexte de théorie des nombres de cet article, l'application miroir $z(q)$ et la coordonnée canonique correspondante $q(z)$ jouent exactement le même rôle car $q(z)\in z\mathbb{Z}[[z]]$ si et seulement si $z(q)\in q\mathbb{Z}[[q]]$ (voir \cite[Introduction]{Lian Yau 2}). On formulera donc le critère pour $q(z)$ uniquement mais il vaut aussi pour $z(q)$.

\subsection{\'Enoncé du critère}\label{Enonce du crit}

Avant d'énoncer le critère d'intégralité des coefficients de Taylor de $q(z)$, nous rappelons la définition de \textit{l'application de Landau} associée à un quotient de factorielles. \'Etant données $\textbf{e}:=(e_1,\dots,e_{q_1})$ et $\textbf{f}:=(f_1,\dots,f_{q_2})$ deux suites d'entiers positifs, on note $\Delta_{(\textbf{e},\textbf{f}\,)}$ la fonction de Landau associée à $\mathcal{Q}_{(\textbf{e},\textbf{f}\,)}$ définie, pour tout $x\in\mathbb{R}$, par 
$$
\Delta_{(\textbf{e},\textbf{f}\,)}(x):=\sum_{i=1}^{q_1}\lfloor e_ix\rfloor-\sum_{j=1}^{q_2}\lfloor f_jx\rfloor,
$$
où $\lfloor\,.\,\rfloor$ désigne la fonction partie entière. La fonction $\Delta_{(\textbf{e},\textbf{f}\,)}$ est constante par morceaux. Notons $M_{(\textbf{e},\textbf{f}\,)}:=\max\{e_1,\dots,e_{q_1},f_1,\dots,f_{q_2}\}$. La fonction $\Delta_{(\textbf{e},\textbf{f}\,)}$ est nulle sur $[0,1/M_{(\textbf{e},\textbf{f}\,)}[$ car $0\leq e_i/M_{(\textbf{e},\textbf{f}\,)}\leq 1$ et $0\leq f_j/M_{(\textbf{e},\textbf{f}\,)}\leq 1$. La proposition suivante montre que la fonction de Landau permet de caractériser les suites $\textbf{e}$ et $\textbf{f}$ telles que, pour tout $n\in\mathbb{N}$, $\mathcal{Q}_{(\textbf{e},\textbf{f}\,)}(n)$ est entier.

\begin{propo2}[Landau, Bober]\label{propo Landau}
Soit $\textbf{e}$ et $\textbf{f}$ deux suites d'entiers strictement positifs disjointes. On a la dichotomie suivante.
\begin{itemize}
\item[\textup{$(i)$}]{Si, pour tout $x\in[0,1]$, on a $\Delta_{(\textbf{e},\textbf{f}\,)}(x)\geq 0$, alors, pour tout $n\in\mathbb{N}$, on a $\mathcal{Q}_{(\textbf{e},\textbf{f}\,)}(n)\in\mathbb{N}$.}
\item[\textup{$(ii)$}]{S'il existe un $x\in[0,1]$ tel que $\Delta_{(\textbf{e},\textbf{f}\,)}(x)\leq -1$, alors il n'existe qu'un nombre fini de nombres premiers $p$ tels que tous les termes de la suite $\mathcal{Q}_{(\textbf{e},\textbf{f}\,)}$ soient dans $\mathbb{Z}_p$.}
\end{itemize}
\end{propo2}

\begin{Remarque}
Le point $(i)$ est dû à Landau qui énonce une condition nécessaire et suffisante dans \cite{Landau}, et le point $(ii)$ est dû à Bober (voir \cite{Bober}). 
\end{Remarque}

Dans la littérature, on peut distinguer trois résultats établissant l'intégralité des coefficients de Taylor d'applications miroir appartenant à des ensembles de plus en plus grand. 

Le premier résultat a été démontré par Lian et Yau dans \cite{Lian Yau}, dans le cas où $\mathcal{Q}_{(\textbf{e},\textbf{f}\,)}(n)=\frac{(pn)!}{(n!)^p}$, où $p$ est un nombre premier. 

Ce résultat a été généralisé par Zudilin dans \cite{Zudilin}. Soit $N\in\mathbb{N}$ et $N=p_1^{a_1}\dots p_{\ell}^{a_{\ell}}$ sa décomposition en facteurs premiers. On note $A_N$ et $B_N$ les multi-ensembles (\footnote{Un multi-ensemble est un ensemble dans lequel on autorise les répétitions des éléments.}) définis par $A_N:=\{N,\frac{N}{p_{j_1}p_{j_2}},\frac{N}{p_{j_1}p_{j_2}p_{j_3}p_{j_4}},\dots\}_{1\leq j_1<j_2<\dots\leq \ell}$ et $B_N:=\{1,\dots,1,\frac{N}{p_{j_1}},\frac{N}{p_{j_1}p_{j_2}p_{j_3}},\dots\}_{1\leq j_1<j_2<\dots\leq\ell}$, avec un nom\-bre de $1$ dans $B_n$ égal à $\varphi(N)$, où $\varphi$ est la fonction indicatrice d'Euler. Zudilin a montré que si $\textbf{e}$ est une suite constituée des éléments du multi-ensemble $\bigcup_{i=1}^kA_{N_i}$ et si $\textbf{f}$ est une suite constituée des éléments du multi-ensemble $\bigcup_{i=1}^kB_{N_i}$, où les $N_i$ sont des entiers strictement positifs ayant le même ensemble de diviseurs premiers, alors l'application miroir associée à $(\textbf{e},\textbf{f}\,)$ a tous ses coefficients de Taylor entiers (\footnote{Le cas traité par Lian et Yau correspond au choix des paramètres $k=1$ et $N_1=p$ avec $p$ premier.}). Par exemple, on peut appliquer le théorème de Zudilin aux suites $\mathcal{Q}_{(\textbf{e},\textbf{f}\,)}(n)=\frac{(4n)!}{(2n)!(n!)^2}$ et $\mathcal{Q}_{(\textbf{e},\textbf{f}\,)}(n)=\frac{(6n)!}{(3n)!(2n)!n!}$, respectivement attachées aux choix des paramètres $k=1,N_1=4$ et $k=1,N_1=6$.

Enfin, Krattenthaler et Rivoal ont démontré la conjecture de Zudilin (voir \cite[p. 605]{Zudilin}).

\begin{theo2}[Krattenthaler-Rivoal, \cite{Tanguy}]\label{Theo Tanguy}
Soit $k\in\mathbb{N}$, $k\geq 1$, et $N_1,\dots,N_k$ des entiers strictement positifs. Soit $\textbf{e}$ la suite constituée des éléments du multi-ensemble $\bigcup_{i=1}^kA_{N_i}$ et $\textbf{f}$ la suite constituée des éléments du multi-ensemble $\bigcup_{i=1}^kB_{N_i}$. Alors l'application miroir associée à $(\textbf{e},\textbf{f}\,)$ a tous ses coefficients entiers.
\end{theo2}

Il s'avère que le cas traité par Krattenthaler et Rivoal correspond exactement aux quotients de factorielles dont l'application de Landau associée est croissante sur $[0,1[$ (\footnote{Voir la fin de la partie \ref{section delta facto} pour une explication détaillée.}). On peut donc reformuler ce théorème en portant les conditions sur la fonction $\Delta_{(\textbf{e},\textbf{f}\,)}$.

\begin{theobis}[Krattenthaler-Rivoal, \cite{Tanguy}]\label{Theo Tanguy Delta}
Soit $\textbf{e}$ et $\textbf{f}$ deux suites d'entiers strictement positifs disjointes vérifiant $|\textbf{e}|=|\textbf{f}\,|$. Si $\Delta_{(\textbf{e},\textbf{f}\,)}$ est croissante sur $[0,1[$, alors $q_{(\textbf{e},\textbf{f}\,)}\in z\mathbb{Z}[[z]].$
\end{theobis}

\begin{Remarque}
La croissance de $\Delta_{(\textbf{e},\textbf{f}\,)}$ sur $[0,1[$ implique sa positivité sur $[0,1]$ et donc, d'après le critère de Landau, la suite $\mathcal{Q}_{(\textbf{e},\textbf{f}\,)}$ est à termes entiers.
\end{Remarque}

Le but de cet article est de démontrer le théorème suivant, qui caractérise complètement les applications miroir d'origine hypergéométrique (sous forme factorielle) ayant tous leur coefficients de Taylor entiers. Il contient les résultats des auteurs précédents.

\begin{theo}\label{conj equiv}
Soit $\textbf{e}$ et $\textbf{f}$ deux suites d'entiers strictement positifs disjointes, vérifiant $|\textbf{e}|=|\textbf{f}\,|$ et telles que $\mathcal{Q}_{(\textbf{e},\textbf{f}\,)}$ soit une suite à termes entiers (ce qui équivaut à $\Delta_{(\textbf{e},\textbf{f}\,)}\geq 0$ sur $[0,1]$). On a alors la dichotomie suivante.
\begin{itemize}
\item[$(i)$] Si, pour tout $x\in[1/M_{(\textbf{e},\textbf{f}\,)},1[$, on a $\Delta_{(\textbf{e},\textbf{f}\,)}(x)\geq 1$, alors $q_{(\textbf{e},\textbf{f}\,)}(z)\in z\mathbb{Z}[[z]]$.
\item[$(ii)$] S'il existe un $x\in[1/M_{(\textbf{e},\textbf{f}\,)},1[$ tel que $\Delta_{(\textbf{e},\textbf{f}\,)}(x)=0$, alors il n'existe qu'un nombre fini de nombres premiers~$p$ tels que $q_{(\textbf{e},\textbf{f}\,)}(z)\in z\mathbb{Z}_p[[z]]$.
\end{itemize}
\end{theo}

Nous allons maintenant énoncer un critère pour l'intégralité des applications \textit{de type miroir} $q_{L,(\textbf{e},\textbf{f}\,)}$ définies, pour tout entier $L\geq 1$, par $q_{L,(\textbf{e},\textbf{f}\,)}:=\exp(G_{L,(\textbf{e},\textbf{f}\,)}(z)/F_{(\textbf{e},\textbf{f}\,)}(z))$, où $G_{L,(\textbf{e},\textbf{f}\,)}$ est la série formelle 
\begin{equation}\label{définition G L}
G_{L,(\textbf{e},\textbf{f}\,)}(z):=\sum_{n=1}^{\infty}\frac{(e_1n)!\cdots(e_{q_1}n)!}{(f_1n)!\cdots(f_{q_2}n)!}H_{Ln}z^n.
\end{equation}
On a $q_{L,(\textbf{e},\textbf{f}\,)}(z)\in 1+z\mathbb{Q}[[z]]$ et $z^{-1}q_{(\textbf{e},\textbf{f}\,)}(z)=(\prod_{i=1}^{q_1}q_{e_i,(\textbf{e},\textbf{f}\,)}^{e_i}(z))/(\prod_{j=1}^{q_2}q_{f_j,(\textbf{e},\textbf{f}\,)}^{f_j}(z))$, de sorte que si, pour tout $L\in\{1,\dots,M_{(\textbf{e},\textbf{f}\,)}\}$, on a $q_{L,(\textbf{e},\textbf{f}\,)}\in\mathbb{Z}[[z]]$, alors on a $q_{(\textbf{e},\textbf{f}\,)}\in z\mathbb{Z}[[z]]$. Ainsi, le point $(i)$ du théorème \ref{conj equiv L} implique le point $(i)$ du théorème \ref{conj equiv}.

\begin{theo}\label{conj equiv L}
Soit $\textbf{e}$ et $\textbf{f}$ deux suites d'entiers strictement positifs disjointes, vérifiant $|\textbf{e}|=|\textbf{f}\,|$ et telles que $\mathcal{Q}_{(\textbf{e},\textbf{f}\,)}$ soit une suite à termes entiers (ce qui équivaut à $\Delta_{(\textbf{e},\textbf{f}\,)}\geq 0$ sur $[0,1]$). On a alors la dichotomie suivante.
\begin{itemize}
\item[$(i)$] Si, pour tout $x\in[1/M_{(\textbf{e},\textbf{f}\,)},1[$, on a $\Delta_{(\textbf{e},\textbf{f}\,)}(x)\geq 1$, alors, pour tout $L\in\{1,\dots,M_{(\textbf{e},\textbf{f}\,)}\}$, on a $q_{L,(\textbf{e},\textbf{f}\,)}(z)\in \mathbb{Z}[[z]]$.
\item[$(ii)$] S'il existe un $x\in[1/M_{(\textbf{e},\textbf{f}\,)},1[$ tel que $\Delta_{(\textbf{e},\textbf{f}\,)}(x)=0$, alors, pour tout $L\in\{1,\dots,M_{(\textbf{e},\textbf{f}\,)}\}$, il n'existe qu'un nombre fini de nombres premiers~$p$ tels que $q_{L,(\textbf{e},\textbf{f}\,)}(z)\in \mathbb{Z}_p[[z]]$.
\end{itemize}
\end{theo}

\begin{Remarques}
\begin{itemize}
\item On remarquera l'analogie entre le critère de Landau et les théorèmes \ref{conj equiv} et \ref{conj equiv L}.
\item Montrer le cas $(i)$ du théorème \ref{conj equiv} revient à montrer que la conclusion du théorème bis perdure lorsque $\Delta_{(\textbf{e},\textbf{f}\,)}$ n'est pas forcément croissante sur $[0,1[$ mais lorsque l'on a la condition plus faible $\Delta_{(\textbf{e},\textbf{f}\,)}\geq 1$ sur $[1/M_{(\textbf{e},\textbf{f}\,)},1[$. Un exemple d'une telle fonction $\Delta_{(\textbf{e},\textbf{f}\,)}$ est donné par les suites $\textbf{e}=(3,3)$ et $\textbf{f}=(2,1,1,1,1)$.
\item Le théorème \ref{conj equiv L} est une généralisation du théorème 1 de \cite{Tanguy}.
\item Si la suite $\mathcal{Q}_{(\textbf{e},\textbf{f}\,)}$ est à termes entiers alors, d'après le critère de Landau, $\Delta_{(\textbf{e},\textbf{f}\,)}$ est positive sur $[0,1]$. Ainsi, s'il existe un $x\in [1/M_{(\textbf{e},\textbf{f}\,)},1[$ tel que $\Delta_{(\textbf{e},\textbf{f}\,)}(x)<1$, alors on a $\Delta_{(\textbf{e},\textbf{f}\,)}(x)=0$. Ce qui justifie les dichotomies annoncées dans les théorèmes \ref{conj equiv} et \ref{conj equiv L}.
\item Nous verrons en fin de partie \ref{section delta facto} que $M_{(\textbf{e},\textbf{f}\,)}$ est en fait le terme maximal de la suite $\textbf{e}$ et on a $M_{(\textbf{e},\textbf{f}\,)}\geq 2$.
\end{itemize}
\end{Remarques}

\subsection{Trame de la démonstration des théorèmes \ref{conj equiv} et \ref{conj equiv L}}

Dans la partie~\ref{sectioncongruform}, on énonce et démontre notre théorème \ref{theo généralisé}, qui généralise un critère de congruences formelles de Dwork. Ce dernier était crucial pour les résultats de Lian-Yau, Zudilin et Krattenthaler-Rivoal. Le théorème~\ref{theo généralisé} est au coeur de la preuve des théorèmes \ref{conj equiv} et \ref{conj equiv L} puisque l'on montre dans la partie \ref{newreforprob} que le critère de Dwork ne suffit pas pour démontrer le point $(i)$ du théorème \ref{conj equiv L}.

Dans la partie \ref{section equiv Zp}, on ramène les théorèmes \ref{conj equiv} et \ref{conj equiv L} à la preuve d'un énoncé $p$-adique. 

La partie \ref{section sur la condition suffisante} est consacrée à la preuve de l'assertion $(i)$ du théorème \ref{conj equiv L}, ce qui est de loin la partie la plus longue et la plus technique de l'article. On doit en particulier démontrer un certain nombre d'estimations $p$-adiques fines afin d'être en position d'appliquer le théorème \ref{theo généralisé}.

Dans la partie \ref{section(ii)}, on démontre les assertions $(ii)$ des théorèmes \ref{conj equiv} et \ref{conj equiv L} qui découlent assez vite de la reformulation de ces théorèmes établie dans la partie \ref{section equiv Zp}.

On finit, dans la partie~\ref{préliminaires}, par caractériser les fonctions hypergéométriques généralisées dont les coefficients peuvent se mettre sous forme de factorielles et on décrit les sauts des applications de Landau sur $[0,1]$.

\section{Congruences formelles}\label{sectioncongruform}

La preuve de l'assertion $(i)$ du théorème \ref{conj equiv L} est basée essentiellement sur la généralisation suivante du théorème de Dwork \cite[Theorem 1, p. 296]{Dwork 1}. Soit $p$ un nombre premier. On note $\Omega$ le complété de la clôture algébrique de $\mathbb{Q}_p$ et $\mathcal{O}$ l'anneau des entiers de $\Omega$.

\begin{theo}\label{theo généralisé}
Fixons un premier $p$. Soit $(\textbf{A}_r)_{r\geq 0}$ une suite d'applications de $\mathbb{N}$ dans $\Omega\setminus\{0\}$ et $(\textbf{g}_r)_{r\geq 0}$ une suite d'applications de $\mathbb{N}$ dans $\mathcal{O}\setminus\{0\}$ telles que, pour tout $r\geq 0$, on ait
\begin{itemize}
\item[$(i)$]{$|\textbf{A}_r(0)|_p=1$;}
\item[$(ii)$]{pour tout $m\in\mathbb{N}$, on a $\textbf{A}_r(m)\in \textbf{g}_r(m)\mathcal{O}$;}
\end{itemize}
et telles qu'il existe un $k_0\in\mathbb{N}$ tel que
\begin{itemize}
\item[$(iii)$]{pour tout $m\in\mathbb{N}$ et tout $r\geq 0$,

si $v_p(m)\geq k_0$ alors, pour tout $v,u$ et $s$ dans $\mathbb{N}$ tels que $v<p$, $u<p^s$, on a
$$
\frac{\textbf{A}_r(v+up+mp^{s+1})}{\textbf{A}_r(v+up)}-\frac{\textbf{A}_{r+1}(u+mp^{s})}{\textbf{A}_{r+1}(u)}\in p^{s+k_0+1}\frac{\textbf{g}_{r+s+1}(m)}{\textbf{g}_r(v+up)}\mathcal{O};
$$

si $v_p(m)\leq k_0-1$, alors $\frac{\textbf{A}_r(mp)}{\textbf{A}_r(0)}-\frac{\textbf{A}_{r+1}(m)}{\textbf{A}_{r+1}(0)}\in p^{v_p(m)+1}\textbf{g}_{r+1}(m)\mathcal{O};$
}
\item[$(iv)$] {pour tout $k\in\{1,\dots,k_0\}$, tout $v\in\{1,\dots,p-1\}$, tout $m\in\mathbb{N}$ et tout $r\geq 0$, on a $\textbf{g}_r(v+mp^{k})\in p^{k}\textbf{g}_{r}(mp^k)\mathcal{O}$ et $\textbf{g}_r(mp^{k})\in \textbf{g}_{r+k}(m)\mathcal{O}$.
}
\end{itemize}
Alors, pour tout $a\in\{0,\dots,p-1\}$, tout $m$ et $s$ dans $\mathbb{N}$, tout $r\geq 0$ et tout $K\in\mathbb{Z}$, on a
\begin{multline*}
\textbf{S}_r(a,K,s,p,m):=\\
\sum_{j=mp^s}^{(m+1)p^s-1}\left(\textbf{A}_r(a+p(K-j))\textbf{A}_{r+1}(j)-\textbf{A}_{r+1}(K-j)\textbf{A}_r(a+jp)\right)\in p^{s+1}\textbf{g}_{s+r+1}(m)\mathcal{O},
\end{multline*}
où l'on pose, pour tout $r\geq 0$, $\textbf{A}_r(\ell)=0$ si $\ell<0$. 
\end{theo}

\begin{Remarques}
\begin{itemize}
\item Le théorème de Dwork correspond au cas où $k_0=0$, auquel cas la condition $(iii)$ est identique à la condition $(iii)$ du critère de Dwork et la condition $(iv)$ est vide. La preuve pour $k_0\geq 1$ est très différente de celle donnée par Dwork pour $k_0=0$.
\item On peut trouver une généralisation différente du critère de Dwork dans \cite{Kaori}. Cette généralisation ne semble pas adaptée à notre situation.
\item Les fonctions $\textbf{A}_r$ et $\textbf{g}_r$ peuvent dépendre de $p$. On utilisera cette souplesse.
\end{itemize}
\end{Remarques}

Dans la suite, si $k_0$ est un entier naturel, on appellera \textit{$k_0$-couple de Dwork} tout couple de suites $\left((\textbf{A}_r)_{r\geq 0},(\textbf{g}_r)_{r\geq 0}\right)$ où les $\textbf{A}_r$ sont des applications de $\mathbb{N}$ dans $\Omega\setminus\{0\}$, les $\textbf{g}_r$ sont des applications de $\mathbb{N}$ dans $\mathcal{O}\setminus\{0\}$, telles que, pour tout $r\geq 0$, $\textbf{A}_r$ et $\textbf{g}_r$ vérifient les conditions $(i)$,$(ii)$,$(iii)$ et $(iv)$ du théorème \ref{theo généralisé}.

Le but de la fin de cette partie est de démontrer le théorème \ref{theo généralisé}. Pour cela, on va avoir besoin d'un certain nombre de résultats intermédiaires.

\subsection{Lemmes préparatoires}

On énonce et démontre trois lemmes.

\begin{lemme}\label{lemme bizzare 1.1}
Soit $(\textbf{g}_r)_{r\geq 0}$ une suite d'applications de $\mathbb{N}$ dans $\mathcal{O}\setminus\{0\}$ telle qu'il existe un entier naturel $k_0\geq 1$ tel que 
\begin{itemize}
\item[$(IV)$] {pour tout $k\in\{1,\dots,k_0\}$, tout $v\in\{1,\dots,p-1\}$, tout $m\in\mathbb{N}$ et tout $r\geq 0$, on a $\textbf{g}_r(v+mp^{k})\in p^{k}\textbf{g}_{r}(mp^k)\mathcal{O}$ et $\textbf{g}_r(mp^{k})\in \textbf{g}_{r+k}(m)\mathcal{O}$.}
\end{itemize}
Alors, pour tout $w\in\mathbb{N}$, $w\geq 1$, et tout $r\geq 0$, on a $\textbf{g}_r(w)\in p^{k_0}\mathcal{O}$.
\end{lemme}

\begin{Remarque}
\begin{itemize}
\item La condition $(IV)$ du lemme \ref{lemme bizzare 1.1} est la condition $(iv)$ du théorème \ref{theo généralisé}. 
\item La condition $w\neq 0$ est essentielle car, pour $w=0$, l'hypothèse $(IV)$ ne donne pas mieux que $\textbf{g}_r(0)\in\mathcal{O}\setminus\{0\}$.
\end{itemize}
\end{Remarque}

\begin{proof}
On écrit $w:=\sum_{\ell=0}^{N}w_{\ell}p^{\ell}$ le développement $p$-adique de $w$, où $w_N\neq 0$. On va raisonner par récurrence sur $N$.
\medskip
\begin{itemize}
\item Supposons $N=0$. 
\end{itemize}
Dans ce cas, on a $w=w_0\in\{1,\dots,p-1\}$. En appliquant $(IV)$ avec $v=w,k=k_0$ et $m=0$, on obtient bien $\textbf{g}_r(w)\in p^{k_0}\textbf{g}_{r}(0)\mathcal{O}\subset p^{k_0}\mathcal{O}$.
\medskip
\begin{itemize}
\item Supposons $N\geq 1$. 
\end{itemize}
Soit $\eta$ le plus petit entier naturel tel que $w_{\eta}\neq 0$. Si $\eta=0$, alors $w_0\geq 1$ donc, d'après $(IV)$ appliquée en $v=w_0,k=1$ et $m=\sum_{\ell=0}^{N-1}w_{\ell+1}p^{\ell}$, on obtient $\textbf{g}_r(w)\in p\textbf{g}_{r}(mp)\mathcal{O}\subset p\textbf{g}_{r+1}(m)\mathcal{O}$. On a $m\geq w_Np^{N-1}>0$. Ainsi, par hypothèse de récurrence, on a $\textbf{g}_{r+1}(m)\in p^{k_0}\mathcal{O}$, d'où le résultat.

Si $\eta\geq 1$, alors $w=p\sum_{\ell=0}^{N-1}w_{\ell+1}p^{\ell}$. D'après $(IV)$ appliquée en $k=1$ et $m=\sum_{\ell=0}^{N-1}w_{\ell+1}p^{\ell}>0$, on obtient $\textbf{g}_r(w)=\textbf{g}_r(mp)\in \textbf{g}_{r+1}(m)\mathcal{O}$ et on conclut par l'hypothèse de récurrence. Ceci achève la preuve du lemme.
\end{proof}

\begin{lemme}\label{lemme bizzare 1.2}
Soit $\left((\textbf{A}_r)_{r\geq 0},(\textbf{g}_r)_{r\geq 0}\right)$ un $k_0$-couple de Dwork avec $k_0\geq 1$. Alors, pour tout $k\in\{0,\dots,k_0\}$, tout $a\in\{0,\dots,p-1\}$, tout $m\in\mathbb{N}$, tout $K\in\mathbb{Z}$ et tout $r\geq 0$, on a $\textbf{S}_r(a,K,0,p,mp^{k})\in p^{k+1}\textbf{g}_{r+1}(mp^k)\mathcal{O}$.
\end{lemme}

\begin{proof}
Soit $k\in\{0,\dots,k_0\}$, $a\in\{0,\dots,p-1\}$, $m\in\mathbb{N}$, $K\in\mathbb{Z}$ et $r\geq 0$. On a 
$$
\textbf{S}_r(a,K,0,p,mp^k)=\textbf{A}_r(a+p(K-mp^k))\textbf{A}_{r+1}(mp^k)-\textbf{A}_{r+1}(K-mp^k)\textbf{A}_r(a+mp^{k+1}).
$$
Si $K-mp^k<0$ alors $\textbf{S}_r(a,K,0,p,mp^k)=0$ et le lemme \ref{lemme bizzare 1.2} est trivialement vrai. On suppose donc $K-mp^k\geq 0$ dans la suite de la démonstration. On va distinguer plusieurs cas.
\medskip
\begin{itemize}
\item Supposons $a\neq 0$ et $k\leq k_0-1$.
\end {itemize}
\medskip
D'après $(ii)$, on a $\textbf{A}_r(a+p(K-mp^{k}))\in\textbf{g}_r(a+p(K-mp^{k}))\mathcal{O}$ et, d'après le lemme \ref{lemme bizzare 1.1}, on a $\textbf{g}_r(a+p(K-mp^{k}))\in p^{k_0}\mathcal{O}$ car $a+p(K-mp^k)\geq 1$. Comme $k+1\leq k_0$, on obtient donc
\begin{equation}\label{Etape 1-1}
\textbf{A}_r(a+p(K-mp^k))\in p^{k+1}\mathcal{O}.
\end{equation}
Toujours d'après $(ii)$, on a
\begin{equation}\label{Etape 1-2}
\textbf{A}_{r+1}(mp^k)\in\textbf{g}_{r+1}(mp^k)\mathcal{O},
\end{equation}
\begin{equation}\label{Etape 1-3}
\textbf{A}_{r+1}(K-mp^k)\in \textbf{g}_{r+1}(K-mp^k)\mathcal{O}\subset \mathcal{O}
\end{equation}
et $\textbf{A}_r(a+mp^{k+1})\in\textbf{g}_r(a+mp^{k+1})\mathcal{O}$. Comme $1\leq k+1\leq k_0$, on peut appliquer $(iv)$ en $k+1$ et on obtient $\textbf{g}_r(a+mp^{k+1})\in p^{k+1}\textbf{g}_r(mp^{k+1})\mathcal{O}\subset p^{k+1}\textbf{g}_{r+1}(mp^k)\mathcal{O}$ et donc
\begin{equation}\label{Etape 1-4}
\textbf{A}_r(a+mp^{k+1})\in p^{k+1}\textbf{g}_{r+1}(mp^k)\mathcal{O}.
\end{equation}
Ainsi, d'après \eqref{Etape 1-1}, \eqref{Etape 1-2}, \eqref{Etape 1-3} et \eqref{Etape 1-4}, on obtient bien $\textbf{S}_r(a,K,0,p,mp^k)\in p^{k+1}\textbf{g}_{r+1}(mp^k)\mathcal{O}$.
\medskip
\begin{itemize}
\item Supposons $a\neq 0,k=k_0$ et $K-mp^{k_0}>0$. 
\end{itemize}
\medskip
D'après $(ii)$, on a $\textbf{A}_r(a+p(K-mp^{k_0}))\in\textbf{g}_r(a+p(K-mp^{k_0}))\mathcal{O}$ et, d'après $(iv)$, on a \\$\textbf{g}_r(a+p(K-mp^{k_0}))\in p\textbf{g}_{r}(p(K-mp^{k_0}))\mathcal{O}$. Comme $K-mp^{k_0}>0$, le lemme \ref{lemme bizzare 1.1} implique donc que 
\begin{equation}\label{Etape 2-1}
\textbf{A}_r(a+p(K-mp^{k_0}))\in p^{k_0+1}\mathcal{O}.
\end{equation}
D'après $(ii)$, on a
\begin{equation}\label{Etape 2-2}
\textbf{A}_{r+1}(mp^{k_0})\in\textbf{g}_{r+1}(mp^{k_0})\mathcal{O}
\end{equation}
et
\begin{equation}\label{Etape 2-3}
\textbf{A}_{r+1}(K-mp^{k_0})\in \textbf{g}_{r+1}(K-mp^{k_0})\mathcal{O}\subset p^{k_0}\mathcal{O},
\end{equation}
où l'inclusion dans \eqref{Etape 2-3} est obtenue de nouveau \textit{via} le lemme \ref{lemme bizzare 1.1}.

Enfin, d'après $(ii)$, on a $\textbf{A}_r(a+mp^{k_0+1})\in\textbf{g}_r(a+mp^{k_0+1})\mathcal{O}.$ D'après $(iv)$, on obtient 
$$
\textbf{g}_r(a+mp^{k_0+1})=\textbf{g}_r(a+(mp^{k_0})p)\in p\textbf{g}_r(mp^{k_0+1})\mathcal{O}\subset p\textbf{g}_{r+1}(mp^{k_0})\mathcal{O}
$$
et donc
\begin{equation}\label{Etape 2-4}
\textbf{A}_r(a+mp^{k_0+1})\in p\textbf{g}_{r+1}(mp^{k_0})\mathcal{O}.
\end{equation}
Ainsi, d'après \eqref{Etape 2-1}, \eqref{Etape 2-2}, \eqref{Etape 2-3} et \eqref{Etape 2-4}, on obtient bien $\textbf{S}_r(a,K,0,p,mp^{k_0})\in p^{k_0+1}\textbf{g}_{r+1}(mp^{k_0})\mathcal{O}$.
\medskip
\begin{itemize}
\item Supposons $a\neq 0,k=k_0$ et $K-mp^{k_0}=0$. On a alors
\end{itemize}
\begin{align*}
\textbf{S}_r(a,K,0,p,mp^{k_0})
&=\textbf{A}_r(a)\textbf{A}_{r+1}(mp^{k_0})-\textbf{A}_{r+1}(0)\textbf{A}_r(a+mp^{k_0+1})\\
&=-\textbf{A}_r(a)\textbf{A}_{r+1}(0)\left(\frac{\textbf{A}_r(a+mp^{k_0+1})}{\textbf{A}_{r}(a)}-\frac{\textbf{A}_{r+1}(mp^{k_0})}{\textbf{A}_{r+1}(0)}\right),
\end{align*}
avec, d'après $(i)$ et $(ii)$, $\textbf{A}_r(a)\textbf{A}_{r+1}(0)\in\textbf{g}_r(a)\mathcal{O}$. Comme $v_p(mp^{k_0})\geq k_0$, on obtient, d'après $(iii)$ appliquée en $s=0$, $v=a$, $u=0$ et $mp^{k_0}$ à la place de $m$, que 
$$
\frac{\textbf{A}_r(a+mp^{k_0+1})}{\textbf{A}_{r}(a)}-\frac{\textbf{A}_{r+1}(mp^{k_0})}{\textbf{A}_{r+1}(0)}\in p^{k_0+1}\frac{\textbf{g}_{r+1}(mp^{k_0})}{\textbf{g}_r(a)}\mathcal{O}.
$$ 
La congruence voulue en découle.
\medskip
\begin{itemize}
\item Il reste un seul cas: $a=0$. 
\end{itemize}
\medskip
Dans ce cas, on a
\begin{multline*}
\textbf{S}_r(0,K,0,p,mp^k)=\textbf{A}_r(p(K-mp^k))\textbf{A}_{r+1}(mp^k)-\textbf{A}_{r+1}(K-mp^k)\textbf{A}_r(mp^{k+1})\notag\\
=\textbf{A}_r(0)\textbf{A}_{r+1}(0)\left(\frac{\textbf{A}_r(p(K-mp^k))}{\textbf{A}_r(0)}\frac{\textbf{A}_{r+1}(mp^k)}{\textbf{A}_{r+1}(0)}-\frac{\textbf{A}_{r+1}(K-mp^k)}{\textbf{A}_{r+1}(0)}\frac{\textbf{A}_r(mp^{k+1})}{\textbf{A}_r(0)}\right).\label{expli quoiquoi}
\end{multline*}
On écrit le terme de droite sous la forme
\begin{multline*}
\frac{\textbf{A}_r(p(K-mp^k))}{\textbf{A}_r(0)}\frac{\textbf{A}_{r+1}(mp^k)}{\textbf{A}_{r+1}(0)}-\frac{\textbf{A}_{r+1}(K-mp^k)}{\textbf{A}_{r+1}(0)}\frac{\textbf{A}_r(mp^{k+1})}{\textbf{A}_r(0)}=\\
\frac{\textbf{A}_{r+1}(mp^k)}{\textbf{A}_{r+1}(0)}\left(\frac{\textbf{A}_r(p(K-mp^k))}{\textbf{A}_r(0)}-\frac{\textbf{A}_{r+1}(K-mp^k)}{\textbf{A}_{r+1}(0)}\right)\\-\frac{\textbf{A}_{r+1}(K-mp^k)}{\textbf{A}_{r+1}(0)}\left(\frac{\textbf{A}_r(mp^{k+1})}{\textbf{A}_r(0)}-\frac{\textbf{A}_{r+1}(mp^{k})}{\textbf{A}_{r+1}(0)}\right).
\end{multline*}
Remarquons maintenant que l'on a 
\begin{equation}\label{remarquons}
\frac{\textbf{A}_r(mp^{k+1})}{\textbf{A}_r(0)}-\frac{\textbf{A}_{r+1}(mp^{k})}{\textbf{A}_{r+1}(0)}\in p^{k+1}\textbf{g}_{r+1}(mp^k)\mathcal{O}.
\end{equation}
En effet, si $v_p(mp^k)\geq k_0$ alors, en appliquant $(iii)$ avec $v=u=s=0$ et $mp^k$ à la place de $m$, on obtient
$$
\frac{\textbf{A}_r(mp^{k+1})}{\textbf{A}_r(0)}-\frac{\textbf{A}_{r+1}(mp^{k})}{\textbf{A}_{r+1}(0)}\in p^{k_0+1}\frac{\textbf{g}_{r+1}(mp^k)}{\textbf{g}_r(0)}\mathcal{O}.
$$
De plus, d'après $(i)$ et $(ii)$, $\textbf{g}_r(0)$ est inversible dans $\mathcal{O}$ et, par hypothèse, on a $k\leq k_0$, donc on obtient bien \eqref{remarquons} dans ce cas. Si en revanche $v_p(mp^k)\leq k_0-1$ alors, d'après $(iii)$ avec $mp^k$ à la place de $m$, on obtient
$$
\frac{\textbf{A}_r(mp^{k+1})}{\textbf{A}_r(0)}-\frac{\textbf{A}_{r+1}(mp^{k})}{\textbf{A}_{r+1}(0)}\in p^{v_p(mp^k)+1}\textbf{g}_{r+1}(mp^k)\mathcal{O}\subset  p^{k+1}\textbf{g}_{r+1}(mp^k)\mathcal{O},
$$
ce qui achève la vérification de \eqref{remarquons}.

Ainsi, si $K-mp^k=0$, alors on a bien $\textbf{S}_r(0,K,0,p,mp^k)\in p^{k+1}\textbf{g}_{r+1}(mp^k)\mathcal{O}$. Si $K-mp^k>0$, alors on a
\begin{equation}\label{remarquons 203}
\frac{\textbf{A}_r(p(K-mp^k))}{\textbf{A}_r(0)}-\frac{\textbf{A}_{r+1}(K-mp^k)}{\textbf{A}_{r+1}(0)}\in p\textbf{g}_{r+1}(K-mp^k)\mathcal{O}.
\end{equation}
En effet, si $v_p(K-mp^k)\geq k_0$ alors, en appliquant $(iii)$ avec $v=u=s=0$ et $K-mp^k$ à la place de $m$, on obtient
$$
\frac{\textbf{A}_r(p(K-mp^k))}{\textbf{A}_r(0)}-\frac{\textbf{A}_{r+1}(K-mp^{k})}{\textbf{A}_{r+1}(0)}\in p^{k_0+1}\frac{\textbf{g}_{r+1}(K-mp^k)}{\textbf{g}_r(0)}\mathcal{O}.
$$
De plus, d'après $(i)$ et $(ii)$, $\textbf{g}_r(0)$ est inversible dans $\mathcal{O}$ et, par hypothèse, on a $k_0\geq 1$, donc on obtient bien \eqref{remarquons 203} dans ce cas. Si en revanche $v_p(K-mp^k)\leq k_0-1$ alors, d'après $(iii)$ avec $K-mp^k$ à la place de $m$, on obtient
$$
\frac{\textbf{A}_r(p(K-mp^k))}{\textbf{A}_r(0)}-\frac{\textbf{A}_{r+1}(K-mp^{k})}{\textbf{A}_{r+1}(0)}\in p^{v_p(K-mp^k)+1}\textbf{g}_{r+1}(K-mp^k)\mathcal{O}\subset  p\textbf{g}_{r+1}(K-mp^k)\mathcal{O},
$$
ce qui achève la vérification de \eqref{remarquons 203}.

D'après le lemme \ref{lemme bizzare 1.1}, on a $\textbf{g}_{r+1}(K-mp^k)\in p^{k_0}\mathcal{O}$. D'après $(i)$ et $(ii)$, on a de plus $\frac{\textbf{A}_{r+1}(mp^k)}{\textbf{A}_{r+1}(0)}\in\textbf{g}_{r+1}(mp^k)\mathcal{O}$, donc
$$
\textbf{A}_r(0)\textbf{A}_{r+1}(0)\frac{\textbf{A}_{r+1}(mp^k)}{\textbf{A}_{r+1}(0)}\left(\frac{\textbf{A}_r(p(K-mp^k))}{\textbf{A}_r(0)}-\frac{\textbf{A}_{r+1}(K-mp^k)}{\textbf{A}_{r+1}(0)}\right)\in p^{k_0+1}\textbf{g}_{r+1}(mp^k)\mathcal{O}
$$
et comme également $\frac{\textbf{A}_{r+1}(K-mp^k)}{\textbf{A}_{r+1}(0)}\in\textbf{g}_{r+1}(K-mp^k)\mathcal{O}\subset\mathcal{O}$, on a
$$
\textbf{A}_r(0)\textbf{A}_{r+1}(0)\frac{\textbf{A}_{r+1}(K-mp^k)}{\textbf{A}_{r+1}(0)}\cdot\left(\frac{\textbf{A}_r(mp^{k+1})}{\textbf{A}_r(0)}-\frac{\textbf{A}_{r+1}(mp^k)}{\textbf{A}_{r+1}(0)}\right)\in p^{k+1}\textbf{g}_{r+1}(mp^k)\mathcal{O}.
$$
On a donc bien $\textbf{S}_r(0,K,0,p,mp^k)\in p^{k+1}\textbf{g}_{r+1}(mp^k)\mathcal{O}$,
ce qui achève la preuve du lemme.
\end{proof}

Enfin, on aura besoin du lemme suivant.

\begin{lemme}\label{lemme bizzare 2}
Soit $\left((\textbf{A}_r)_{r\geq 0},(\textbf{g}_r)_{r\geq 0}\right)$ un $k_0$-couple de Dwork avec $k_0\geq 1$. Soit $a\in\{0,\dots,p-1\}$, $K\in\mathbb{Z}$ et $s\in\mathbb{N}$, $s\geq 1$. Si, pour tout $s_0<s$ , pour tout $m\in\mathbb{N}$ et tout $r\geq 0$, on a
\begin{equation}\label{condi one}
\textbf{S}_r(a,K,s_0,p,mp^{k_0})\in p^{s_0+k_0+1}\textbf{g}_{s_0+r+1}(mp^{k_0})\mathcal{O},
\end{equation}
et 
\begin{equation}\label{condi two}
\textbf{S}_r(a,K,s_0,p,m)\in p^{s_0+1}\textbf{g}_{s_0+r+1}(m)\mathcal{O}.
\end{equation}
Alors, pour tout $m\in\mathbb{N}$ et tout $r\geq 0$, on a $\textbf{S}_r(a,K,s,p,m)\in p^{s+1}\textbf{g}_{s+r+1}(m)\mathcal{O}$.
\end{lemme}

\begin{proof}
On va d'abord montrer que, pour tout $m\in\mathbb{N}$, tout $r\geq 1$ et tout $k\leq\min(s,k_0)$, on a $\textbf{S}_r(a,K,s,p,m)\equiv \textbf{S}_{r}(a,K,s-k,p,mp^{k})\mod p^{s+1}\textbf{g}_{s+r+1}(m)\mathcal{O}$.
Pour cela, on va raisonner par récurrence sur $k$.

Si $k=0$, il n'y a rien à montrer.

Supposons que $\min(s,k_0)\geq k\geq 1$. Par hypothèse de récurrence, on a
\begin{equation}\label{**2}
\textbf{S}_r(a,K,s,p,m)\equiv \textbf{S}_{r}(a,K,s-k+1,p,mp^{k-1})\mod p^{s+1}\textbf{g}_{s+r+1}(m)\mathcal{O}.
\end{equation}
Comme $\{mp^s,\dots,mp^s+p^{s-k+1}-1\}=\bigcup_{v=0}^{p-1}\{mp^s+vp^{s-k},\dots,mp^s+vp^{s-k}+p^{s-k}-1\}$, on a
\begin{equation}\label{*2}
\textbf{S}_{r}(a,K,s-k+1,p,mp^{k-1})=\sum_{v=0}^{p-1}\textbf{S}_{r}(a,K,s-k,p,v+mp^{k}).
\end{equation}
Comme $s\geq k\geq 1$, on a $0\leq s-k<s$ et on obtient, d'après \eqref{condi two}, que, pour tout $v\in\{0,\dots,p-1\}$, on a $\textbf{S}_r(a,K,s-k,p,v+mp^{k})\in p^{s-k+1}\textbf{g}_{s-k+r+1}(v+mp^{k})\mathcal{O}$. Comme également $k\in\{1,\dots,k_0\}$ on obtient d'après $(iv)$, que, pour tout $v\in\{1,\dots,p-1\}$, on a 
$$
\textbf{g}_{s-k+r+1}(v+mp^{k})\in p^k\textbf{g}_{s-k+r+1}(mp^k)\mathcal{O}\subset p^{k}\textbf{g}_{s+r+1}(m)\mathcal{O}.
$$
En utilisant ces informations dans \eqref{*2}, on obtient (il reste seulement le terme pour $v=0$):
$$
\textbf{S}_r(a,K,s-k+1,p,mp^{k-1})\equiv \textbf{S}_r(a,K,s-k,p,mp^{k})\mod p^{s+1}\textbf{g}_{s+r+1}(m)\mathcal{O},
$$
ce qui, joint à \eqref{**2}, montre que $\textbf{S}_r(a,K,s,p,m)\equiv \textbf{S}_r(a,K,s-k,p,mp^{k})\mod p^{s+1}\textbf{g}_{s+r+1}(m)\mathcal{O}$ et achève la récurrence sur $k$.
\medskip

On a donc $\textbf{S}_r(a,K,s,p,m)\equiv \textbf{S}_r(a,K,s-\min(s,k_0),p,mp^{\min(s,k_0)})\mod p^{s+1}\textbf{g}_{s+r+1}(m)\mathcal{O}$.
\medskip

Si $s<k_0$, alors $\textbf{S}_r(a,K,s,p,m)\equiv \textbf{S}_r(a,K,0,p,mp^s)\mod p^{s+1}\textbf{g}_{s+r+1}(m)\mathcal{O}$. En appliquant le lemme \ref{lemme bizzare 1.2} en $k=s$, puis $(iv)$ en $k=s$, on obtient $\textbf{S}_r(a,K,0,p,mp^s)\in p^{s+1}\textbf{g}_{r+1}(mp^s)\mathcal{O}\subset p^{s+1}\textbf{g}_{s+r+1}(m)\mathcal{O}$ et donc $\textbf{S}_r(a,K,s,p,m)\in p^{s+1}\textbf{g}_{s+r+1}(m)\mathcal{O}$, comme voulu.

Si maintenant $s\geq k_0$ alors $\textbf{S}_r(a,K,s,p,m)\equiv \textbf{S}_r(a,K,s-k_0,p,mp^{k_0})\mod p^{s+1}\textbf{g}_{s+r+1}(m)\mathcal{O}$. Or, comme $s\geq k_0\geq 1$, on a $0\leq s-k_0<s$ et on obtient (d'après \eqref{condi one} et $(iv)$ appliquée en $k=k_0$) que $\textbf{S}_r(a,K,s-k_0,p,mp^{k_0})\in p^{s-k_0+k_0+1}\textbf{g}_{s-k_0+r+1}(mp^{k_0})\mathcal{O}\subset p^{s+1}\textbf{g}_{s+r+1}(m)\mathcal{O}$, ce qui achève la preuve du lemme.
\end{proof}

\subsection{Démonstration du théorème \ref{theo généralisé}}

Comme dit au début de la partie \ref{sectioncongruform}, le cas $k_0=0$ correspond au critère de Dwork. Il nous suffit donc de démontrer le cas où il existe un $k_0\geq 1$ tel que les hypothèses $(iii)$ et $(iv)$ soient vérifiées. En particulier, on utilisera les lemmes \ref{lemme bizzare 1.1}, \ref{lemme bizzare 1.2} et \ref{lemme bizzare 2}. La trame de la démonstration s'inspire de celle du théorème de Dwork, mais elle diffère assez sensiblement dans les détails. 

On doit montrer que, pour tout $a\in\{0,\dots,p-1\}$, tout $m$ et $s$ dans $\mathbb{N}$, tout $r\geq 0$ et tout $K\in\mathbb{Z}$, on a
\begin{equation}\label{TheBut}
\textbf{S}_r(a,K,s,p,m)\in p^{s+1}\textbf{g}_{s+r+1}(m)\mathcal{O}.
\end{equation}

Pour tout $a\in\{0,\dots,p-1\}$, tout $j\in\mathbb{N}$, tout $r\geq0$ et tout $K\in\mathbb{Z}$, on note
$$
\textbf{U}_r(a,K,p,j):=\textbf{A}_r(a+p(K-j))\textbf{A}_{r+1}(j)-\textbf{A}_{r+1}(K-j)\textbf{A}_r(a+jp),
$$
de sorte que $\textbf{S}_r(a,K,s,p,m)=\sum_{j=mp^s}^{(m+1)p^s-1}\textbf{U}_r(a,K,p,j)$. 
\medskip
Pour tout $s\in\mathbb{N}$, $s\geq 1$, on note $\alpha_s$ l'assertion suivante: pour tout $a\in\{0,\dots,p-1\}$, tout $u\in\{0,\dots,s-1\}$, tout $m\in\mathbb{N}$, tout $r\geq0$ et tout $K\in\mathbb{Z}$, on a les congruences
$$
\textbf{S}_r(a,K,u,p,m)\in p^{u+1}\textbf{g}_{u+r+1}(m)\mathcal{O}\quad\textup{et}\quad\textbf{S}_r(a,K,u,p,mp^{k_0})\in p^{u+k_0+1}\textbf{g}_{u+r+1}(mp^{k_0})\mathcal{O}.
$$
\medskip
Pour tout $s\in\mathbb{N}$, $s\geq 1$, et tout $t\in\{0,\dots,s\}$, on note $\beta_{t,s}$ l'assertion suivante: pour tout $a\in\{0,\dots,p-1\}$, tout $m\in\mathbb{N}$, tout $r\geq0$ et tout $K\in\mathbb{Z}$, on a la congruence
\begin{multline*}
\textbf{S}_r(a,K+mp^{s+k_0},s,p,mp^{k_0})\equiv\\ \sum_{j=0}^{p^{s-t}-1}\frac{\textbf{A}_{t+r+1}(j+mp^{s-t+k_0})}{\textbf{A}_{t+r+1}(j)}\textbf{S}_r(a,K,t,p,j)\mod p^{s+k_0+1}\textbf{g}_{s+r+1}(mp^{k_0})\mathcal{O}.
\end{multline*}

Nous allons maintenant énoncer trois lemmes permettant de montrer \eqref{TheBut}.

\begin{lemme}\label{Assertion 1}
L'assertion $\alpha_1$ est vraie.
\end{lemme}

\begin{lemme}\label{Assertion 2}
Pour tout $s,r$ et $m$ dans $\mathbb{N}$, tout $a\in\{0,\dots,p-1\}$, tout $j\in\{0,\dots,p^s-1\}$ et tout $K\in\mathbb{Z}$, on a
$$
\textbf{U}_r(a,K+mp^{s+k_0},p,j+mp^{s+k_0})\equiv\frac{\textbf{A}_{r+1}(j+mp^{s+k_0})}{\textbf{A}_{r+1}(j)}\textbf{U}_r(a,K,p,j)\mod p^{s+k_0+1}\textbf{g}_{s+r+1}(mp^{k_0})\mathcal{O}.
$$
\end{lemme}

\begin{lemme}\label{Assertion 3}
Pour tout $s\in\mathbb{N}$, $s\geq 1$, et tout $t\in\{0,\dots,s-1\}$, les assertions $\alpha_s$ et $\beta_{t,s}$ impliquent l'assertion $\beta_{t+1,s}$.
\end{lemme}

Avant de prouver ces lemmes, nous allons montrer que leur validité implique bien \eqref{TheBut}. On va montrer que $\alpha_s$ est vraie pour tout $s\geq 1$ par récurrence sur $s$, ce qui donnera en particulier la conclusion du théorème \ref{theo généralisé}. D'après le lemme \ref{Assertion 1}, $\alpha_1$ est vraie. Supposons $\alpha_s$ vraie pour un $s\geq 1$ fixé. On remarque que $\beta_{0,s}$ est l'assertion
\begin{multline*}
\beta_{0,s}:\textbf{S}_r(a,K+mp^{s+k_0},s,p,mp^{k_0})\equiv\\
\sum_{j=0}^{p^s-1}\frac{\textbf{A}_{r+1}(j+mp^{s+k_0})}{\textbf{A}_{r+1}(j)}\textbf{S}_r(a,K,0,p,j)\mod p^{s+k_0+1}\textbf{g}_{s+r+1}(mp^{k_0})\mathcal{O}.
\end{multline*}
Comme $\textbf{S}_r(a,K,0,p,j)=\textbf{U}_r(a,K,p,j)$, on a
$$
\sum_{j=0}^{p^s-1}\frac{\textbf{A}_{r+1}(j+mp^{s+k_0})}{\textbf{A}_{r+1}(j)}\textbf{S}_r(a,K,0,p,j)=\sum_{j=0}^{p^s-1}\frac{\textbf{A}_{r+1}(j+mp^{s+k_0})}{\textbf{A}_{r+1}(j)}\textbf{U}_r(a,K,p,j)
$$
et, d'après le lemme \ref{Assertion 2}, on obtient
\begin{align*}
\sum_{j=0}^{p^s-1}&\frac{\textbf{A}_{r+1}(j+mp^{s+k_0})}{\textbf{A}_{r+1}(j)}\textbf{U}_r(a,K,p,j)\\
&\equiv\sum_{j=0}^{p^s-1}\textbf{U}_r(a,K+mp^{s+k_0},p,j+mp^{s+k_0})\mod p^{s+k_0+1}\textbf{g}_{s+r+1}(mp^{k_0})\mathcal{O}\\
&\equiv \textbf{S}_r(a,K+mp^{s+k_0},s,p,mp^{k_0})\mod p^{s+k_0+1}\textbf{g}_{s+r+1}(mp^{k_0})\mathcal{O}.
\end{align*}
Ainsi, l'assertion $\beta_{0,s}$ est vraie. On obtient alors, par le lemme \ref{Assertion 3}, la validité de $\beta_{1,s}$. Par itération du lemme \ref{Assertion 3}, on obtient finalement $\beta_{s,s}$, qui est l'assertion
\begin{equation}\label{assert beta s}
\textbf{S}_r(a,K+mp^{s+k_0},s,p,mp^{k_0})\equiv\frac{\textbf{A}_{s+r+1}(mp^{k_0})}{\textbf{A}_{s+r+1}(0)}\textbf{S}_r(a,K,s,p,0)\mod p^{s+k_0+1}\textbf{g}_{s+r+1}(mp^{k_0})\mathcal{O}.
\end{equation}
\medskip
Nous allons maintenant montrer que, pour tout $a\in\{0,\dots,p-1\}$, tout $s\in\mathbb{N}$, tout $r\geq0$ et tout $K\in\mathbb{Z}$, on a $\textbf{S}_r(a,K,s,p,0)\in p^{s+k_0+1}\mathcal{O}$.
Soit $T\in\mathbb{N}$ tel que $(T+1)p^s>K$. On a
\begin{align}
\sum_{m=0}^T \textbf{S}_r(a,K,s,p,m)
&=\sum_{m=0}^T\sum_{j=mp^s}^{(m+1)p^s-1}\left(\textbf{A}_r(a+p(K-j))\textbf{A}_{r+1}(j)-\textbf{A}_{r+1}(K-j)\textbf{A}_r(a+jp)\right)\notag\\
&=\sum_{j=0}^K\left(\textbf{A}_r(a+p(K-j))\textbf{A}_{r+1}(j)-\textbf{A}_{r+1}(K-j)\textbf{A}_r(a+jp)\right)\label{rappel conv}\\
&=0\label{expli somme 0},
\end{align}
où l'on a utilisé dans \eqref{rappel conv} le fait que $\textbf{A}_r(\ell)=0$ pour $\ell<0$, et \eqref{expli somme 0} a lieu car le terme de la somme \eqref{rappel conv} est changé en son opposé lorsque l'on change l'indice $j$ en $K-j$. 

Comme $\alpha_s$ est vraie, on obtient \textit{via} le lemme \ref{lemme bizzare 2} (avec $s_0=u$): $\textbf{S}_r(a,K,s,p,m)\in p^{s+1}\textbf{g}_{s+r+1}(m)\mathcal{O}$, ce qui prouve la première partie de l'assertion $\alpha_{s+1}$. Si $m>0$, alors d'après le lemme \ref{lemme bizzare 1.1}, on a $\textbf{g}_{s+r+1}(m)\in p^{k_0}\mathcal{O}$. Donc, on obtient, pour tout $m>0$, que $\textbf{S}_r(a,K,s,p,m)\in p^{s+k_0+1}\mathcal{O}$. D'où également, $\textbf{S}_r(a,K,s,p,0)=-\sum_{m=1}^T \textbf{S}_r(a,K,s,p,m)\in p^{s+k_0+1}\mathcal{O}$. De plus, d'après les conditions $(i)$ et $(ii)$, pour tout $m\in\mathbb{N}$ et tout $r\geq 0$, on a $\frac{\textbf{A}_{s+r+1}(mp^{k_0})}{\textbf{A}_{s+r+1}(0)}\in \textbf{g}_{s+r+1}(mp^{k_0})\mathcal{O}$.
Ainsi, d'après $\beta_{s,s}$ (\textit{i.e.} \eqref{assert beta s}), pour tout $m\in\mathbb{N}$ et tout $K\in\mathbb{Z}$, on obtient $\textbf{S}_r(a,K+mp^{s+k_0},s,p,mp^{k_0})\in p^{s+k_0+1}\textbf{g}_{s+r+1}(mp^{k_0})\mathcal{O}$ et donc $\alpha_{s+1}$ est vraie. Ceci achève la récurrence sur $s$. Ainsi, pour tout $s\in\mathbb{N}$, $s\geq 1$, $\alpha_s$ est vraie et, en particulier, $\textbf{S}_r(a,K,s,p,m)\in p^{s+1}\textbf{g}_{s+r+1}(m)\mathcal{O}$. Il ne reste plus qu'à démontrer les lemmes \ref{Assertion 1}, \ref{Assertion 2} et \ref{Assertion 3}.

\begin{proof}[Démonstration du lemme \ref{Assertion 1}]
En appliquant le lemme \ref{lemme bizzare 1.2}  avec $k=0$ et $k=k_0$, on obtient respectivement $\textbf{S}_r(a,K,0,p,m)\in p\textbf{g}_{r+1}(m)\mathcal{O}$ et $\textbf{S}_r(a,K,0,p,mp^{k_0})\in p^{k_0+1}\textbf{g}_{r+1}(mp^{k_0})\mathcal{O}$,
ce qui n'est rien d'autre que l'assertion $\alpha_1$ et termine la preuve du lemme \ref{Assertion 1}.
\end{proof}

\begin{proof}[Démonstration du lemme \ref{Assertion 2}]
On a
\begin{multline}\label{wesh wesh yo}
\textbf{U}_r(a,K+mp^{s+k_0},p,j+mp^{s+k_0})-\frac{\textbf{A}_{r+1}(j+mp^{s+k_0})}{\textbf{A}_{r+1}(j)}\textbf{U}_r(a,K,p,j)\\
=-\textbf{A}_{r+1}(K-j)\textbf{A}_r(a+jp)\left(\frac{\textbf{A}_r(a+jp+mp^{s+k_0+1})}{\textbf{A}_r(a+jp)}-\frac{\textbf{A}_{r+1}(j+mp^{s+k_0})}{\textbf{A}_{r+1}(j)}\right).
\end{multline}
Comme $a<p,j<p^s$ et $v_p(mp^{k_0})\geq k_0$, l'hypothèse $(iii)$ implique que le terme de droite de l'égalité \eqref{wesh wesh yo} est dans $\textbf{A}_{r+1}(K-j)\textbf{A}_r(a+jp)p^{s+k_0+1}\frac{\textbf{g}_{s+r+1}(mp^{k_0})}{\textbf{g}_r(a+jp)}\mathcal{O}$. De plus, d'après la condition $(ii)$, on a $\textbf{A}_r(a+jp)\in\textbf{g}_r(a+jp)\mathcal{O}$ et $\textbf{A}_{r+1}(K-j)\in\textbf{g}_{r+1}(K-j)\mathcal{O}\subset\mathcal{O}$. Ces estimations montrent que le membre de gauche de \eqref{wesh wesh yo} est dans $p^{s+k_0+1}\textbf{g}_{s+r+1}(mp^{k_0})\mathcal{O}$, comme voulu.
\end{proof}

\begin{proof}[Démonstration du lemme \ref{Assertion 3}]
Pour $t<s$, on écrit $\beta_{t,s}$ sous la forme
\begin{multline}
\textbf{S}_r(a,K+mp^{s+k_0},s,p,mp^{k_0})\equiv\\
\sum_{i=0}^{p-1}\sum_{\mu=0}^{p^{s-t-1}-1}\frac{\textbf{A}_{t+r+1}(i+\mu p+mp^{s-t+k_0})}{\textbf{A}_{t+r+1}(i+\mu p)}\textbf{S}_r(a,K,t,p,i+\mu p)\mod p^{s+k_0+1}\textbf{g}_{s+r+1}(mp^{k_0})\mathcal{O}.\label{beta t s}
\end{multline}
On veut montrer $\beta_{t+1,s}$, qui s'écrit
\begin{multline*}
\textbf{S}_r(a,K+mp^{s+k_0},s,p,mp^{k_0})\equiv\\
\sum_{\mu=0}^{p^{s-t-1}-1}\frac{\textbf{A}_{t+r+2}(\mu+mp^{s-t+k_0-1})}{\textbf{A}_{t+r+2}(\mu)}\textbf{S}_r(a,K,t+1,p,\mu)\mod p^{s+k_0+1}\textbf{g}_{s+r+1}(mp^{k_0})\mathcal{O}.
\end{multline*}

On remarque que $\textbf{S}_r(a,K,t+1,p,\mu)=\sum_{i=0}^{p-1}\textbf{S}_r(a,K,t,p,i+\mu p)$. Ainsi, en posant
$$
X:=\textbf{S}_r(a,K+mp^{s+k_0},s,p,mp^{k_0})-\sum_{\mu=0}^{p^{s-t-1}-1}\frac{\textbf{A}_{t+r+2}(\mu+mp^{s-t+k_0-1})}{\textbf{A}_{t+r+2}(\mu)}\sum_{i=0}^{p-1}\textbf{S}_r(a,K,t,p,i+\mu p),
$$
il ne reste plus qu'à montrer que $X\in p^{s+k_0+1}\textbf{g}_{s+r+1}(mp^{k_0})\mathcal{O}$. D'après $\beta_{t,s}$ sous la forme \eqref{beta t s}, on obtient que
\begin{multline*}
X\equiv \sum_{i=0}^{p-1}\sum_{\mu=0}^{p^{s-t-1}-1}\textbf{S}_r(a,K,t,p,i+\mu p)\\
\times\left(\frac{\textbf{A}_{t+r+1}(i+\mu p+mp^{s-t+k_0})}{\textbf{A}_{t+r+1}(i+\mu p)}-\frac{\textbf{A}_{t+r+2}(\mu+mp^{s-t-1+k_0})}{\textbf{A}_{t+r+2}(\mu)}\right)\mod p^{s+k_0+1}\textbf{g}_{s+r+1}(mp^{k_0})\mathcal{O}.
\end{multline*}
Or, d'après l'hypothèse $(iii)$ appliquée avec $s-t-1$ pour $s$ et $mp^{k_0}$ pour $m$, on a
$$
\frac{\textbf{A}_{t+r+1}(i+\mu p+mp^{s-t+k_0})}{\textbf{A}_{t+r+1}(i+\mu p)}-\frac{\textbf{A}_{t+r+2}(\mu+mp^{s-t-1+k_0})}{\textbf{A}_{t+r+2}(\mu)}\in p^{s-t+k_0}\frac{\textbf{g}_{s+r+1}(mp^{k_0})}{\textbf{g}_{t+r+1}(i+\mu p)}\mathcal{O}.
$$
De plus, comme $t<s$ et puisque $\alpha_s$ est vraie, on a $\textbf{S}_r(a,K,t,p,i+\mu p)\in p^{t+1}\textbf{g}_{t+r+1}(i+\mu p)\mathcal{O}$. Donc on a bien $X\in p^{s+k_0+1}\textbf{g}_{s+r+1}(mp^{k_0})\mathcal{O}$. Ceci achève la preuve du lemme \ref{Assertion 3} et donc celle du théorème \ref{theo généralisé}. 
\end{proof}

\section{Un énoncé $p$-adique équivalent au critère}\label{section equiv Zp}

On se place sous les hypothèses des théorèmes \ref{conj equiv} et \ref{conj equiv L}. On fixe $L\in\{1,\dots,M_{(\textbf{e},\textbf{f}\,)}\}$ dans cette partie. On rappelle que $q_{(\textbf{e},\textbf{f}\,)}\in z\mathbb{Z}[[z]]$, respectivement $q_{L,(\textbf{e},\textbf{f}\,)}\in\mathbb{Z}[[z]]$, si, et seulement si, pour tout nombre premier $p$, on a $q_{(\textbf{e},\textbf{f}\,)}\in z\mathbb{Z}_p[[z]]$, respectivement $q_{L,(\textbf{e},\textbf{f}\,)}\in\mathbb{Z}_p[[z]]$. 

Nous allons définir, pour tout nombre premier $p$, des éléments $\Phi_p(a+Kp)$ et $\Phi_{L,p}(a+Kp)$ de $\mathbb{Q}_p$, où $a\in\{0,\dots,p-1\}$ et $K\in\mathbb{N}$, et montrer que $q_{(\textbf{e},\textbf{f}\,)}\in z\mathbb{Z}_p[[z]]$, respectivement $q_{L,(\textbf{e},\textbf{f}\,)}\in\mathbb{Z}_p[[z]]$, si, et seulement si, pour tout nombre premier $p$, tout $a\in\{0,\dots,p-1\}$ et tout $K\in\mathbb{N}$, on a $\Phi_p(a+Kp)\in p\mathbb{Z}_p$, respectivement $\Phi_{L,p}(a+Kp)\in p\mathbb{Z}_p$. 

Pour alléger les notations, on notera $\Delta:=\Delta_{(\textbf{e},\textbf{f}\,)}$, $\mathcal{Q}:=\mathcal{Q}_{(\textbf{e},\textbf{f}\,)}$, $F:=F_{(\textbf{e},\textbf{f}\,)}$, $G:=G_{(\textbf{e},\textbf{f}\,)}$, $G_L:=G_{L,(\textbf{e},\textbf{f}\,)}$, $q:=q_{(\textbf{e},\textbf{f}\,)}$ et $q_L:=q_{L,(\textbf{e},\textbf{f}\,)}$, comme dans toute la suite de l'article. On fixe un nombre premier $p$ dans cette partie.

Avant de donner les preuves des théorèmes \ref{conj equiv} et \ref{conj equiv L}, on va les reformuler. Le résultat classique suivant est dû à Dieudonné et Dwork (voir \cite[Chap. VI, Sec. 2, Lemma 3]{Kobliz}; \cite[Chap. 14, Sec. 2]{Lang}).

\begin{lemme}\label{réduction modulo p}
Soit $F(z)$ une série formelle dans $1+z\mathbb{Q}[[z]]$. Alors $F(z)\in1+z\mathbb{Z}_p[[z]]$ si et seulement si $\frac{F(z^p)}{F(z)^p}\in 1+pz\mathbb{Z}_p[[z]]$.
\end{lemme}

On déduit de ce lemme le corollaire suivant (voir \cite[Lemma 5, p. 610]{Zudilin}), qui va nous permettre \og d'éliminer\fg\,l'exponentielle dans les expressions $q(z)=z\exp(G(z)/F(z))$ et $q_L(z)=\exp(G_L(z)/F(z))$.

\begin{cor}\label{cor reduction}
Soit $f(z)\in z\mathbb{Q}[[z]]$. On a $e^{f(z)}\in 1+z\mathbb{Z}_p[[z]]$ si et seulement si $f(z^p)-pf(z)\in pz\mathbb{Z}_p[[z]]$.
\end{cor}

D'après les identités \eqref{definition de G facto} et \eqref{définition G L} définissant respectivement $G$ et $G_L$, on a $G(0)=G_L(0)=0$ et donc $G(z)/F(z)$ et $G_L(z)/F(z)$ sont dans $z\mathbb{Q}[[z]]$. Ainsi, d'après le corollaire \ref{cor reduction}, on a $q(z)\in z\mathbb{Z}_p[[z]]$, respectivement $q_L(z)\in\mathbb{Z}_p[[z]]$, si, et seulement si $\frac{G}{F}(z^p)-p\frac{G}{F}(z)\in pz\mathbb{Z}_p[[z]]$, respectivement $\frac{G_L}{F}(z^p)-p\frac{G_L}{F}(z)\in pz\mathbb{Z}_p[[z]]$. 

Or, comme $\mathcal{Q}$ est une suite à termes entiers, on a $F(z)\in 1+z\mathbb{Z}[[z]]\subset 1+z\mathbb{Z}_p[[z]]$. Ainsi, $q(z)\in z\mathbb{Z}_p[[z]]$, respectivement $q_L(z)\in\mathbb{Z}_p[[z]]$, si, et seulement si on a $F(z)G(z^p)-pF(z^p)G(z)\in pz\mathbb{Z}_p[[z]]$, respectivement $F(z)G_L(z^p)-pF(z^p)G_L(z)\in pz\mathbb{Z}_p[[z]]$.

D'après l'identité \eqref{definition de G facto} définissant $G$, le coefficient de $z^{a+Kp}$ dans $F(z)G(z^p)-pF(z^p)G(z)$ est 
\begin{multline*}
\Phi_p(a+Kp):=\\
\sum_{j=0}^K\mathcal{Q}(K-j)\mathcal{Q}(a+jp)\left(\sum_{i=1}^{q_1}e_i(H_{e_i(K-j)}-pH_{e_i(a+jp)})-\sum_{i=1}^{q_2}f_i(H_{f_i(K-j)}-pH_{f_i(a+jp)})\right)
\end{multline*}
et, d'après l'identité \eqref{définition G L} définissant $G_L$, le coefficient de $z^{a+Kp}$ dans $F(z)G_L(z^p)-pF(z^p)G_L(z)$ est 
$$
\Phi_{L,p}(a+Kp):=\sum_{j=0}^K\mathcal{Q}(K-j)\mathcal{Q}(a+jp)(H_{L(K-j)}-pH_{L(a+jp)}).
$$

On a donc $q(z)\in z\mathbb{Z}_p[[z]]$, respectivement $q_L(z)\in\mathbb{Z}_p[[z]]$, si, et seulement si, pour tout $a\in\{0,\dots,p-1\}$ et tout $K\in\mathbb{N}$, on a $\Phi_p(a+Kp)\in p\mathbb{Z}_p$, respectivement $\Phi_{L,p}(a+Kp)\in p\mathbb{Z}_p$.

\section{Démonstration des cas $(i)$ des théorèmes \ref{conj equiv} et \ref{conj equiv L}}\label{section sur la condition suffisante}

On se place sous les hypothèses des thèorèmes \ref{conj equiv} et \ref{conj equiv L}. On suppose de plus que, pour tout $x\in[1/M_{(\textbf{e},\textbf{f}\,)},1[$, on a $\Delta(x)\geq 1$. Comme il a été dit en partie \ref{Enonce du crit}, le point $(i)$ du théorème \ref{conj equiv L} entraîne la validité du point $(i)$ du théorème \ref{conj equiv}. Le but de cette partie est donc de montrer que, pour tout $L\in\{1,\dots,M_{(\textbf{e},\textbf{f}\,)}\}$, on a $q_L(z)\in\mathbb{Z}[[z]]$. D'après la partie \ref{section equiv Zp}, il nous suffit de montrer que, pour tout $L\in\{1,\dots,M_{(\textbf{e},\textbf{f}\,)}\}$, tout nombre premier $p$, tout $a\in\{0,\dots,p-1\}$ et tout $K\in\mathbb{N}$, on a $\Phi_{L,p}(a+Kp)\in p\mathbb{Z}_p$. On fixe $L\in\{1,\dots,M_{(\textbf{e},\textbf{f}\,)}\}$ dans cette partie.

\subsection{Nouvelle reformulation du problème}\label{newreforprob}

Pour tout premier $p$, tout $a\in\{0,\dots,p-1\}$ et tout $K,s$ et $m$ dans $\mathbb{N}$, on définit 
$$
S(a,K,s,p,m):=\sum_{j=mp^s}^{(m+1)p^s-1}\left(\mathcal{Q}(a+jp)\mathcal{Q}(K-j)-\mathcal{Q}(j)\mathcal{Q}(a+(K-j)p)\right),
$$
où l'on pose $\mathcal{Q}(\ell)=0$ si $\ell$ est un entier strictement négatif.

Le but de cette partie est de produire, pour tout nombre premier $p$, une fonction $g_p$ de $\mathbb{N}$ dans $\mathbb{Z}_p$ telle que: si pour tout premier $p$, tout $a\in\{0,\dots,p-1\}$ et tout $K,s$ et $m$ dans $\mathbb{N}$, on a $S(a,K,s,p,m)\in p^{s+1}g_p(m)\mathbb{Z}_p$, alors on a $\Phi_{L,p}(a+pK)\in p\mathbb{Z}_p$.
Démontrer le cas $(i)$ du théorème~\ref{conj equiv L} reviendra alors à minorer convenablement la valuation $p$-adique de $S(a,K,s,p,m)$ pour tout nombre premier $p$. Cette méthode de réduction est une adaptation de l'approche du problème faite par Dwork dans \cite{Dwork 1}.

\subsubsection{Une réécriture de $\Phi_{L,p}(a+Kp)$ modulo $p\mathbb{Z}_p$}\label{section une réécriture de}

Cette étape est l'analogue d'une réécriture effectuée par Krattenthaler et Rivoal dans la partie 2 de \cite{Tanguy}. On fixe un nombre premier $p$. Nous allons montrer que
\begin{equation}\label{*}
\Phi_{L,p}(a+Kp)\equiv\sum_{j=0}^{K}H_{Lj}\big(\mathcal{Q}(a+jp)\mathcal{Q}(K-j)-\mathcal{Q}(j)\mathcal{Q}(a+(K-j)p)\big)\mod\;p\mathbb{Z}_p.
\end{equation}

Pour tout $a\in\{0,\dots,p-1\}$ et tout $j\in\mathbb{N}$, on a
\begin{align}
pH_{L(a+jp)}
&=p\left(\sum_{i=1}^{Ljp}\frac{1}{i}+\sum_{i=1}^{La}\frac{1}{Ljp+i}\right)\notag\\
&\equiv p\left(\sum_{i=1}^{Lj}\frac{1}{ip}+\sum_{i=1}^{\lfloor La/p\rfloor}\frac{1}{Ljp+ip}\right)\mod p\mathbb{Z}_p\notag\\
&\equiv H_{Lj}+\sum_{i=1}^{\lfloor La/p\rfloor}\frac{1}{Lj+i}\mod p\mathbb{Z}_p\label{(25)}.
\end{align}

Nous avons besoin d'un résultat, que l'on démontrera plus loin dans une forme plus générale (lemme \ref{lemme 24 généralisé}, partie \ref{demo lemme 10}) :

Pour tout $L\in\{1,\dots,M_{(\textbf{e},\textbf{f}\,)}\}$, tout $a\in\{0,\dots,p-1\}$ et tout $j\in\mathbb{N}$, on a 
\begin{equation}\label{lemme Q rho 1}
\mathcal{Q}(a+jp)\sum_{i=1}^{\lfloor La/p\rfloor}\frac{1}{Lj+i}\in p\mathbb{Z}_p.
\end{equation}

En appliquant \eqref{lemme Q rho 1} à \eqref{(25)}, on obtient $\mathcal{Q}(a+jp)pH_{L(a+jp)}\equiv\mathcal{Q}(a+jp)H_{Lj}\mod\;p\mathbb{Z}_p$ et comme $\mathcal{Q}(K-j)\in\mathbb{Z}_p$, cela donne
\begin{align*}
\Phi_{L,p}(a+Kp)
&=\sum_{j=0}^K\mathcal{Q}(K-j)\mathcal{Q}(a+jp)(H_{L(K-j)}-pH_{L(a+jp)})\\
&\equiv\sum_{j=0}^{K}\mathcal{Q}(K-j)\mathcal{Q}(a+jp)(H_{L(K-j)}-H_{Lj})\mod p\mathbb{Z}_p\\
&\equiv\sum_{j=0}^{K}H_{Lj}\left(\mathcal{Q}(a+jp)\mathcal{Q}(K-j)-\mathcal{Q}(j)\mathcal{Q}(a+(K-j)p)\right)\mod\;p\mathbb{Z}_p,
\end{align*}
ce qui est bien l'équation \eqref{*} attendue.

On utilise maintenant un lemme combinatoire dû à Dwork (voir \cite[Lemma 4.2, p. 308]{Dwork 1}) qui nous permet d'écrire 
$$
\sum_{j=0}^{K}H_{Lj}\left(\mathcal{Q}(a+jp)\mathcal{Q}(K-j)-\mathcal{Q}(j)\mathcal{Q}(a+(K-j)p)\right)=\sum_{s=0}^r\sum_{m=0}^{p^{r+1-s}-1}W_L(a,K,s,p,m),
$$
où $r$ est tel que $K<p^r$, et $W_L(a,K,s,p,m):=(H_{Lmp^s}-H_{L\lfloor m/p\rfloor p^{s+1}})S(a,K,s,p,m)$. Si l'on montre que, pour tout $m$ et $s$ dans $\mathbb{N}$, on a $W_L(a,K,s,p,m)\in p\mathbb{Z}_p$, alors on aura bien $\Phi_{L,p}(a+Kp)\in p\mathbb{Z}_p$, comme voulu.

Dans la suite de l'article, on notera $\{\cdot\}$ la fonction partie fractionnaire. Pour tout $m\in\mathbb{N}$, on pose $\mu_p(m):=\sum_{\ell=1}^{\infty}\textbf{1}_{[1/M_{(\textbf{e},\textbf{f}\,)},1[}(\{m/p^{\ell}\})$, où $\textbf{1}_{[1/M_{(\textbf{e},\textbf{f}\,)},1[}$ est la fonction caractéristique de $[1/M_{(\textbf{e},\textbf{f}\,)},1[$. Pour tout $m\in\mathbb{N}$, on pose $g_p(m):=p^{\mu_p(m)}$. On utilise maintenant le lemme suivant que l'on démontre dans la partie \ref{demo lemme 10}.

\begin{lemme}\label{valuation diff H}
Pour tout nombre premier $p$, tout $L\in\{1,\dots,M_{(\textbf{e},\textbf{f}\,)}\}$ et tout $s$ et $m$ dans $\mathbb{N}$, on a
$$
p^{s+1}g_p(m)\left(H_{Lmp^s}-H_{L\lfloor m/p\rfloor p^{s+1}}\right)\in p\mathbb{Z}_p.
$$
\end{lemme}

D'après le lemme \ref{valuation diff H}, si on montre que, pour tout $a\in\{0,\dots,p-1\}$, tout $K,s$ et $m$ dans $\mathbb{N}$, on a $S(a,K,s,p,m)\in p^{s+1}g_p(m)\mathbb{Z}_p$, alors, on aura $q_L(z)\in\mathbb{Z}_p[[z]]$, ce qui est la reformulation annoncée.

\subsubsection{Démonstration de \eqref{lemme Q rho 1} et du lemme \ref{valuation diff H}}\label{demo lemme 10}

Nous allons énoncer un résultat plus général permettant de démontrer le résultat \eqref{lemme Q rho 1} et le lemme \ref{valuation diff H}.

\begin{lemme}\label{lemme 24 généralisé}
Soit $s\in\mathbb{N}$, $s\geq 1$, $a\in\{0,\dots,p^s-1\}$, $M\geq 1$ et $m\in\mathbb{N}$. Soit $L\in\{1,\dots,M\}$. Si $\lfloor La/p^s\rfloor\geq 1$ alors, pour tout $u\in\{1,\dots,\lfloor La/p^s\rfloor\}$ et tout $\ell\in\{s,\dots,s+v_p(Lm+u)\}$, on a $\left\{\frac{a+mp^s}{p^{\ell}}\right\}\geq\frac{1}{M}$.
\end{lemme}

\begin{proof}
On note $m=\sum_{j=0}^{\infty}m_jp^j$ le développement $p$-adique de $m$. On a 
$$
\left\{\frac{a+mp^s}{p^{\ell}}\right\}=\frac{a+p^s\sum_{j=0}^{\ell-s-1}m_jp^j}{p^{\ell}}.
$$
On a $p^{\ell-s}\mid(u+Lm)$ et donc $p^{\ell-s}\mid (u+Lm-L(\sum_{j=\ell-s}^{\infty}m_jp^j))=(u+L(\sum_{j=0}^{\ell-s-1}m_jp^j))$. Ainsi, on obtient 
$$
p^{\ell-s}\leq u+L\left(\sum_{j=0}^{\ell-s-1}m_jp^j\right)\leq\frac{1}{p^s}La+L\left(\sum_{j=0}^{\ell-s-1}m_jp^j\right)=\frac{L}{p^{s-\ell}}\left\{\frac{a+mp^s}{p^{\ell}}\right\}\leq\frac{M}{p^{s-\ell}}\left\{\frac{a+mp^s}{p^{\ell}}\right\}
$$
et on a bien $\left\{\frac{a+mp^s}{p^{\ell}}\right\}\geq\frac{1}{M}$.
\end{proof}

Nous allons maintenant appliquer le lemme \ref{lemme 24 généralisé} pour démontrer \eqref{lemme Q rho 1} en utilisant le fait que, pour tout $n\in\mathbb{N}$, on a $v_p\left(\mathcal{Q}(n)\right)=\sum_{\ell=1}^{\infty}\Delta\big(\{\frac{n}{p^{\ell}}\}\big)$. En effet, pour tout entier positif $c$, on a $\lfloor cx\rfloor=\lfloor c\{x\}\rfloor+c\lfloor x\rfloor$ et donc $\Delta(x)=\Delta(\{x\})+\left(|\textbf{e}|-|\textbf{f}\,|\right)\lfloor x\rfloor$. Ainsi, on a $|\textbf{e}|=|\textbf{f}\,|$ si et seulement si $\Delta$ est $1$-périodique. On rappelle que si $m$ est un entier naturel, on a la formule $v_p((m)!)=\sum_{\ell=1}^{\infty}\big\lfloor \frac{m}{p^{\ell}}\big\rfloor$. Ainsi, on obtient bien 
$$
v_p\left(\mathcal{Q}(n)\right)=v_p\left(\frac{(e_1n)!\cdots(e_{q_1}n)!}{(f_1n)!\cdots(f_{q_2}n)!}\right)=\sum_{\ell=1}^{\infty}\Delta\left(\frac{n}{p^{\ell}}\right)=\sum_{\ell=1}^{\infty}\Delta\left(\left\{\frac{n}{p^{\ell}}\right\}\right).
$$

\begin{proof}[Démonstration de \eqref{lemme Q rho 1}]
Soit $L\in\{1,\dots,M_{(\textbf{e},\textbf{f}\,)}\}$, $a\in\{0,\dots,p-1\}$ et $j\in\mathbb{N}$. Il faut montrer que $\mathcal{Q}(a+jp)\sum_{i=1}^{\lfloor La/p\rfloor}\frac{1}{Lj+i}\in p\mathbb{Z}_p$. Si $\lfloor La/p\rfloor=0$, c'est évident. Supposons que $\lfloor La/p\rfloor\geq 1$. En appliquant le lemme \ref{lemme 24 généralisé} avec $s=1$, $m=j$ et $M=M_{(\textbf{e},\textbf{f}\,)}$, on obtient que, pour tout $i\in\{1,\dots,\lfloor La/p\rfloor\}$ et tout $\ell\in\{1,\dots,1+v_p(i+Lj)\}$, on a $\{(a+jp)/p^{\ell}\}\geq 1/M_{(\textbf{e},\textbf{f}\,)}$ et donc $\Delta(\{(a+jp)/p^{\ell}\})\geq 1$. Comme $\Delta$ est positive sur $\mathbb{R}$, cela donne
$$
v_p(\mathcal{Q}(a+jp))=\sum_{\ell=1}^{\infty}\Delta\left(\left\{\frac{a+jp}{p^{\ell}}\right\}\right)\geq\sum_{\ell=1}^{1+v_p(Lj+i)}\Delta\left(\left\{\frac{a+jp}{p^{\ell}}\right\}\right)\geq 1+v_p(Lj+i),
$$
ce qui achève la preuve de \eqref{lemme Q rho 1}.
\end{proof}

\begin{proof}[Démonstration du lemme \ref{valuation diff H}]
Soit $L\in\{1,\dots,M_{(\textbf{e},\textbf{f}\,)}\}$ et $m$ et $s$ dans $\mathbb{N}$. Il faut montrer que $p^{s+1}g_p(m)(H_{Lmp^s}-H_{L\lfloor m/p\rfloor p^{s+1}})\in p\mathbb{Z}_p$. On écrit $m=b+qp$, où $b\in\{0,\dots,p-1\}$ et $q\in\mathbb{N}$. On a alors $Lmp^s=Lbp^s+Lqp^{s+1}$ et $L\lfloor m/p\rfloor p^{s+1}=Lqp^{s+1}$. Ainsi, on obtient
$$
H_{Lmp^s}-H_{L\lfloor m/p\rfloor p^{s+1}}=\sum_{j=1}^{Lbp^s}\frac{1}{Lqp^{s+1}+j}\equiv\sum_{i=1}^{\lfloor Lb/p\rfloor}\frac{1}{Lqp^{s+1}+ip^{s+1}}\mod \frac{1}{p^s}\mathbb{Z}_p
$$
et donc $p^{s+1}g_p(m)(H_{Lmp^s}-H_{L\lfloor m/p\rfloor p^{s+1}})\equiv g_p(b+qp)\sum_{i=1}^{\lfloor Lb/p\rfloor}\frac{1}{Lq+i}\mod p\mathbb{Z}_p$. Il nous reste à montrer que $g_p(b+qp)\sum_{i=1}^{\lfloor Lb/p\rfloor}\frac{1}{Lq+i}\in p\mathbb{Z}_p$. Si $\lfloor Lb/p\rfloor=0$, c'est évident. Supposons que $\lfloor Lb/p\rfloor\geq 1$. En appliquant le lemme \ref{lemme 24 généralisé} avec $s=1$, $M=M_{(\textbf{e},\textbf{f}\,)}$, $a=b$ et $q$ à la place de $m$, on obtient que, pour tout $i\in\{1,\dots,\lfloor Lb/p\rfloor\}$ et tout $\ell\in\{1,\dots,1+v_p(i+Lq)\}$, on a $\{(b+qp)/p^{\ell}\}\geq1/M_{(\textbf{e},\textbf{f}\,)}$ et donc
$$
v_p(g_p(b+qp))=\sum_{\ell=1}^{\infty}\textbf{1}_{[1/M_{(\textbf{e},\textbf{f}\,)},1[}\left(\left\{\frac{b+qp}{p^{\ell}}\right\}\right)\geq\sum_{\ell=1}^{1+v_p(Lq+i)}\textbf{1}_{[1/M_{(\textbf{e},\textbf{f}\,)},1[}\left(\left\{\frac{b+qp}{p^{\ell}}\right\}\right)\geq 1+v_p(Lq+i),
$$
ce qui achève la preuve du lemme.
\end{proof}

\subsection{Application du théorème \ref{theo généralisé}}

La stratégie utilisée par Krattenthaler et Rivoal dans \cite{Tanguy} est d'appliquer le critère de Dwork (cas $k_0=0$ du théorème \ref{theo généralisé}) avec le choix des fonctions $A_r=g_r=\mathcal{Q}$ pour tout $r\geq 0$. Cette stratégie ne marche cependant pas toujours quand $\Delta$ n'est pas croissante sur $[0,1]$, même si on ne prend pas forcément $A_r=g_r=\mathcal{Q}$ pour tout $r\geq 0$. En effet, si on dispose de deux suites $(A_r)_{r\geq 0}$ et $(g_r)_{r\geq 0}$ vérifiant les conditions du critère de Dwork et telles qu'il existe un $r\geq 0$ tel que $A_r=A_{r+1}=\mathcal{Q}$ et, pour tout $s,r$ et $m$ dans $\mathbb{N}$, $p^{s+1}g_{s+r+1}(m)(H_{Lmp^s}-H_{L\lfloor m/p\rfloor p^{s+1}})\in\mathcal{O}$, alors $(iii)$ nous dit que, pour tout nombre premier $p$, tout $v<p$, tout $s,r$ et $m$ dans $\mathbb{N}$ et tout $u<p^s$, on a
\begin{equation}\label{contre-exemple}
(H_{Lmp^s}-H_{L\lfloor m/p\rfloor p^{s+1}})\left(\frac{\mathcal{Q}(v+up+mp^{s+1})}{\mathcal{Q}(v+up)}-\frac{\mathcal{Q}(u+mp^{s})}{\mathcal{Q}(u)}\right)\in \frac{1}{\textbf{g}_r(v+up)}\mathcal{O}\subset\frac{1}{\mathcal{Q}(v+up)}\mathcal{O}.
\end{equation}
Or \eqref{contre-exemple} n'est pas vérifiée par la suite $\mathcal{Q}_{(\textbf{e},\textbf{f}\,)}$ définie par $\textbf{e}=(10,5)$ et $\textbf{f}=(4,4,3,2,1,1)$. En effet, pour $p=3$, $L=10$, $v=1$, $s=1$, $m=1$ et $u=2$ on obtient 
$$ v_3\left(H_{30}\left(\frac{\mathcal{Q}_{(\textbf{e},\textbf{f}\,)}(16)}{\mathcal{Q}_{(\textbf{e},\textbf{f}\,)}(7)}-\frac{\mathcal{Q}_{(\textbf{e},\textbf{f}\,)}(5)}{\mathcal{Q}_{(\textbf{e},\textbf{f}\,)}(2)}\right)\right)=-4\quad\textup{et}\quad v_3\left(\frac{1}{\mathcal{Q}_{(\textbf{e},\textbf{f}\,)}(7)}\right)=-3.
$$

Précisons la manière dont nous allons utiliser le théorème \ref{theo généralisé} pour terminer la démonstration du cas $(i)$ du théorème \ref{conj equiv L}. Nous allons montrer dans les parties suivantes qu'il existe un entier naturel $\lambda_p$ tel qu'en posant $\textbf{A}_r=\mathcal{Q}$ et $\textbf{g}_r=g_p$ pour tout $r\geq 0$, $((\textbf{A}_r)_{r\geq 0},(\textbf{g}_r)_{r\geq 0})$ est un $\lambda_p$-couple de Dwork. En appliquant alors le théorème \ref{theo généralisé}, on obtiendra bien $S(a,K,s,p,m)\in p^{s+1}g_p(m)\mathbb{Z}_p$, comme voulu.

Dans les parties suivantes, on vérifie les hypothèses d'application du théorème \ref{theo généralisé}.

\subsection{Vérification des conditions $(i)$, $(ii)$ et $(iv)$ du théorème \ref{theo généralisé}}\label{partie 10}

On fixe $p$ un nombre premier et on note $g:=g_p$ et $\mu:=\mu_p$. Pour tout $r\geq 0$, on pose $\textbf{A}_r=\mathcal{Q}$ et $\textbf{g}_r=g$. On définit $\lambda_p$ comme étant l'unique entier naturel vérifiant $p^{\lambda_p}<M_{(\textbf{e},\textbf{f}\,)}\leq p^{\lambda_p+1}$. On va montrer dans cette partie que les suites $(\textbf{A}_r)_{r\geq 0}$ et $(\textbf{g}_r)_{r\geq 0}$ vérifient les conditions $(i)$, $(ii)$ et $(iv)$ du théorème \ref{theo généralisé} avec $k_0=\lambda_p$. Pour cela, nous allons uniquement nous servir de l'inégalité $p^{\lambda_p}<M_{(\textbf{e},\textbf{f}\,)}$. L'inégalité $M_{(\textbf{e},\textbf{f}\,)}\leq p^{\lambda_p+1}$ nous servira à démontrer la condition $(iii)$ du théorème \ref{theo généralisé} dans la partie suivante.
\medskip
\begin{itemize}
\item Vérification de $(i)$ et $(ii)$.
\end{itemize}
\medskip
Pour tout $r$ et $m$ dans $\mathbb{N}$, on a $|\textbf{A}_r(0)|_p=|\mathcal{Q}(0)|_p=1$. De plus, $v_p(\textbf{g}_r(m))=\mu(m)\geq 0$, donc on a bien $\textbf{g}_r(m)\in \mathbb{Z}_p\setminus\{0\}$. Il ne reste plus qu'à montrer que $\textbf{A}_r(m)\in \textbf{g}_r(m)\mathbb{Z}_p$, ce qui revient donc à montrer qu'on a $\mu(m)\leq v_p(\mathcal{Q}(m))$. C'est bien le cas puisque, pour tout $\ell\in\mathbb{N}$, $\ell\geq 1$, on a $\Delta\left(\left\{\frac{m}{p^{\ell}}\right\}\right)\geq \textsc{1}_{[1/M_{(\textbf{e},\textbf{f}\,)},1[}\left(\left\{\frac{m}{p^{\ell}}\right\}\right)$, car $\Delta(x)\geq 1$ pour $1>x\geq 1/M_{(\textbf{e},\textbf{f}\,)}$. On obtient bien $v_p(\mathcal{Q}(m))=\sum_{\ell=1}^{\infty}\Delta\left(\left\{\frac{m}{p^{\ell}}\right\}\right)\geq\sum_{\ell=1}^{\infty}\textsc{1}_{[1/M_{(\textbf{e},\textbf{f}\,)},1[}\left(\left\{\frac{m}{p^{\ell}}\right\}\right)=\mu(m)$. D'où le résultat.
\medskip
\begin{itemize}
\item Vérification de $(iv)$ avec $k_0=\lambda_p$.
\end{itemize}
\medskip
Si $\lambda_p=0$, il n'y a rien à vérifier. Supposons $\lambda_p\geq 1$. Soit $k\in\{1,\dots,\lambda_p\}$, $v\in\{1,\dots,p-1\}$ et $m\in\mathbb{N}$. On a $p^{\lambda_p}<M_{(\textbf{e},\textbf{f}\,)}$ donc $p^k<M_{(\textbf{e},\textbf{f}\,)}$. Ainsi $1/M_{(\textbf{e},\textbf{f}\,)}<1/p^k$ et donc, pour tout $\ell\in\{1,\dots,k\}$, on a $\textsc{1}_{[1/M_{(\textbf{e},\textbf{f}\,)},1[}\left(\left\{\frac{v+mp^k}{p^{\ell}}\right\}\right)=\textsc{1}_{[1/M_{(\textbf{e},\textbf{f}\,)},1[}\left(\left\{\frac{v}{p^{\ell}}\right\}\right)=1$. On a alors 
$$ v_p(g(v+mp^k))=\sum_{\ell=1}^{\infty}\textsc{1}_{[1/M_{(\textbf{e},\textbf{f}\,)},1[}\left(\left\{\frac{v+mp^k}{p^{\ell}}\right\}\right)=k+\sum_{\ell=k+1}^{\infty}\textsc{1}_{[1/M_{(\textbf{e},\textbf{f}\,)},1[}\left(\left\{\frac{v+mp^k}{p^{\ell}}\right\}\right).
$$
Pour tout $\ell\geq k+1$, on a $\left\{\frac{v+mp^k}{p^{\ell}}\right\}>\left\{\frac{mp^k}{p^{\ell}}\right\}$. En effet, on écrit $m=cp^{\ell-k}+d$ où $c\in\mathbb{N}$ et $d\in\{0,\dots,p^{\ell-k}-1\}$, et on obtient bien $\left\{\frac{v+mp^k}{p^{\ell}}\right\}=\left\{\frac{v+dp^k}{p^{\ell}}\right\}=\frac{v+dp^k}{p^{\ell}}>\frac{dp^k}{p^{\ell}}=\left\{\frac{dp^k}{p^{\ell}}\right\}=\left\{\frac{mp^k}{p^{\ell}}\right\}$. 

Donc $v_p(g(v+mp^k))\geq k+\sum_{\ell=k+1}^{\infty}\textsc{1}_{[1/M_{(\textbf{e},\textbf{f}\,)},1[}\left(\left\{\frac{mp^k}{p^{\ell}}\right\}\right)=k+v_p(g(mp^k))$ et on a bien $g(v+mp^k)\in p^kg(mp^k)\mathbb{Z}_p$. De plus,
\begin{align*}
v_p(g(mp^k))
&=\sum_{\ell=1}^{\infty}\textsc{1}_{[1/M_{(\textbf{e},\textbf{f}\,)},1[}\left(\left\{\frac{mp^k}{p^{\ell}}\right\}\right)=\sum_{\ell=k+1}^{\infty}\textsc{1}_{[1/M_{(\textbf{e},\textbf{f}\,)},1[}\left(\left\{\frac{mp^k}{p^{\ell}}\right\}\right)\\
&=\sum_{\ell=1}^{\infty}\textsc{1}_{[1/M_{(\textbf{e},\textbf{f}\,)},1[}\left(\left\{\frac{m}{p^{\ell}}\right\}\right)=v_p(g(m)),
\end{align*}
donc $g(mp^k)\in g(m)\mathbb{Z}_p$, comme voulu.

\subsection{Vérification de la condition $(iii)$ du théorème \ref{theo généralisé}}

On rappelle que $\lambda_p$ est l'unique entier naturel vérifiant $p^{\lambda_p}<M_{(\textbf{e},\textbf{f}\,)}\leq p^{\lambda_p+1}$. On va montrer que les suites $(\textbf{A}_r)_{r\geq 0}$ et $(\textbf{g}_r)_{r\geq 0}$ vérifient la condition $(iii)$ du théorème \ref{theo généralisé} avec $k_0=\lambda_p$. Pour cela, nous n'utiliserons que l'inégalité $M_{(\textbf{e},\textbf{f}\,)}\leq p^{\lambda_p+1}$. D'après la partie précédente, la vérification de la condition $(iii)$ montrera que $((\textbf{A}_r)_{r\geq 0},(\textbf{g}_r)_{r\geq 0})$ est un $\lambda_p$-couple de Dwork et ainsi achèvera la preuve du point $(i)$ du théorème~\ref{conj equiv L}.
Nous allons montrer que la condition $(iii)$ est vérifiée en deux étapes, selon que $v_p(m)\geq \lambda_p$ ou que $v_p(m)\leq \lambda_p-1$. La preuve est relativement longue et décomposée en nombreuses étapes.

\subsubsection{Lorsque $v_p(m)\leq \lambda_p-1$}

Nous devons montrer ici que si $\lambda_p\geq 1$ et si $v_p(m)=k\leq \lambda_p-1$, alors on a $\frac{\mathcal{Q}(mp)}{\mathcal{Q}(0)}-\frac{\mathcal{Q}(m)}{\mathcal{Q}(0)}\in p^{k+1}g(m)\mathbb{Z}_p$. Comme $\mathcal{Q}(0)=1$, cela revient à montrer que l'on a
\begin{equation}\label{Le But 2}
\mathcal{Q}(m)\left(\frac{\mathcal{Q}(mp)}{\mathcal{Q}(m)}-1\right)\in p^{k+1}g(m)\mathbb{Z}_p.
\end{equation}
On écrit $m=p^km'$, où $m'\in\mathbb{N}$ et \eqref{Le But 2} devient
\begin{equation}\label{Le But 2 bis}
\mathcal{Q}(m)\left(\frac{\mathcal{Q}(m'p^{k+1})}{\mathcal{Q}(m'p^k)}-1\right)\in p^{k+1}g(m)\mathbb{Z}_p.
\end{equation}
Pour conclure, on utilise le lemme suivant.

\begin{lemme}\label{lemme Q(ap)/Q(a)}
Pour tout $s\in\mathbb{N}$, tout $c\in\{0,\dots,p^s-1\}$ et tout $m\in\mathbb{N}$, on a
$$
\frac{\mathcal{Q}(c)}{\mathcal{Q}(cp)}\frac{\mathcal{Q}(cp+mp^{s+1})}{\mathcal{Q}(c+mp^s)}\in 1+p^{s+1}\mathbb{Z}_p.
$$
\end{lemme}

En appliquant le lemme \ref{lemme Q(ap)/Q(a)} avec $s=k$, $c=0$ et $m'$ à la place de $m$, on obtient $\frac{\mathcal{Q}(m'p^{k+1})}{\mathcal{Q}(m'p^k)}\in 1+p^{k+1}\mathbb{Z}_p$ et donc $\left(\frac{\mathcal{Q}(m'p^{k+1})}{\mathcal{Q}(m'p^k)}-1\right)\in p^{k+1}\mathbb{Z}_p$. De plus, d'après la condition $(ii)$ du théorème~\ref{theo généralisé}, on a $\mathcal{Q}(m)\in g(m)\mathbb{Z}_p$, donc on a bien \eqref{Le But 2 bis}. Il ne nous reste plus qu'à démontrer le lemme~\ref{lemme Q(ap)/Q(a)}. Pour cela, nous allons utiliser certaines propriétés de la fonction gamma $p$-adique définie par $\Gamma_p(n):=(-1)^n\gamma_p(n)$, où $\gamma_p(n):=\underset{(k,p)=1}{\prod_{k=1}^{n-1}}k$. On peut étendre $\Gamma_p$ à tout $\mathbb{Z}_p$ mais on n'en aura pas besoin ici. On résume les propriétés qui nous serviront pour prouver le lemme \ref{lemme Q(ap)/Q(a)}.

\begin{lemme}\label{lemme7}
\begin{itemize}
\item[$(i)$] Pour tout $n\in\mathbb{N}$, on a l'identité $\frac{(np)!}{n!}=p^n\gamma_p(1+np)$.
\item[$(ii)$] Pour tout $k,n$ et $s$ dans $\mathbb{N}$, on a $\Gamma_p(k+np^s)\equiv\Gamma_p(k)\;\mod\;p^s$.
\end{itemize}
\end{lemme}

Le point $(i)$ du lemme \ref{lemme7} s'obtient en remarquant que $\gamma_p(1+np)=\frac{(np)!}{n!p^n}$. Le point $(ii)$ du lemme \ref{lemme7} est le lemme 1.1 de \cite{Lang}.

\begin{proof}[Démonstration du lemme \ref{lemme Q(ap)/Q(a)}]
On a
\begin{align*}
\frac{\mathcal{Q}(cp+mp^{s+1})}{\mathcal{Q}(c+mp^{s})}
&=\prod_{i=1}^{q_1}\frac{(e_i(cp+mp^{s+1}))!}{(e_i(c+mp^{s}))!}\prod_{i=1}^{q_2}\frac{(f_i(c+mp^{s}))!}{(f_i(cp+mp^{s+1}))!}\\
&=\left(\prod_{i=1}^{q_1}p^{e_i(c+mp^{s})}\gamma_p(1+e_icp+e_imp^{s+1})\right)\left(\prod_{i=1}^{q_2}\frac{1}{p^{f_i(c+mp^{s})}\gamma_p(1+f_icp+f_imp^{s+1})}\right)\\
&=(p^{mp^{s}})^{(|\textbf{e}|-|\textbf{f}\,|)}\frac{\prod_{i=1}^{q_1}(p^{e_ic}(-1)^{1+e_icp+e_imp^{s+1}}\Gamma_p(1+e_icp+e_imp^{s+1}))}{\prod_{i=1}^{q_2}\left(p^{f_ic}(-1)^{1+f_icp+f_imp^{s+1}}\Gamma_p(1+f_icp+f_imp^{s+1})\right)}\\
&=(p^{mp^{s}}(-1)^{mp^{s+1}})^{(|\textbf{e}|-|\textbf{f}\,|)}\frac{\prod_{i=1}^{q_1}\left(p^{e_ic}(-1)^{1+e_icp}\;\Gamma_p(1+e_icp+e_imp^{s+1})\right)}{\prod_{i=1}^{q_2}\left(p^{f_ic}(-1)^{1+f_icp}\;\Gamma_p(1+f_icp+f_imp^{s+1})\right)}\\
&=\frac{\prod_{i=1}^{q_1}p^{e_ic}(-1)^{1+e_icp}}{\prod_{i=1}^{q_2}p^{f_ic}(-1)^{1+f_icp}}\cdot\frac{\prod_{i=1}^{q_1}\Gamma_p(1+e_icp+e_imp^{s+1})}{\prod_{i=1}^{q_2}\Gamma_p(1+f_icp+f_imp^{s+1})}.
\end{align*}
D'après le point $(ii)$ du lemme \ref{lemme7}, pour tout $n\in\mathbb{N}$, on a $\Gamma_p(1+ncp+nmp^{s+1})\equiv\Gamma_p(1+ncp)\mod\;p^{s+1}$. On obtient donc 
$$
\frac{\prod_{i=1}^{q_1}\Gamma_p(1+e_icp+e_imp^{s+1})}{\prod_{i=1}^{q_2}\Gamma_p(1+f_icp+f_imp^{s+1})}=\frac{\prod_{i=1}^{q_1}\left(\Gamma_p(1+e_icp)+O(p^{s+1})\right)}{\prod_{i=1}^{q_2}\left(\Gamma_p(1+f_icp)+O(p^{s+1})\right)},
$$
où l'on note $x=O(p^k)$ lorsque $x\in p^k\mathbb{Z}_p$. De plus, par définition de $\Gamma_p$, pour tout $n\in\mathbb{N}$, on a $\Gamma_p(1+ncp)\in\mathbb{Z}_p^{\times}$. On obtient alors
$$
\frac{\prod_{i=1}^{q_1}\left(\Gamma_p(1+e_icp)+O(p^{s+1})\right)}{\prod_{i=1}^{q_2}\left(\Gamma_p(1+f_icp)+O(p^{s+1})\right)}=\frac{\prod_{i=1}^{q_1}\Gamma_p(1+e_icp)}{\prod_{i=1}^{q_2}\Gamma_p(1+f_icp)}(1+O(p^{s+1}))
$$
et ainsi,
\begin{align*}
\frac{\mathcal{Q}(cp+mp^{s+1})}{\mathcal{Q}(c+mp^{s})}
&=\frac{\prod_{i=1}^{q_1}p^{e_ic}(-1)^{1+e_icp}}{\prod_{i=1}^{q_2}p^{f_ic}(-1)^{1+f_icp}}\cdot\frac{\prod_{i=1}^{q_1}\Gamma_p(1+e_icp)}{\prod_{i=1}^{q_2}\Gamma_p(1+f_icp)}(1+O(p^{s+1}))\\
&=\frac{\prod_{i=1}^{q_1}p^{e_ic}}{\prod_{i=1}^{q_2}p^{f_ic}}\cdot\frac{\prod_{i=1}^{q_1}\gamma_p(1+e_icp)}{\prod_{i=1}^{q_2}\gamma_p(1+f_icp)}(1+O(p^{s+1}))\\
&=\frac{\mathcal{Q}(cp)}{\mathcal{Q}(c)}(1+O(p^{s+1})).
\end{align*}
C'est ce qu'il fallait démontrer.
\end{proof}

\subsubsection{Lorsque $v_p(m)\geq \lambda_p$}

Le but de cette partie est de démontrer le fait suivant.

Soit $p$ un nombre premier et $\lambda_p$ l'unique entier naturel tel que $p^{\lambda_p}<M_{(\textbf{e},\textbf{f}\,)}\leq p^{\lambda_p+1}$. Pour tout $s\in\mathbb{N}$, tout $v\in\{0,\dots,p-1\}$, tout $u\in\{0,\dots,p^s-1\}$ et tout $m\in\mathbb{N}$, on a
\begin{equation}\label{Hypothèse (iii)}
\frac{\mathcal{Q}(v+up+mp^{s+\lambda_p+1})}{\mathcal{Q}(v+up)}-\frac{\mathcal{Q}(u+mp^{s+\lambda_p})}{\mathcal{Q}(u)}\in p^{s+\lambda_p+1}\frac{g(mp^{\lambda_p})}{g(v+up)}\mathbb{Z}_p.
\end{equation}

L'équation \eqref{Hypothèse (iii)} est vérifiée si et seulement si, pour tout $v\in\{0,\dots,p-1\}$, tout $u\in\{0,\dots,p^s-1\}$ et tout $m\in\mathbb{N}$, on a
\begin{equation}\label{eq 3.11}
\left(1-\frac{\mathcal{Q}(v+up)}{\mathcal{Q}(u)}\frac{\mathcal{Q}(u+mp^{s+\lambda_p})}{\mathcal{Q}(v+up+mp^{s+\lambda_p+1})}\right)\frac{\mathcal{Q}(v+up+mp^{s+\lambda_p+1})}{\mathcal{Q}(v+up)}\in p^{s+\lambda_p+1}\frac{g(mp^{\lambda_p})}{g(v+up)}\mathbb{Z}_p.
\end{equation}
Dans la suite, on pose $ X_s(v,u,m):=\frac{\mathcal{Q}(v+up)}{\mathcal{Q}(u)}\frac{\mathcal{Q}(u+mp^{s+\lambda_p})}{\mathcal{Q}(v+up+mp^{s+\lambda_p+1})}$. Ainsi, pour démontrer \eqref{Hypothèse (iii)}, il nous suffit de montrer que
\begin{equation}\label{eq 3.13}
(X_s(v,u,m)-1)\frac{\mathcal{Q}(v+up+mp^{s+\lambda_p+1})}{g(mp^{\lambda_p})}\in p^{s+\lambda_p+1}\frac{\mathcal{Q}(v+up)}{g(v+up)}\mathbb{Z}_p.
\end{equation}
Afin d'estimer la valuation de $X_s(v,u,m)-1$, posons, pour tout $v\in\{0,\dots,p-1\}$, tout $u\in\{0,\dots,p^s-1\}$ et tout $s$ et $m$ dans $\mathbb{N}$, 
$$
Y_s(v,u,m):=\frac{\prod_{i=1}^{q_2}\prod_{j=1}^{\lfloor f_iv/p\rfloor}\left(1+\frac{f_imp^{s+\lambda_p}}{f_iu+j}\right)}{\prod_{i=1}^{q_1}\prod_{j=1}^{\lfloor e_iv/p\rfloor}\left(1+\frac{e_imp^{s+\lambda_p}}{e_iu+j}\right)}.
$$
Pour $s$ et $m$ dans $\mathbb{N}$ et $a\in\{0,\dots,p^s-1\}$, on note $\eta_s(a,m):=\sum_{\ell=s+1}^{\infty}\Delta\left(\left\{\frac{a+mp^{s}}{p^{\ell}}\right\}\right)$. On va énoncer un certain nombre de lemmes, que l'on démontre dans la partie \ref{démo lemme Y}.

\begin{lemme}\label{définition de Y}
Pour tout $v\in\{0,\dots,p-1\}$, tout $u\in\{0,\dots,p^s-1\}$ et tout $s$ et $m$ dans $\mathbb{N}$, on a $X_s(v,u,m)\in Y_s(v,u,m)\left(1+p^{s+\lambda_p+1}\mathbb{Z}_p\right)$ et 
$$
v_p(Y_s(v,u,m))=(\eta_{s+\lambda_p+1}(v+up,0)-\eta_{s+\lambda_p}(u,0))-(\eta_{s+\lambda_p+1}(v+up,m)-\eta_{s+\lambda_p}(u,m)).
$$
\end{lemme}

\begin{lemme}\label{lemme theta Q(a)}
Soit $s\in\mathbb{N}$, $v\in\{0,\dots,p-1\}$ et $u\in\{0,\dots,p^{s}-1\}$. Si $\{(v+up)/p^{j}\}<1/M_{(\textbf{e},\textbf{f}\,)}$ pour un $j\in\{1,\dots,s+\lambda_p+1\}$, alors $Y_s(v,u,m)\in 1+p^{s+\lambda_p-j+2}\mathbb{Z}_p$.
\end{lemme}

\begin{lemme}\label{valuation du quotient avec g}
Pour tout $s\in\mathbb{N}$, tout $a\in\{0,\dots,p^{s+1}-1\}$ et tout $m\in\mathbb{N}$, on a
\begin{equation}\label{eq du lemme 20}
\eta_{s+\lambda_p+1}(a,m)\geq \mu(mp^{\lambda_p})
\end{equation}
et
\begin{equation}\label{cor valuation du quotient avec g}
v_p\left(\frac{\mathcal{Q}(a+mp^{s+\lambda_p+1})}{g(mp^{\lambda_p})}\right)\geq \sum_{\ell=1}^{s+\lambda_p+1}\Delta\left(\left\{\frac{a}{p^{\ell}}\right\}\right).
\end{equation}
\end{lemme}

\begin{lemme}\label{lemme Q(a) g(a)}
Pour tout $s\in\mathbb{N}$ et tout $a\in\{0,\dots,p^{s+1}-1\}$, on a: $\eta_{s+\lambda_p+1}(a,0)=0$,
$$
v_p(Q(a))=\sum_{\ell=1}^{s+\lambda_p+1}\Delta\left(\left\{\frac{a}{p^{\ell}}\right\}\right)\quad\textup{et}\quad v_p(g(a))=\sum_{\ell=1}^{s+\lambda_p+1}\textsc{1}_{[1/M_{(\textbf{e},\textbf{f}\,)},1[}\left(\left\{\frac{a}{p^{\ell}}\right\}\right).
$$
\end{lemme}

\begin{Remarque}
Le lemme \ref{lemme Q(a) g(a)} repose essentiellement sur l'inégalité $M_{(\textbf{e},\textbf{f}\,)}\leq p^{\lambda_p+1}$.
\end{Remarque}

Afin de montrer \eqref{eq 3.13}, on va maintenant différencier deux cas.

\medskip
\begin{itemize}
\item \textit{Cas} 1: Supposons qu'il existe $j\in\{1,\dots,s+\lambda_p+1\}$ tel que 
\end{itemize}
\begin{equation}\label{ouhou}
\left\{\frac{v+up}{p^j}\right\}<\frac{1}{M_{(\textbf{e},\textbf{f}\,)}}.
\end{equation}
Soit $j_0$ le plus petit des $j\in\{1,\dots,s+\lambda_p+1\}$ vérifiant \eqref{ouhou}. D'après le lemme \ref{lemme theta Q(a)} appliqué en $j_0$, on obtient $Y_s(v,u,m)\in 1+p^{s+\lambda_p-j_0+2}\mathbb{Z}_p$ et donc, d'après le lemme \ref{définition de Y}, $v_p(X_s(v,u,m)-1)\geq s+\lambda_p-j_0+2$.

Montrons qu'on a $v_p(g(v+up))\geq j_0-1$. Si $j_0=1$, c'est évident. Si $j_0\geq 2$, alors, pour tout $\ell\in\{1,\dots,j_0-1\}$, on a $\{(v+up)/p^{\ell}\}\geq 1/M_{(\textbf{e},\textbf{f}\,)}$ et donc $v_p(g(v+up))=\sum_{\ell=1}^{\infty}\textbf{1}_{[1/M_{(\textbf{e},\textbf{f}\,)},1[}\left(\left\{\frac{v+up}{p^{\ell}}\right\}\right)\geq j_0-1$.
D'après \eqref{cor valuation du quotient avec g}, on obtient 
$$
v_p\left((X_s(v,u,m)-1)\frac{\mathcal{Q}(v+up+mp^{s+\lambda_p+1})}{g(mp^{\lambda_p})}\right)\geq v_p(X_s(v,u,m)-1)+\sum_{\ell=1}^{s+\lambda_p+1}\Delta\left(\left\{\frac{v+up}{p^{\ell}}\right\}\right).
$$
D'où
\begin{align}
v_p&\left((X_s(v,u,m)-1)\frac{\mathcal{Q}(v+up+mp^{s+\lambda_p+1})}{g(mp^{\lambda_p})}\right)\notag\\
&\geq v_p(X_s(v,u,m)-1)+v_p(g(v+up))+\left(\sum_{\ell=1}^{s+\lambda_p+1}\Delta\left(\left\{\frac{v+up}{p^{\ell}}\right\}\right)-v_p(g(v+up))\right)\notag\\
&\geq (s+\lambda_p-j_0+2)+j_0-1+v_p\left(\frac{\mathcal{Q}(v+up)}{g(v+up)}\right)\label{explipli45}\\
&\geq s+\lambda_p+1+v_p\left(\frac{\mathcal{Q}(v+up)}{g(v+up)}\right),\notag
\end{align}
où l'on a utilisé l'identité $v_p(\mathcal{Q}(v+up))=\sum_{\ell=1}^{s+\lambda_p+1}\Delta(\{\frac{v+up}{p^{\ell}}\})$ du lemme \ref{lemme Q(a) g(a)} dans \eqref{explipli45}. On a donc bien \eqref{eq 3.13} dans ce cas.
\medskip
\begin{itemize}
\item \textit{Cas} 2: Supposons que, pour tout $j\in\{1,\dots,s+\lambda_p+1\}$, on ait $\{(v+up)/p^{j}\}\geq 1/M_{(\textbf{e},\textbf{f}\,)}$.
\end{itemize} 
\medskip

Ainsi, on a $v_p(g(v+up))=\sum_{\ell=1}^{\infty}\textbf{1}_{[1/M_{(\textbf{e},\textbf{f}\,)},1[}(\{(v+up)/p^{\ell}\})\geq s+\lambda_p+1$. 

Si $v_p(Y_s(v,u,m))\geq 0$, alors, d'après le lemme \ref{définition de Y}, $v_p(X_s(v,u,m)-1)\geq 0$ et, d'après \eqref{cor valuation du quotient avec g} et le lemme \ref{lemme Q(a) g(a)}, on a
\begin{align*}
v_p\left(\frac{\mathcal{Q}(v+up+mp^{s+\lambda_p+1})}{g(mp^{\lambda_p})}\right)&\geq \sum_{\ell=1}^{s+\lambda_p+1}\Delta\left(\left\{\frac{v+up}{p^{\ell}}\right\}\right)=v_p(g(v+up))+v_p\left(\frac{\mathcal{Q}(v+up)}{g(v+up)}\right)\\
&\geq s+\lambda_p+1+v_p\left(\frac{\mathcal{Q}(v+up)}{g(v+up)}\right).
\end{align*}
On a donc bien \eqref{eq 3.13}.

Supposons maintenant que $v_p(Y_s(v,u,m))<0$. Dans ce cas, d'après le lemme \ref{définition de Y}, on a
\begin{align*}
v_p(X_s(v,u,m)-1)&=v_p(Y_s(v,u,m))\\
&=(\eta_{s+\lambda_p+1}(v+up,0)-\eta_{s+\lambda_p}(u,0))-(\eta_{s+\lambda_p+1}(v+up,m)-\eta_{s+\lambda_p}(u,m)).
\end{align*}
D'après le lemme \ref{lemme Q(a) g(a)}, on a $\eta_{s+\lambda_p+1}(v+up,0)=0$ et $\eta_{s+\lambda_p}(u,0)=0$ et donc 
$$
v_p(X_s(v,u,m)-1)=\eta_{s+\lambda_p}(u,m)-\eta_{s+\lambda_p+1}(v+up,m).
$$
De plus, 
\begin{align*}
v_p(\mathcal{Q}(v+up+mp^{s+\lambda_p+1}))
&=\sum_{\ell=1}^{\infty}\Delta\left(\left\{\frac{v+up+mp^{s+\lambda_p+1}}{p^{\ell}}\right\}\right)\\
&=\sum_{\ell=1}^{s+\lambda_p+1}\Delta\left(\left\{\frac{v+up}{p^{\ell}}\right\}\right)+\sum_{\ell=s+\lambda_p+2}^{\infty}\Delta\left(\left\{\frac{v+up+mp^{s+\lambda_p+1}}{p^{\ell}}\right\}\right)\\
&=\sum_{\ell=1}^{s+\lambda_p+1}\Delta\left(\left\{\frac{v+up}{p^{\ell}}\right\}\right)+\eta_{s+\lambda_p+1}(v+up,m).
\end{align*}
Ainsi, on obtient
\begin{eqnarray*}
\lefteqn{v_p\left((X_s(v,u,m)-1)\frac{\mathcal{Q}(v+up+mp^{s+\lambda_p+1})}{g(mp^{\lambda_p})}\right)}\\
&=&\eta_{s+\lambda_p}(u,m)-\eta_{s+\lambda_p+1}(v+up,m)+\sum_{\ell=1}^{s+\lambda_p+1}\Delta\left(\left\{\frac{v+up}{p^{\ell}}\right\}\right)+\eta_{s+\lambda_p+1}(v+up,m)-\mu(mp^{\lambda_p})\\
&=&\sum_{\ell=1}^{s+\lambda_p+1}\Delta\left(\left\{\frac{v+up}{p^{\ell}}\right\}\right)+\eta_{s+\lambda_p}(u,m)-\mu(mp^{\lambda_p})\\
&=&v_p(g(v+up))+v_p\left(\frac{\mathcal{Q}(v+up)}{g(v+up)}\right)+\eta_{s+\lambda_p}(u,m)-\mu(mp^{\lambda_p})\\
&\geq& s+\lambda_p+1+v_p\left(\frac{\mathcal{Q}(v+up)}{g(v+up)}\right)+\eta_{s+\lambda_p}(u,m)-\mu(mp^{\lambda_p}).
\end{eqnarray*}
Si $s=0$, alors on a $u=0$ et $\eta_{\lambda_p}(0,m)=\sum_{\ell=\lambda_p+1}^{\infty}\Delta(\{\frac{mp^{\lambda_p}}{p^{\ell}}\})\geq\sum_{\ell=\lambda_p+1}^{\infty}\textbf{1}_{[1/M_{(\textbf{e},\textbf{f}\,)},1[}(\{\frac{mp^{\lambda_p}}{p^{\ell}}\})=\mu(mp^{\lambda_p})$ et on a bien \eqref{eq 3.13}. En revanche, si $s\geq 1$ alors, en utilisant le lemme \ref{valuation du quotient avec g} avec $s-1$ à la place de $s$ et $a=u$, on obtient $\eta_{s+\lambda_p}(u,m)\geq\mu(mp^{\lambda_p})$, ce qui donne bien \eqref{eq 3.13}. Ceci achève la preuve de l'équation \eqref{Hypothèse (iii)}, modulo celles des divers lemmes.

\subsubsection{Démonstration des lemmes \ref{définition de Y}, \ref{lemme theta Q(a)}, \ref{valuation du quotient avec g} et \ref{lemme Q(a) g(a)}}\label{démo lemme Y}

\begin{proof}[Démonstration du lemme \ref{définition de Y}]
On veut montrer que
$X_s(v,u,m)\in Y_s(v,u,m)(1+p^{s+\lambda_p+1}\mathbb{Z}_p)$.

On a 
\begin{equation}\label{toutadroite} X_s(v,u,m)=\frac{\mathcal{Q}(v+up)}{\mathcal{Q}(up)}\frac{\mathcal{Q}(up+mp^{s+\lambda_p+1})}{\mathcal{Q}(v+up+mp^{s+\lambda_p+1})}\cdot\frac{\mathcal{Q}(up)}{\mathcal{Q}(u)}\frac{\mathcal{Q}(u+mp^{s+\lambda_p})}{\mathcal{Q}(up+mp^{s+\lambda_p+1})}.
\end{equation}
En appliquant le lemme \ref{lemme Q(ap)/Q(a)} avec $c=u$ et $s+\lambda_p$ pour $s$ au terme tout à droite de \eqref{toutadroite}, on obtient
\begin{equation}\label{pour valuation X}
X_s(v,u,m)\in\frac{\mathcal{Q}(v+up)}{\mathcal{Q}(up)}\frac{\mathcal{Q}(up+mp^{s+\lambda_p+1})}{\mathcal{Q}(v+up+mp^{s+\lambda_p+1})}(1+p^{s+\lambda_p+1}\mathbb{Z}_p).
\end{equation}
De plus, on a
\begin{align*}
\frac{\mathcal{Q}(v+up)}{\mathcal{Q}(up)}\cdot\frac{\mathcal{Q}(up+mp^{s+\lambda_p+1})}{\mathcal{Q}(v+up+mp^{s+\lambda_p+1})}
&=\frac{\Big(\prod_{i=1}^{q_1}\prod_{k=1}^{e_iv}(e_iup+k)\Big)\left(\prod_{i=1}^{q_2}\prod_{k=1}^{f_iv}(f_i(up+mp^{s+\lambda_p+1})+k)\right)}{\left(\prod_{i=1}^{q_2}\prod_{k=1}^{f_iv}(f_iup+k)\right)\Big(\prod_{i=1}^{q_1}\prod_{k=1}^{e_iv}(e_i(up+mp^{s+\lambda_p+1})+k)\Big)}\\
&=\frac{\prod_{i=1}^{q_2}\prod_{k=1}^{f_iv}\left(1+\frac{f_imp^{s+\lambda_p+1}}{f_iup+k}\right)}{\prod_{i=1}^{q_1}\prod_{k=1}^{e_iv}\left(1+\frac{e_imp^{s+\lambda_p+1}}{e_iup+k}\right)}.
\end{align*}

Si $d\in\{e_1,\dots,e_{q_1},f_1,\dots,f_{q_2}\}$ et $k\in\{1,\dots,dv\}$, alors $dup+k$ est divisible par $p$ si et seulement s'il existe $j\in\{1,\dots,\lfloor dv/p\rfloor\}$ tel que $k=jp$. On a donc 
$$
\prod_{k=1}^{dv}\left(1+\frac{dmp^{s+\lambda_p+1}}{dup+k}\right)=\prod_{j=1}^{\lfloor dv/p\rfloor}\left(1+\frac{dmp^{s+\lambda_p}}{du+j}\right)(1+O(p^{s+\lambda_p+1})).
$$
D'où
\begin{align*}
\frac{\mathcal{Q}(v+up)}{\mathcal{Q}(up)}\cdot\frac{\mathcal{Q}(up+mp^{s+\lambda_p+1})}{\mathcal{Q}(v+up+mp^{s+\lambda_p+1})}
&=\frac{\prod_{i=1}^{q_2}\prod_{j=1}^{\lfloor f_iv/p\rfloor}\left(1+\frac{f_imp^{s+\lambda_p}}{f_iu+j}\right)}{\prod_{i=1}^{q_1}\prod_{j=1}^{\lfloor e_iv/p\rfloor}\left(1+\frac{e_imp^{s+\lambda_p}}{e_iu+j}\right)}(1+O(p^{s+\lambda_p+1}))\\
&=Y_s(v,u,m)(1+O(p^{s+\lambda_p+1}))
\end{align*}
et donc $X_s(v,u,m)\in Y_s(v,u,m)(1+p^{s+\lambda_p+1}\mathbb{Z}_p)$, comme voulu.
\medskip

On va maintenant montrer qu'on a bien aussi 
$$
v_p(Y_s(v,u,m))=\eta_{s+\lambda_p+1}(v+up,0)-\eta_{s+\lambda_p}(u,0)-(\eta_{s+\lambda_p+1}(v+up,m)-\eta_{s+\lambda_p}(u,m)).
$$
On a vu ci-dessus que $v_p(Y_s(v,u,m))=v_p(X_s(v,u,m))$. Or, d'après \eqref{pour valuation X}, on a aussi
\begin{align*}
v_p(X_s(v,u,m))
&=v_p\left(\frac{\mathcal{Q}(v+up)}{\mathcal{Q}(up)}\cdot\frac{\mathcal{Q}(up+mp^{s+\lambda_p+1})}{\mathcal{Q}(v+up+mp^{s+\lambda_p+1})}\right)\\
&=v_p(\mathcal{Q}(v+up))-v_p(\mathcal{Q}(up))+v_p(\mathcal{Q}(up+mp^{s+\lambda_p+1}))-v_p(\mathcal{Q}(v+up+mp^{s+\lambda_p+1}))\\
&=\sum_{\ell=1}^{\infty}\Delta\left(\left\{\frac{v+up}{p^{\ell}}\right\}\right)-\sum_{\ell=1}^{\infty}\Delta\left(\left\{\frac{up}{p^{\ell}}\right\}\right)+\sum_{\ell=1}^{\infty}\Delta\left(\left\{\frac{up+mp^{s+\lambda_p+1})}{p^{\ell}}\right\}\right)\\
&-\sum_{\ell=1}^{\infty}\Delta\left(\left\{\frac{v+up+mp^{s+\lambda_p+1}}{p^{\ell}}\right\}\right).
\end{align*}
On a
\begin{multline*}
\sum_{\ell=1}^{\infty}\Delta\left(\left\{\frac{v+up}{p^{\ell}}\right\}\right)-\sum_{\ell=1}^{\infty}\Delta\left(\left\{\frac{v+up+mp^{s+\lambda_p+1}}{p^{\ell}}\right\}\right)\\
=\sum_{\ell=1}^{\infty}\Delta\left(\left\{\frac{v+up}{p^{\ell}}\right\}\right)-\sum_{\ell=1}^{s+\lambda_p+1}\Delta\left(\left\{\frac{v+up}{p^{\ell}}\right\}\right)-\sum_{\ell=s+\lambda_p+2}^{\infty}\Delta\left(\left\{\frac{v+up+mp^{s+\lambda_p+1}}{p^{\ell}}\right\}\right)\\
=\sum_{\ell=s+\lambda_p+2}^{\infty}\Delta\left(\left\{\frac{v+up}{p^{\ell}}\right\}\right)-\sum_{\ell=s+\lambda_p+2}^{\infty}\Delta\left(\left\{\frac{v+up+mp^{s+\lambda_p+1}}{p^{\ell}}\right\}\right)\\
=\eta_{s+\lambda_p+1}(v+up,0)-\eta_{s+\lambda_p+1}(v+up,m),
\end{multline*}
et
\begin{align*}
\sum_{\ell=1}^{\infty}
&\Delta\left(\left\{\frac{up}{p^{\ell}}\right\}\right)-\sum_{\ell=1}^{\infty}\Delta\left(\left\{\frac{up+mp^{s+\lambda_p+1}}{p^{\ell}}\right\}\right)\\
&=\sum_{\ell=s+\lambda_p+2}^{\infty}\Delta\left(\left\{\frac{up}{p^{\ell}}\right\}\right)-\sum_{\ell=s+\lambda_p+2}^{\infty}\Delta\left(\left\{\frac{up+mp^{s+\lambda_p+1}}{p^{\ell}}\right\}\right)\\
&=\sum_{\ell=s+\lambda_p+1}^{\infty}\Delta\left(\left\{\frac{u}{p^{\ell}}\right\}\right)-\sum_{\ell=s+\lambda_p+1}^{\infty}\Delta\left(\left\{\frac{u+mp^{s+\lambda_p}}{p^{\ell}}\right\}\right)=\eta_{s+\lambda_p}(u,0)-\eta_{s+\lambda_p}(u,m).
\end{align*}
Donc, $v_p(Y_s(v,u,m))=\eta_{s+\lambda_p+1}(v+up,0)-\eta_{s+\lambda_p}(u,0)-(\eta_{s+\lambda_p+1}(v+up,m)-\eta_{s+\lambda_p}(u,m))$,
ce qui achève la preuve du lemme.
\end{proof}

\begin{proof}[Démonstration du lemme \ref{lemme theta Q(a)}]
Soit $s\in\mathbb{N}$, $v\in\{0,\dots,p-1\}$ et $u\in\{0,\dots,p^s-1\}$. Soit $\sum_{k=0}^{\infty}u_kp^k$ le développement $p$-adique de $u$ et $L\in\{e_1,\dots,e_{q_1},f_1,\dots,f_{q_2}\}$. On définit $s+\lambda_p+1$ entiers naturels par les relations $b_{L,0}:=\lfloor Lv/p\rfloor$ et $b_{L,k+1}:=\lfloor(Lu_k+b_{L,k})/p\rfloor$ pour $k\in\{0,\dots,s+\lambda_p-1\}$. On note, pour tout $x\in\mathbb{R}$, $\lceil x\rceil$ l'entier immédiatement supérieur à $x$ et on définit $s+\lambda_p+1$ entiers naturels par les relations $a_{L,0}:=1$ et $a_{L,k+1}:=\lceil(Lu_k+a_{L,k})/p\rceil$. Dans un premier temps, nous allons montrer par récurrence sur $r$ que l'assertion $\mathcal{A}_{r}$ : 
$$
\prod_{n=1}^{\lfloor Lv/p\rfloor}\left(1+\frac{Lmp^{s+\lambda_p}}{Lu+n}\right)=\prod_{n=a_{L,r}}^{b_{L,r}}\left(1+\frac{Lmp^{s+\lambda_p-r}}{L\sum_{k=r}^{\infty}u_kp^{k-r}+n}\right)\left(1+O(p^{s+\lambda_p-r+1})\right),
$$
est vraie pour tout $r\in\{0,\dots,s+\lambda_p\}$.

On a $b_{L,0}=\lfloor Lv/p\rfloor$ et $a_{L,0}=1$, donc $\mathcal{A}_{0}$ est vraie.

Soit $r\geq 0$. Supposons que $\mathcal{A}_{r}$ est vraie et montrons $\mathcal{A}_{r+1}$. Si $n\in\{a_{L,r},\dots,b_{L,r}\}$, alors $p$ divise $L\sum_{k=r}^{\infty}u_kp^{k-r}+n$ si et seulement si $p$ divise $Lu_r+n$, \textit{i.e.} si et seulement s'il existe $i\in\{\lceil(Lu_r+a_{L,r})/p\rceil,\dots,\lfloor(Lu_r+b_{L,r})/p\rfloor\}$ tel que $Lu_r+n=ip$. On obtient donc
\begin{align}
\prod_{n=a_{L,r}}^{b_{L,r}}\left(1+\frac{Lmp^{s+\lambda_p-r}}{L\sum_{k=r}^{\infty}u_kp^{k-r}+n}\right)
&=\prod_{i=a_{L,r+1}}^{b_{L,r+1}}\left(1+\frac{Lmp^{s+\lambda_p-r}}{L\sum_{k=r+1}^{\infty}u_kp^{k-r}+ip}\right)(1+O(p^{s+\lambda_p-r}))\notag\\
&=\prod_{i=a_{L,r+1}}^{b_{L,r+1}}\left(1+\frac{Lmp^{s+\lambda_p-r-1}}{L\sum_{k=r+1}^{\infty}u_kp^{k-r-1}+i}\right)(1+O(p^{s+\lambda_p-r}))\label{Alr}.
\end{align}
D'après $\mathcal{A}_{r}$ et \eqref{Alr}, on a bien $\mathcal{A}_{r+1}$, ce qui achève la récurrence sur $r$.

Soit $L\in\{e_1,\dots,e_{q_1},f_1,\dots,f_{q_2}\}$. Nous allons montrer par récurrence sur $k$ que l'assertion $\mathcal{B}_{k}$ : $a_{L,k}\geq 1$ et $b_{L,k}\leq \lfloor L\{(v+up)/p^{k+1}\}\rfloor$ est vraie pour tout $k\in\{0,\dots,s+\lambda_p\}$. 

On a $a_{L,0}=1$ et $b_{L,0}=\lfloor Lv/p\rfloor=\lfloor L\{(v+up)/p\}\rfloor$, donc $\mathcal{B}_0$ est vraie.

Soit $k\geq 0$. Supposons que $\mathcal{B}_k$ est vraie et montrons $\mathcal{B}_{k+1}$. On a $a_{L,k+1}=\lceil(Lu_k+a_{L,k})/p\rceil$ et $b_{L,k+1}=\left\lfloor(Lu_k+b_{L,k})/p\right\rfloor$, donc $a_{L,k+1}\geq\lceil(Lu_k+1)/p\rceil\geq 1$ et
$$
b_{L,k+1}\leq\left\lfloor\frac{Lu_k}{p}+\frac{L}{p}\left\{\frac{v+up}{p^{k+1}}\right\}\right\rfloor=\left\lfloor L\left(\frac{u_kp^{k+1}}{p^{k+2}}+\frac{v+p\sum_{i=0}^{k-1}u_ip^i}{p^{k+2}}\right)\right\rfloor=\left\lfloor L\left\{\frac{v+up}{p^{k+2}}\right\}\right\rfloor,
$$
ce qui achève la récurrence sur $k$.

Soit $j\in\{1,\dots,s+\lambda_p+1\}$ tel que $\{(v+up)/p^j\}<1/M_{(\textbf{e},\textbf{f}\,)}$. Pour tout $L\in\{e_1,\dots,e_{q_1},f_1,\dots,f_{q_2}\}$, on obtient, \textit{via} $\mathcal{B}_{j-1}$, que $a_{L,j-1}\geq 1$ et $b_{L,j-1}\leq\lfloor L\{(v+up)/p^{j}\}\rfloor\leq\lfloor M_{(\textbf{e},\textbf{f}\,)}\{(v+up)/p^{j}\}\rfloor=0$. Ainsi, d'après $\mathcal{A}_{j-1}$, on a 
$$
\prod_{n=1}^{\lfloor Lv/p\rfloor}\left(1+\frac{Lmp^{s+\lambda_p}}{Lu+n}\right)=1+O(p^{s+\lambda_p-j+2})
$$
et donc
$$
Y_s(v,u,m)=\frac{\prod_{i=1}^{q_2}\prod_{n=1}^{\lfloor f_iv/p\rfloor}\left(1+\frac{f_imp^{s+\lambda_p}}{f_iu+n}\right)}{\prod_{i=1}^{q_1}\prod_{n=1}^{\lfloor e_iv/p\rfloor}\left(1+\frac{e_imp^{s+\lambda_p}}{e_iu+n}\right)}=\frac{1+O(p^{s+\lambda_p-j+2})}{1+O(p^{s+\lambda_p-j+2})}=1+O(p^{s+\lambda_p-j+2}), 
$$
ce qui achève la preuve du lemme \ref{lemme theta Q(a)}.
\end{proof}

\begin{proof}[Démonstration du lemme \ref{valuation du quotient avec g}]
Dans un premier temps, nous allons montrer qu'on a bien \eqref{eq du lemme 20}. \'Ecrivons $m=\sum_{k=0}^qm_kp^k$, où $m_k\in\{0,\dots,p-1\}$. On a
\begin{align*}
\eta_{s+\lambda_p+1}&(a,m)-\mu(mp^{\lambda_p})\\
&=\sum_{\ell=s+\lambda_p+2}^{\infty}\Delta\left(\left\{\frac{a+mp^{s+\lambda_p+1}}{p^{\ell}}\right\}\right)-\sum_{\ell=1}^{\infty}\textsc{1}_{[1/M_{(\textbf{e},\textbf{f}\,)},1[}\left(\left\{\frac{mp^{\lambda_p}}{p^{\ell}}\right\}\right)\\
&=\sum_{\ell=s+\lambda_p+2}^{\infty}\left(\Delta\left(\left\{\frac{a+mp^{s+\lambda_p+1}}{p^{\ell}}\right\}\right)-\textsc{1}_{[1/M_{(\textbf{e},\textbf{f}\,)},1[}\left(\left\{\frac{mp^{s+\lambda_p+1}}{p^{\ell}}\right\}\right)\right)\\
&=\sum_{\ell=s+\lambda_p+2}^{\infty}\left(\Delta\left(\frac{a+\sum_{k=0}^{\ell-s-\lambda_p-2}m_kp^{k+s+\lambda_p+1}}{p^{\ell}}\right)-\textsc{1}_{[1/M_{(\textbf{e},\textbf{f}\,)},1[}\left(\frac{\sum_{k=0}^{\ell-s-\lambda_p-2}m_kp^{k+s+\lambda_p+1}}{p^{\ell}}\right)\right).
\end{align*}
De plus, pour tout $\ell\geq s+\lambda_p+2$, on a
$$
0\leq \frac{\sum_{k=0}^{\ell-s-\lambda_p-2}m_kp^{k+s+\lambda_p+1}}{p^{\ell}}\leq \frac{a+\sum_{k=0}^{\ell-s-\lambda_p-2}m_kp^{k+s+\lambda_p+1}}{p^{\ell}}\leq\frac{p^{\ell}-1}{p^{\ell}}<1.
$$
Donc
\begin{align*}
\textsc{1}_{[1/M_{(\textbf{e},\textbf{f}\,)},1[}\left(\frac{\sum_{k=0}^{\ell-s-\lambda_p-2}m_kp^{k+s+\lambda_p+1}}{p^{\ell}}\right)=1\quad
&\Longrightarrow\quad 1>\frac{\sum_{k=0}^{\ell-s-\lambda_p-2}m_kp^{k+s+\lambda_p+1}}{p^{\ell}}\geq \frac{1}{M_{(\textbf{e},\textbf{f}\,)}}\\
&\Longrightarrow\quad 1>\frac{a+\sum_{k=0}^{\ell-s-\lambda_p-2}m_kp^{k+s+\lambda_p+1}}{p^{\ell}}\geq \frac{1}{M_{(\textbf{e},\textbf{f}\,)}}\\
&\Longrightarrow\quad \Delta\left(\frac{a+\sum_{k=0}^{\ell-s-\lambda_p-2}m_kp^{k+s+\lambda_p+1}}{p^{\ell}}\right)\geq 1
\end{align*}
et donc $\eta_{s+\lambda_p+1}(a,m)-\mu(mp^{\lambda_p})\geq 0$. Ceci termine la preuve de \eqref{eq du lemme 20}.

Montrons maintenant \eqref{cor valuation du quotient avec g}. On a
\begin{align}
v_p\left(\frac{\mathcal{Q}(a+mp^{s+\lambda_p+1})}{g(mp^{\lambda_p})}\right)
&=\sum_{\ell=1}^{\infty}\Delta\left(\left\{\frac{a+mp^{s+\lambda_p+1}}{p^{\ell}}\right\}\right)-\mu(mp^{\lambda_p})\notag\\
&=\sum_{\ell=1}^{s+\lambda_p+1}\Delta\left(\left\{\frac{a}{p^{\ell}}\right\}\right)+\sum_{\ell=s+\lambda_p+2}^{\infty}\Delta\left(\left\{\frac{a+mp^{s+\lambda_p+1}}{p^{\ell}}\right\}\right)-\mu(mp^{\lambda_p})\notag\\
&=\sum_{\ell=1}^{s+\lambda_p+1}\Delta\left(\left\{\frac{a}{p^{\ell}}\right\}\right)+\eta_{s+\lambda_p+1}(a,m)-\mu(mp^{\lambda_p}),\notag\\
&\geq \sum_{\ell=1}^{s+\lambda_p+1}\Delta\left(\left\{\frac{a}{p^{\ell}}\right\}\right)\label{expli expli 2}.
\end{align}
où dans \eqref{expli expli 2}, on a utilisé l'inégalité \eqref{eq du lemme 20}.
\end{proof}

\begin{proof}[Démonstration du lemme \ref{lemme Q(a) g(a)}]
On a $\eta_{s+\lambda_p+1}(a,0)=\sum_{\ell=s+\lambda_p+2}^{\infty}\Delta\left(\left\{\frac{a}{p^{\ell}}\right\}\right)$, $ v_p(Q(a))=\sum_{\ell=1}^{\infty}\Delta\left(\left\{\frac{a}{p^{\ell}}\right\}\right)$ et $ v_p(g(a))=\sum_{\ell=1}^{\infty}\textsc{1}_{[1/M_{(\textbf{e},\textbf{f}\,)},1[}\left(\left\{\frac{a}{p^{\ell}}\right\}\right)$. Or, si $\ell\geq s+\lambda_p+2$, on a $\left\{\frac{a}{p^{\ell}}\right\}\leq\frac{p^{s+1}-1}{p^{\ell}}<\frac{1}{p^{\lambda_p+1}}\leq\frac{1}{M_{(\textbf{e},\textbf{f}\,)}}$, car $M_{(\textbf{e},\textbf{f}\,)}\leq p^{\lambda_p+1}$. Donc $\Delta\left(\left\{\frac{a}{p^{\ell}}\right\}\right)=\textsc{1}_{[1/M_{(\textbf{e},\textbf{f}\,)},1[}\left(\left\{\frac{a}{p^{\ell}}\right\}\right)=0$. D'où le résultat.
\end{proof}

\section{Démonstration des points $(ii)$ des théorèmes \ref{conj equiv} et \ref{conj equiv L}}\label{section(ii)}

On se place sous les hypothèses des thèorèmes \ref{conj equiv} et \ref{conj equiv L}. On suppose de plus que $\Delta$ s'annule sur $[1/M_{(\textbf{e},\textbf{f}\,)},1[$. Le but de cette partie est de montrer qu'il n'existe qu'un nombre fini de premiers $p$ tels que $q(z)\in z\mathbb{Z}_p[[z]]$ et que, pour tout $L\in\{1,\dots,M_{(\textbf{e},\textbf{f}\,)}\}$, il n'existe qu'un nombre fini de premiers $p$ tels que $q_L(z)\in\mathbb{Z}_p[[z]]$. On fixe $L\in\{1,\dots,M_{(\textbf{e},\textbf{f}\,)}\}$ dans cette partie.
 
D'après la partie \ref{section equiv Zp}, il suffit de montrer que, pour tout nombre premier $p$ assez grand, il existe $a\in\{0,\dots,p-1\}$ et $K\in\mathbb{N}$ tels que $\Phi_p(a+Kp)\notin p\mathbb{Z}_p$ et $\Phi_{L,p}(a+Kp)\notin p\mathbb{Z}_p$. En fait, nous allons montrer que pour tout nombre premier $p$ assez grand, il existe $a\in\{0,\dots,p-1\}$ tel que $\Phi_p(a)\notin p\mathbb{Z}_p$ et tel que $\Phi_{L,p}(a)\notin p\mathbb{Z}_p$ ou $\Phi_{L,p}(a+p)\notin p\mathbb{Z}_p$.

\subsection{Démonstration du point $(ii)$ du théorème \ref{conj equiv}}\label{(ii) 1}

Nous allons montrer que, pour tout nombre premier $p$ assez grand, il existe $a\in\{0,\dots,p-1\}$ tel que $\Phi_p(a)\notin p\mathbb{Z}_p$. Dans ce cas, on a $\Phi_p(a)=-p\mathcal{Q}(a)\Big(\sum_{i=1}^{q_1}e_iH_{e_ia}-\sum_{i=1}^{q_2}f_iH_{f_ia}\Big)$.
\medskip

Pour tout $d\in\mathbb{N}$, on a $H_{da}=\sum_{i=1}^{da}\frac{1}{i}\equiv\sum_{j=1}^{\lfloor da/p\rfloor}\frac{1}{jp}\mod\mathbb{Z}_p$. Pour tout $x\in[0,1]$, on pose $\Psi(x):=\sum_{i=1}^{q_1}\sum_{j=1}^{\lfloor e_ix\rfloor}\frac{e_i}{j}-\sum_{i=1}^{q_2}\sum_{j=1}^{\lfloor f_ix\rfloor}\frac{f_i}{j}$. Ainsi, pour tout $a\in\{0,\dots,p-1\}$, on a $\Phi_p(a)\equiv -\mathcal{Q}(a)\Psi(a/p)\mod p\mathbb{Z}_p$. Il suffit donc de montrer que, pour tout premier $p$ assez grand, il existe $a\in\{0,\dots,p-1\}$ tel que $v_p\left(\mathcal{Q}(a)\right)=v_p\left(\Psi(a/p)\right)=0$. 
\medskip

Pour tout $d\in\mathbb{N}$, $d\geq 1$, les sauts de l'application $x\mapsto\lfloor dx\rfloor$ sur $[0,1]$ se font aux abscisses $\frac{1}{d},\frac{2}{d},\dots,\frac{d-1}{d},1$. Soit $\mathcal{D}:=\{e_1,\dots,e_{q_1},f_1,\dots,f_{q_2}\}$ et $\gamma_1<\dots<\gamma_t=1$ les rationnels qui vérifient $\{\gamma_1,\dots,\gamma_t\}=\bigcup_{d\in\mathcal{D}}\{\frac{1}{d},\frac{2}{d},\dots,\frac{d-1}{d},1\}$. Pour tout $i\in\{1,\dots,t\}$, on note $m_i\in\mathbb{Z}$ l'amplitude du saut de $\Delta$ en $\gamma_i$. Pour tout $x\in[1/M_{(\textbf{e},\textbf{f}\,)},1[$, il existe un $i\in\{1,\dots,t-1\}$ tel que $x\in[\gamma_i,\gamma_{i+1}[$ et on a alors $\Psi(x)=\sum_{n=1}^{q_1}\sum_{j=1}^{\lfloor e_nx\rfloor}\frac{1}{j/e_n}-\sum_{n=1}^{q_2}\sum_{j=1}^{\lfloor f_nx\rfloor}\frac{1}{j/f_n}=\sum_{k=1}^i\frac{m_k}{\gamma_k}$.
\medskip

La fonction $\Delta$ prend la valeur $0$ sur $[1/M_{(\textbf{e},\textbf{f}\,)},1[$ donc il existe un $i_0\in\{1,\dots,t-1\}$ tel que $\Delta$ soit nulle sur $[\gamma_{i_0},\gamma_{i_0+1}[$. Il existe une constante $\mathcal{P}_1$ telle que, pour tout premier $p\geq\mathcal{P}_1$, il existe un $a_p\in\{0,\dots,p-1\}$ tel que $a_p/p\in[\gamma_{i_0},\gamma_{i_0+1}[$. Ainsi, pour tout premier $p\geq\mathcal{P}_1$, on a $\Delta(a_p/p)=0$ et $\Psi(a_p/p)=\sum_{k=1}^{i_0}\frac{m_k}{\gamma_k}$. Il existe une constante $\mathcal{P}_2\geq\mathcal{P}_1$ telle que, pour tout premier $p\geq\mathcal{P}_2$, on ait $\sum_{k=1}^{i_0}\frac{m_k}{\gamma_k}\in\mathbb{Z}_p^{\times}\cup\{0\}$. Le lemme suivant nous permet de conclure que l'on a pas $\sum_{k=1}^{i_0}\frac{m_k}{\gamma_k}=0$.

\begin{lemme}\label{newlemma}
Soit $\textbf{e}:=(e_1,\dots,e_{q_1})$ et $\textbf{f}:=(f_1,\dots,f_{q_2})$ deux suites d'entiers strictement positifs disjointes. On se place dans les conditions ci-dessus. S'il existe $i_0\in\{1,\dots,t\}$ tel que $\Delta_{(\textbf{e},\textbf{f}\,)}$ soit positive sur $[\gamma_1,\gamma_{i_0}[$ et nulle sur $[\gamma_{i_0},\gamma_{i_0+1}[$, alors on a
$$
\sum_{k=1}^{i_0}\frac{m_k}{\gamma_k}>0\quad\textup{et}\quad\prod_{k=1}^{i_0}\left(1+\frac{1}{\gamma_k}\right)^{m_k}>1.
$$
\end{lemme}

\begin{Remarque}
On utilisera l'inégalité $\prod_{k=1}^{i_0}\left(1+\frac{1}{\gamma_k}\right)^{m_k}>1$ dans la démonstration du point $(ii)$ du théorème \ref{conj equiv L}.
\end{Remarque}

D'après ce lemme, pour tout premier $p\geq\max(\mathcal{P}_2,M_{(\textbf{e},\textbf{f}\,)})$, il existe un $a_p\in\{0,\dots,p-1\}$ tel que $\Delta(a_p/p)=0$ et $\Psi(a_p/p)\in\mathbb{Z}_p^{\times}$, car $\Psi(a_p/p)=\sum_{k=1}^{i_0}\frac{m_k}{\gamma_k}\neq 0$. Pour tout $\ell\in\mathbb{N}$, $\ell\geq 2$, on a $a_p/p^{\ell}<1/p\leq1/M_{(\textbf{e},\textbf{f}\,)}$ et donc $\Delta(\{a_p/p^{\ell}\})=0$. Ainsi, on obtient bien $v_p(\mathcal{Q}(a_p))=\sum_{\ell=1}^{\infty}\Delta(\{a_p/p^{\ell}\})=0$ et $\Psi(a_p/p)\in\mathbb{Z}_p^{\times}$, ce qui termine la démonstration de l'assertion $(ii)$ du théorème \ref{conj equiv}.

\begin{proof}[Démonstration du lemme \ref{newlemma}]
Les suites $\textbf{e}$ et $\textbf{f}$ sont disjointes donc $\gamma_1=1/M_{(\textbf{e},\textbf{f}\,)}$ est un saut effectif de $\Delta$. Comme, pour tout $x\in[0,1]$, on a $\Delta(x)\geq 0$, on obtient $\Delta(\gamma_1)\geq 1$, \textit{i.e.} $m_1\geq 1$. Comme $\Delta$ est nulle sur $[\gamma_{i_0},\gamma_{i_0+1}[$, on a $\sum_{k=1}^{i_0}m_k=0$. Or $m_1\geq 1$ donc $i_0\geq 2$. On va maintenant montrer par récurrence sur $\ell$ que, pour tout $\ell\in\{2,\dots,i_0\}$, on a
\begin{equation}\label{igeq2}
\sum_{k=1}^{\ell}\frac{m_k}{\gamma_k}>\frac{1}{\gamma_{\ell}}\sum_{k=1}^{\ell}m_k\quad\textup{et}\quad\prod_{k=1}^{\ell}\left(1+\frac{1}{\gamma_k}\right)^{m_k}>\left(1+\frac{1}{\gamma_{\ell}}\right)^{\sum_{k=1}^{\ell}m_k}.
\end{equation}
On a $\frac{1}{\gamma_1}>\dots>\frac{1}{\gamma_t}$ et $m_1\geq 1$ donc 
$$
\frac{m_1}{\gamma_1}+\frac{m_2}{\gamma_2}>\frac{m_1}{\gamma_2}+\frac{m_2}{\gamma_2}\quad\textup{et}\quad\left(1+\frac{1}{\gamma_1}\right)^{m_1}\left(1+\frac{1}{\gamma_2}\right)^{m_2}>\left(1+\frac{1}{\gamma_2}\right)^{m_1+m_2}.
$$ 
Ainsi, \eqref{igeq2} est vraie pour $\ell=2$. Si $i_0\geq 3$, soit $\ell\in\{2,\dots,i_0-1\}$ tel que \eqref{igeq2} soit vraie. Par hypothèse de récurrence, on a 
$$\sum_{k=1}^{\ell+1}\frac{m_k}{\gamma_k}>\frac{1}{\gamma_{\ell}}\sum_{k=1}^{\ell}m_k+\frac{m_{\ell+1}}{\gamma_{\ell+1}}\quad\textup{et}\quad\prod_{k=1}^{\ell+1}\left(1+\frac{1}{\gamma_k}\right)^{m_k}>\left(1+\frac{1}{\gamma_{\ell}}\right)^{\sum_{k=1}^{\ell}m_k}\left(1+\frac{1}{\gamma_{\ell+1}}\right)^{m_{\ell+1}}.
$$ 
Comme $\Delta$ est positive sur $[\gamma_1,\gamma_{i_0}[$, on obtient $\sum_{k=1}^{\ell}m_k\geq 0$ et donc $\frac{1}{\gamma_{\ell}}\sum_{k=1}^{\ell}m_k\geq \frac{1}{\gamma_{\ell+1}}\sum_{k=1}^{\ell}m_k$ et $\left(1+\frac{1}{\gamma_{\ell}}\right)^{\sum_{k=1}^{\ell}m_k}\geq\left(1+\frac{1}{\gamma_{\ell+1}}\right)^{\sum_{k=1}^{\ell}m_k}$. Ainsi, on a bien 
$$\sum_{k=1}^{\ell+1}\frac{m_k}{\gamma_k}>\frac{1}{\gamma_{\ell+1}}\sum_{k=1}^{\ell+1}m_k\quad\textup{et}\quad\prod_{k=1}^{\ell+1}\left(1+\frac{1}{\gamma_k}\right)^{m_k}>\left(1+\frac{1}{\gamma_{\ell+1}}\right)^{\sum_{k=1}^{\ell+1}m_k},
$$ 
ce qui achève la récurrence. 

En utilisant \eqref{igeq2} avec $\ell=i_0$, on obtient $\sum_{k=1}^{i_0}\frac{m_k}{\gamma_k}>\frac{1}{\gamma_{i_0}}\sum_{k=1}^{i_0}m_k= 0$ et $\prod_{k=1}^{i_0}\left(1+\frac{1}{\gamma_k}\right)^{m_k}>\left(1+\frac{1}{\gamma_{i_0}}\right)^{\sum_{k=1}^{i_0}m_k}=1$, ce qui achève la preuve du lemme.
\end{proof}

\subsection{Démonstration du point $(ii)$ du théorème \ref{conj equiv L}}

On fixe $L\in\{1,\dots,M_{(\textbf{e},\textbf{f}\,)}\}$ dans cette partie. Nous allons montrer que, pour tout nombre premier $p$ assez grand, il existe $a\in\{0,\dots,p-1\}$ tel que $\Phi_{L,p}(a)\notin p\mathbb{Z}_p$ ou $\Phi_{L,p}(a+p)\notin p\mathbb{Z}_p$. On rappelle qu'on a $\Phi_{L,p}(a)=-p\mathcal{Q}(a)H_{La}$.
\medskip

Comme on l'a dit dans la partie \ref{(ii) 1}, on a $H_{La}\equiv\sum_{j=1}^{\lfloor La/p\rfloor}\frac{1}{jp}\mod p\mathbb{Z}_p$, ce qui donne $\Phi_{L,p}(a)\equiv-\mathcal{Q}(a)H_{\lfloor La/p\rfloor}\mod p\mathbb{Z}_p$. Pour tout premier $p$ et tout $a\in\{0,\dots,p-1\}$, on a $H_{\lfloor La/p\rfloor}\in\{0,H_1,\dots,H_L\}$. Il existe une constante $\mathcal{P}_1$ telle que, pour tout premier $p\geq\mathcal{P}_1$, on ait $\{H_1,\dots,H_L\}\subset\mathbb{Z}_p^{\times}$. D'après la partie \ref{(ii) 1}, il existe une constante $\mathcal{P}_2>M_{(\textbf{e},\textbf{f}\,)}$ telle que, pour tout premier $p\geq\mathcal{P}_2$, il existe $a\in\{0,\dots,p-1\}$ tel que $v_p(\mathcal{Q}(a))=0$. Pour tout premier $p\geq\mathcal{P}_0:=\max(\mathcal{P}_1,\mathcal{P}_2)$, alors il existe $a_p\in\{0,\dots,p-1\}$ tel que $v_p(\mathcal{Q}(a_p))=0$ et $H_{\lfloor La_p/p\rfloor}\in\mathbb{Z}_p^{\times}\cup\{0\}$. Ainsi, si $\lfloor La_p/p\rfloor\geq 1$, alors on a $\Phi_{L,p}(a_p)\notin p\mathbb{Z}_p$. On remarque cependant que si $L=1$, alors on a toujours $\lfloor 1\cdot a_p/p\rfloor=0$ et donc $\Phi_{1,p}(a_p)\in p\mathbb{Z}_p$. On fixe un premier $p\geq\mathcal{P}_0$ dans la suite.
\medskip 

Nous allons maintenant montrer que si $\lfloor La_p/p\rfloor=0$, alors on a $\Phi_{L,p}(a_p+p)\notin p\mathbb{Z}_p$, ce qui achèvera la preuve du point $(ii)$ du théorème \ref{conj equiv L}. Si $\lfloor La_p/p\rfloor=0$, alors on a $-p\mathcal{Q}(a_p)H_{La_p}\in p\mathbb{Z}_p$. De plus, on a $H_{L(a_p+p)}=\sum_{j=1}^{L(a_p+p)}\frac{1}{j}\equiv\frac{1}{p}\sum_{i=1}^{\lfloor L(a_p+p)/p\rfloor}\frac{1}{i}\mod\mathbb{Z}_p$ et $\lfloor L(a_p+p)/p\rfloor=L$ donc $pH_{L(a_p+p)}\equiv H_L\mod p\mathbb{Z}_p$. On obtient 
\begin{align*}
\Phi_{L,p}(a_p+p)
&=\sum_{j=0}^1\mathcal{Q}(1-j)\mathcal{Q}(a_p+jp)(H_{L(1-j)}-pH_{L(a_p+jp)})\\
&=\mathcal{Q}(1)\mathcal{Q}(a_p)(H_L-pH_{La_p})-\mathcal{Q}(a_p+p)pH_{L(a_p+p)}\\
&\equiv\mathcal{Q}(1)\mathcal{Q}(a_p)H_L-\mathcal{Q}(a_p+p)H_L\mod p\mathbb{Z}_p\\
&\equiv H_L\mathcal{Q}(a_p)\mathcal{Q}(1)\left(1-\frac{\mathcal{Q}(a_p+p)}{\mathcal{Q}(a_p)\mathcal{Q}(1)}\right)\mod p\mathbb{Z}_p,
\end{align*}
avec $H_L$ et $\mathcal{Q}(a_p)$ dans $\mathbb{Z}_p^{\times}$. De plus, on a $p>M_{(\textbf{e},\textbf{f}\,)}$ donc $v_p(\mathcal{Q}(1))=\sum_{\ell=1}^{\infty}\Delta(\{1/p^{\ell}\})=0$ car $\Delta$ est nulle sur $[0,1/M_{(\textbf{e},\textbf{f}\,)}[$. Il nous suffit donc de montrer que $1-\frac{\mathcal{Q}(a_p+p)}{\mathcal{Q}(a_p)\mathcal{Q}(1)}\in\mathbb{Z}_p^{\times}$. On a
$$
\frac{\mathcal{Q}(a_p+p)}{\mathcal{Q}(a_p)}=\frac{\prod_{i=1}^{q_1}\prod_{j=1}^{e_ip}(e_ia_p+j)}{\prod_{i=1}^{q_2}\prod_{j=1}^{f_ip}(f_ia_p+j)}.
$$
Pour tout $d\in\{1,\dots,M_{(\textbf{e},\textbf{f}\,)}\}$, on a $\prod_{j=1}^{dp}(da_p+j)=\prod_{k=0}^{d-1}\prod_{j=1}^{p}(da_p+j+kp)$. Comme $p>M_{(\textbf{e},\textbf{f}\,)}$, on a $\lfloor da_p/p^2\rfloor=0$ donc il existe $n_d\in\{0,\dots,p-1\}$ tel que $da_p=n_d+\lfloor da_p/p\rfloor p$. Pour tout $k\in\{0,\dots,d-1\}$, on a alors
$$
\underset{j\neq p-n_d}{\prod_{j=1}^p}(da_p+j+kp)=\prod_{m=1}^{p-1}(m+O(p))=-1+O(p)\quad\textup{et}\quad da_p+p-n_d+kp=p(\lfloor da_p/p\rfloor+1+k),
$$
ce qui donne $\prod_{j=1}^{dp}(da_p+j)=(-1+O(p))^dp^d\prod_{k=0}^{d-1}(\lfloor da_p/p\rfloor+1+k)=((-1)^d+O(p))p^d\prod_{k=0}^{d-1}(\lfloor da_p/p\rfloor+1+k)$. Ainsi, on obtient
\begin{align*}
\frac{\mathcal{Q}(a_p+p)}{\mathcal{Q}(a_p)}
&=\frac{\prod_{i=1}^{q_1}\Big(\big((-1)^{e_i}+O(p)\big)p^{e_i}\prod_{k=1}^{e_i}(\lfloor e_ia_p/p\rfloor+k)\Big)}{\prod_{i=1}^{q_2}\Big(\big((-1)^{f_i}+O(p)\big)p^{f_i}\prod_{k=1}^{f_i}(\lfloor f_ia_p/p\rfloor+k)\Big)}\\
&=\big((-1)^{|\textbf{e}|}+O(p)\big)\big((-1)^{|\textbf{f}\,|}+O(p)\big)p^{|\textbf{e}|-|\textbf{f}\,|}\frac{\prod_{i=1}^{q_1}(\lfloor e_ia_p/p\rfloor+e_i)!}{\prod_{i=1}^{q_2}(\lfloor f_ia_p/p\rfloor+f_i)!}\cdot\frac{\prod_{i=1}^{q_2}\lfloor f_ia_p/p\rfloor!}{\prod_{i=1}^{q_1}\lfloor e_ia_p/p\rfloor!}\\
&=\frac{\prod_{i=1}^{q_2}\lfloor f_ia_p/p\rfloor!}{\prod_{i=1}^{q_1}\lfloor e_ia_p/p\rfloor!}\cdot\frac{\prod_{i=1}^{q_1}e_i!}{\prod_{i=1}^{q_2}f_i!}\cdot\frac{\prod_{i=1}^{q_1}\prod_{k=1}^{\lfloor e_ia_p/p\rfloor}(e_i+k)}{\prod_{i=1}^{q_2}\prod_{k=1}^{\lfloor f_ia_p/p\rfloor}(f_i+k)}(1+O(p)),
\end{align*}
ce qui donne
$$
\frac{\mathcal{Q}(a_p+p)}{\mathcal{Q}(a_p)\mathcal{Q}(1)}=\frac{\prod_{i=1}^{q_1}\prod_{k=1}^{\lfloor e_ia_p/p\rfloor}(\frac{e_i}{k}+1)}{\prod_{i=1}^{q_2}\prod_{k=1}^{\lfloor f_ia_p/p\rfloor}(\frac{f_i}{k}+1)}(1+O(p))=\prod_{k=1}^{i_0}\left(1+\frac{1}{\gamma_k}\right)^{m_k}(1+O(p)),
$$
où la dernière égalité a lieu car $a_p/p\in[\gamma_{i_0},\gamma_{i_0+1}[$.
Le lemme \ref{newlemma} nous dit que $\prod_{k=1}^{i_0}\left(1+\frac{1}{\gamma_k}\right)^{m_k}> 1$, on ne peut donc avoir $\prod_{k=1}^{i_0}\left(1+\frac{1}{\gamma_k}\right)^{m_k}=1+O(p)$ que pour un nombre fini de premier $p$, car rappelons que si $x\in\mathbb{Q}$ est dans $p\mathbb{Z}_p$ pour une infinité de nombres premiers $p$, alors $x=0$. Ainsi, pour tout premier $p$ assez grand, on a bien $\lfloor La_p/p\rfloor=0\Rightarrow\Phi_{L,p}(a_p+p)\notin p\mathbb{Z}_p$, ce qui termine la preuve du point $(ii)$ du théorème \ref{conj equiv L}.

\section{Résultats hypergéométriques}\label{préliminaires}

Le but de cette partie est de caractériser les séries hypergéométriques généralisées dont les coefficients peuvent se mettre sous forme factorielle. Cela nous permettra de décrire les sauts de l'application de Landau sur $[0,1]$ et d'en conclure que $\Delta_{(\textbf{e},\textbf{f}\,)}$ est croissante sur $[0,1[$ si et seulement si l'équation différentielle fuchsienne associée à $F_{(\textbf{e},\textbf{f}\,)}$ a tous ses exposants égaux à $0$ à l'origine.

\subsection{Séries hypergéométriques définies par des quotients de factorielles}\label{sectionunicité}

Soit $\textbf{e}$ et $\textbf{f}$ deux suites d'entiers positifs. La série formelle $F_{(\textbf{e},\textbf{f}\,)}(z)=\sum_{n=0}^{\infty}\mathcal{Q}_{(\textbf{e},\textbf{f}\,)}(n)z^n$ est une série hypergéométrique (voir \cite[Lemma 4.1, p. 431]{Bober}). Nous allons caractériser les suites $(\alpha_1,\dots,\alpha_r)\in\mathbb{C}^r$ et $(\beta_1,\dots,\beta_s)\in(\mathbb{C}\setminus\mathbb{Z}_{\leq 0})^s$ telles qu'il existe deux suites d'entiers positifs $\textbf{e}$ et $\textbf{f}$ vérifiant
\begin{equation}\label{écriturunik}
_{r}F_{s}\left(\begin{array}{c}\alpha_1,\dots,\alpha_r\\\beta_1,\dots,\beta_s\end{array};Cz\right):=\displaystyle\sum_{n=0}^{\infty}C^n\frac{(\alpha_1)_n\cdots(\alpha_r)_n}{(\beta_1)_n\cdots(\beta_s)_n}\frac{z^n}{n!}=\sum_{n=0}^{\infty}\frac{(e_1n)!\cdots(e_{q_1}n)!}{(f_1n)!\cdots(f_{q_2}n)!}z^n=:F_{(\textbf{e},\textbf{f}\,)}(z),
\end{equation}
où $(x)_{n} := x(x+1)\cdots(x+n-1)$ pour $n\geq 1$ et $(x)_0 = 1$ (symbole de Pochhammer). Nous allons avoir besoin du résultat d'unicité suivant.

\begin{propo}\label{unicité}
Soit $C$ et $C'$ deux constantes strictement positives, $\alpha_1,...,\alpha_r$, $\alpha_1',...,\alpha_{r'}'$, $\beta_1,..., \beta_s$ et $\beta_1',..., \beta_{s'}'$ des complexes où aucun des $\beta_j$ ni aucun des $\beta_j'$ n'est un entier négatif tels que, pour tout $(i,j)\in\{1,\dots,r\}\times\{1,\dots,s\}$, on ait $\alpha_i\neq\beta_j$, et, pour tout $(i,j)\in\{1,\dots,r'\}\times\{1,\dots,s'\}$, on ait $\alpha_i'\neq\beta_j'$. Si, pour tout $n\in\mathbb{N}$, on a
\begin{equation}\label{id uni}
C^n\frac{(\alpha_1)_n\cdots(\alpha_r)_n}{n!(\beta_1)_n\cdots(\beta_s)_n}=C'^n\frac{(\alpha_1')_n\cdots(\alpha_{r'}')_n}{n!(\beta_1')_n\cdots(\beta_{s'}')_n},
\end{equation}
alors on a $C=C'$, $r=r'$, $s=s'$ et il existe une permutation $\sigma\in S_r$ et une permutation $\tau\in S_s$ telles que, pour tout $i\in\{1,\dots,r\}$, on ait $\alpha_i=\alpha_{\sigma(i)}'$ et, pour tout $j\in\{1,\dots,s\}$, on ait $\beta_j=\beta_{\tau(j)}'$.
\end{propo}

\begin{proof}
Soit $i$ dans $\mathbb{N}$, en divisant l'identité \eqref{id uni} avec $n=i+1$ par celle avec $n=i$ on obtient, pour tout $i$ dans $\mathbb{N}$, $i\geq 1$, $C\frac{(\alpha_1+i)\cdots(\alpha_r+i)}{(\beta_1+i)\cdots(\beta_s+i)}=C'\frac{(\alpha_1'+i)\cdots(\alpha_{r'}'+i)}{(\beta_1'+i)\cdots(\beta_{s'}'+i)}$. Ainsi, les deux fractions rationnelles $P(X):=C\frac{(\alpha_1+X)\cdots(\alpha_r+X)}{(\beta_1+X)\cdots(\beta_s+X)}$ et $Q(X):=C'\frac{(\alpha_1'+X)\cdots(\alpha_{r'}'+X)}{(\beta_1'+X)\cdots(\beta_{s'}'+X)}$ sont égales. Elles ont donc les mêmes racines avec mêmes multiplicités et les mêmes pôles avec mêmes ordres. D'où le résultat.
\end{proof}

Afin de caractériser la forme des suites $(\alpha_1,\dots,\alpha_r)$ et $(\beta_1,\dots,\beta_s)$ vérifiant \eqref{écriturunik}, on définit un certain type de partition de multi-ensemble.

\begin{Définition}
Pour $N\in\mathbb{N}$, $N\geq 1$, on note $R_N := \left\{w/N\;:\;1\leq w\leq N,\; (w,N)=1\right\}$.
On dira qu'un multi-ensemble $\{\alpha_1,\dots,\alpha_r\}$ de réels est $R$-partitionné selon le multi-ensemble $\{N_1,\dots,N_k\}$ s'il existe des entiers strictement positifs $N_1,\dots,N_k$ et une partition $E_1,\dots,E_k$ de $\{1,\dots,r\}$ tels que, pour tout $j\in\{1,\dots,k\}$, $Card(E_j)=\varphi(N_j)$ et $\left\{\alpha_i\;:\;i\in E_j\right\}=R_{N_j}$.
\end{Définition}

Par exemple, le multi-ensemble $\{1/3,1/2,1/2,2/3\}$ est $R$-partitionné selon le multi-ensemble $\{2,2,3\}$.

Si $N$ est un entier non nul, on note $ C_N:=N^{\varphi(N)}\prod_{p\mid N} p^{\frac{\varphi(N)}{p-1}}\in\mathbb{N}$, $C_N\geq 1$. Si $\alpha$ est une suite $R$-partitionnée selon $(N_1,\dots,N_k)$, on note $C_{\alpha}:=C_{N_1}\cdots C_{N_k}$. Si $\alpha$ et $\beta$ sont deux suites $R$-partitionnées, on note $C_{(\alpha,\beta)}:=C_{\alpha}/C_{\beta}$.
On note $\textsl{Q}$ l'ensemble des suites de la forme $\mathcal{Q}_{(\textbf{e},\textbf{f}\;)}(n):=\frac{(e_1 n)!\cdots (e_{q_1} n)!}{(f_1 n)!\cdots (f_{q_2} n)!}$, où $q_1,q_2\in\mathbb{N}$, $q_1\,q_2\neq 0$, $e_1,\dots,e_{q_1},f_1,\dots,f_{q_2}\in\mathbb{N}$. On note $\textsl{P}$ l'ensemble des suites de la forme $\mathcal{P}_{(\alpha,\beta)}(n):=C_{(\alpha,\beta)}^n\frac{(\alpha_1)_n\cdots(\alpha_r)_n}{n!(\beta_1)_n\cdots(\beta_s)_n}$, où $r,s\in\mathbb{N}$, $rs\neq 0$, $(\alpha_1,\dots,\alpha_r)$ et $(\beta_1,\dots,\beta_s)$ sont deux suites $R$-partitionnées.

\begin{propo}\label{propo equiv}
On a $\textsl{P}=\textsl{Q}$ et, si $C_{(\alpha,\beta)}^n\frac{(\alpha_1)_n\cdots(\alpha_r)_n}{n!(\beta_1)_n\cdots(\beta_s)_n}=\frac{(e_1 n)!\cdots (e_{q_1} n)!}{(f_1 n)!\cdots (f_{q_2} n)!}$ pour tout $n\geq 0$,
alors $ C_{(\alpha,\beta)}=\frac{e_1^{e_1}\cdots e_{q_1}^{e_{q_1}}}{f_1^{f_1}\cdots f_{q_2}^{f_{q_2}}}$ et $ r-s-1=|\textbf{e}|-|\textbf{f}\,|$.
\end{propo}

Cette proposition caractérise complètement les fonctions hypergéométriques dont les coefficients peuvent se mettre sous forme de quotients de factorielles. Pour démontrer la proposition~\ref{propo equiv}, on va utiliser un lemme dû à Zudilin (\cite[Lemma 4, p. 609]{Zudilin}).

\begin{lemme}[Zudilin]\label{propoZudi 1}
Soit $N\geq 2$ un entier. On écrit $N=p_1^{a_1}p_2^{a_2}\cdots p_{\ell}^{a_{\ell}}$ sa décomposition en produit de facteurs premiers. Pour tout $n\in\mathbb{N}$, on a alors $C_N^n\frac{\prod_{\alpha\in R_N}(\alpha)_n}{(n!)^{\varphi(N)}}=\frac{(e_1 n)!\cdots (e_{q_1} n)!}{(f_1 n)!\cdots (f_{q_2} n)!}$, où
\begin{small}
$$
\{e_i\}_{i=1,...,q_1}=\left\{N,\frac{N}{p_{j_1}p_{j_2}},\frac{N}{p_{j_1}p_{j_2}p_{j_3}p_{j_4}},...\right\}_{1\leq j_1 < j_2 <...\leq \ell},
$$
$$
\{f_j\}_{j=1,...,q_2}=\left\{1,...,1,\frac{N}{p_{j_1}},\frac{N}{p_{j_1}p_{j_2}p_{j_3}},...\right\}_{1\leq j_1 < j_2 <...\leq \ell}
$$
\end{small}
et de plus, $|\textbf{e}|=|\textbf{f}\,|$.
\end{lemme}

\begin{proof}[Démonstration de la proposition \ref{propo equiv}]
Montrons que $\textsl{P}\subset\textsl{Q}$.

Soit $\alpha:=(\alpha_1,\dots,\alpha_r)$ et $\beta:=(\beta_1,\dots,\beta_s)$ deux suites $R$-partitionnées respectivement selon $(N_1,\dots,N_k)$ et $(N_1',\dots,N_{k'}')$. Quitte à réordonner les suites $(N_1,\dots,N_k)$ et $(N_1',\dots,N_{k'}')$, on peut supposer qu'il existe $k_0\in\{0,\dots,k\}$ et $k_0'\in\{0,\dots,k'\}$ tels que, pour tout $i\in\{1,\dots,k_0\}$ et tout $j\in\{1,\dots,k_0'\}$, on ait $N_i\geq 2$ et $N_j'\geq 2$, et, pour tout $i\in\{k_0+1,\dots,k\}$ et tout $j\in\{k_0'+1,\dots,k'\}$, on ait $N_i=1$ et $N_j'=1$.  Pour tout $n\in\mathbb{N}$, on peut alors écrire
\begin{equation}\label{cette}
\frac{(\alpha_1)_n\dots(\alpha_r)_n}{n!(\beta_1)_n\dots(\beta_s)_n}=\frac{\prod_{i=1}^{k_0}\prod_{\alpha\in R_{N_i}}(\alpha)_n}{\prod_{j=1}^{k_0'}\prod_{\beta\in R_{N_j'}}(\beta)_n}(1)_n^{k-k_0-(k'-k_0')-1}.
\end{equation}
On a $C_{(\alpha,\beta)}=\frac{\prod_{i=1}^kC_{N_i}}{\prod_{j=1}^{k'}C_{N_j'}}=\frac{\prod_{i=1}^{k_0}C_{N_i}}{\prod_{j=1}^{k_0'}C_{N_j'}}$, car $C_1=1$. Ainsi, en multipliant \eqref{cette} par $C_{(\alpha,\beta)}^n$, pour tout $n\in\mathbb{N}$, on obtient 
\begin{multline*}
C_{(\alpha,\beta)}^n\frac{(\alpha_1)_n\cdots(\alpha_r)_n}{n!(\beta_1)_n\cdots(\beta_s)_n}=\\
\left(\prod_{i=1}^{k_0} C_{N_i}^n\frac{\prod_{\alpha\in R_{N_i}}(\alpha)_n}{(n!)^{\varphi(N_i)}}\right)\left(\prod_{j=1}^{k_0'}\frac{1}{C_{N_j'}^n}\frac{(n!)^{\varphi(N_j')}}{\prod_{\beta\in R_{N_j'}}(\beta)_n}\right)(n!)^{\sum_{i=1}^{k_0}\varphi(N_i)-\sum_{j=1}^{k_0'}\varphi(N_j')}(1)_n^{k-k_0-(k'-k_0')-1}.
\end{multline*}
On a $(1)_n^{k-k_0-(k'-k_0')-1}=n!^{k-k_0-(k'-k_0')-1}=n!^{\sum_{i=k_0+1}^k\varphi(N_i)-\sum_{j=k_0'+1}^{k'}\varphi(N_j')-1}$, car, pour tout $i\in\{k_0+1,\dots,k\}$ et tout $j\in\{k_0'+1,\dots,k'\}$, on a $\varphi(N_i)=\varphi(N_j')=\varphi(1)=1$.
Ainsi, d'après le lemme \ref{propoZudi 1}, il existe des entiers strictement positifs $e_1,\dots,e_{q_1},f_1,\dots,f_{q_2}$ tels que 
$$
C_{(\alpha,\beta)}^n\frac{(\alpha_1)_n\cdots(\alpha_r)_n}{n!(\beta_1)_n\cdots(\beta_s)_n}=\frac{(e_1 n)!\cdots (e_{q_1} n)!}{(f_1 n)!\cdots (f_{q_2} n)!}(n!)^{\sum_{i=1}^k\varphi(N_i)-\sum_{j=1}^{k'}\varphi(N_j')-1}.
$$
D'où le fait que $\textsl{P}\subset\textsl{Q}$.

\medskip Montrons l'inclusion inverse $\textsl{Q}\subset\textsl{P}$.

Pour tout $\ell\in\{1,\dots,e_1\}$, on a
\begin{equation}\label{eq4}
\left(\frac{\ell}{e_1}\right)_n=\frac{\ell}{e_1}\left(\frac{\ell}{e_1}+1\right)\cdots\left(\frac{\ell}{e_1}+n-1\right)=\frac{\ell}{e_1}\frac{\ell+e_1}{e_1}\cdots\frac{\ell+(n-1)e_1}{e_1}.
\end{equation}
Les numérateurs de \eqref{eq4} correspondent aux nombres dans $\{1,\dots,n e_1\}$ qui sont congrus à $\ell$ modulo $e_1$.
Ainsi, si $e_1\geq 2$, alors on a
$$
(e_1 n)!=\left(\frac{1}{e_1}\right)_n \left(\frac{2}{e_1}\right)_n\cdots\left(\frac{e_1}{e_1}\right)_n e_1^{e_1 n}=\left(\frac{1}{e_1}\right)_n\cdots\left(\frac{e_1-1}{e_1}\right)_n (1)_n \left(e_1^{e_1}\right)^n
$$
et si $e_1=1$, alors on a simplement $n!=(1)_n$. On peut donc écrire
\begin{equation}\label{eq1}
\frac{(e_1 n)!\cdots (e_{q_1} n)!}{(f_1 n)!\cdots (f_{q_2} n)!}=C^n\frac{\prod_{\underset{e_i\geq 2}{1\leq i\leq q_1}}\left(\left(\frac{1}{e_i}\right)_n\cdots\left(\frac{e_i-1}{e_i}\right)_n\right)}{n!\prod_{\underset{f_j\geq 2}{1\leq j\leq q_2}}\left(\left(\frac{1}{f_j}\right)_n\cdots\left(\frac{f_j-1}{f_j}\right)_n\right)}(1)_n^{q_1-q_2+1},
\end{equation}
où $C:=\frac{e_1^{e_1}\cdots e_{q_1}^{e_{q_1}}}{f_1^{f_1}\cdots f_{q_2}^{f_{q_2}}}$. On a donc $r-s-1=\sum_{i=1}^{q_1} (e_i-1)-\sum_{j=1}^{q_2} (f_j-1)+q_1-q_2+1-1=|\textbf{e}|-|\textbf{f}\,|$.

Si $e_1\geq 2$, alors on a
\begin{equation}\label{R-partition}
\begin{small}
\left\{\frac{1}{e_1},\dots,\frac{e_1-1}{e_1}\right\}=\bigcup_{d\mid e_1,\; d\geq 2}R_d.
\end{small}
\end{equation}
En effet, notons $e_1=p_1^{a_1}\cdots p_s^{a_s}$ la décomposition en facteurs premiers de $e_1$. On a 
$$
\left\{1,\dots,e_1-1\right\}=\underset{(b_1,\dots,b_s)\neq(a_1,\dots,a_s)}{\bigcup_{0\leq b_1\leq a_1,\dots,0\leq b_s\leq a_s}}\left\{p_1^{b_1}\cdots p_s^{b_s}r \;:\; (r, p_1\cdots p_s)=1, 1\leq r<p_1^{a_1-b_1}\cdots p_s^{a_s-b_s}\right\}.
$$
Ainsi, on obtient
\begin{small}
\begin{align*}
\left\{\frac{1}{e_1},\dots,\frac{e_1-1}{e_1}\right\}
&=\underset{(b_1,\dots,b_s)\neq(a_1,\dots,a_s)}{\bigcup_{0\leq b_1\leq a_1,\dots,0\leq b_s\leq a_s}}\left\{\frac{r}{p_1^{a_1-b_1}\cdots p_s^{a_s-b_s}}\;:\; (r, p_1\cdots p_s)=1, r<p_1^{a_1-b_1}\cdots p_s^{a_s-b_s}\right\}\\
&=\bigcup_{d\mid e_1,\; d\geq 2}R_d.
\end{align*}
\end{small}
En utilisant l'identité \eqref{R-partition} dans \eqref{eq1}, on obtient
$$
\frac{(e_1 n)!\cdots (e_{q_1} n)!}{(f_1 n)!\cdots (f_{q_2} n)!}=C^n\frac{\prod_{i=1}^{q_1}\prod_{\underset{d\geq 2}{d\mid e_i}}\prod_{\alpha\in R_d}(\alpha)_n}{n!\prod_{j=1}^{q_2}\prod_{\underset{d\geq 2}{d\mid f_j}}\prod_{\beta\in R_d}(\beta)_n}\left(\prod_{\gamma\in R_1}(\gamma)_n\right)^{q_1-q_2+1},
$$
où ici $\prod_{\gamma\in R_1}(\gamma)_n=(1)_n$. Si $q_1-q_2+1$ est positif alors on regroupe le produit des $\gamma$ avec celui des $\alpha$ au numérateur, et si $q_1-q_2+1$ est négatif alors on regroupe le produit des $\gamma$ avec celui des $\beta$ au dénominateur. On indexe les $\alpha$ et $\beta$ afin d'obtenir l'écriture 
\begin{equation}\label{eq5}
\frac{(e_1 n)!\cdots (e_{q_1} n)!}{(f_1 n)!\cdots (f_{q_2} n)!}=C^n\frac{(\alpha_1)_n\cdots(\alpha_r)_n}{n!(\beta_1)_n\cdots(\beta_s)_n},
\end{equation}
où les suites $\alpha:=(\alpha_i)_{i=1,\dots,r}$ et $\beta:=(\beta_j)_{j=1,\dots,s}$ sont $R$-partitionnées.

Il ne reste plus qu'à montrer que $C=C_{(\alpha,\beta)}$. Comme $\textsl{P}\subset\textsl{Q}$, il existe des entiers strictement positifs $e_1',\dots,e_{q_1'}',f_1',\dots,f_{q_2'}'$ tels que, pour tout $n\in\mathbb{N}$, on ait
\begin{equation}\label{eq6}
C_{(\alpha,\beta)}^n\frac{(\alpha_1)_n\cdots(\alpha_r)_n}{n!(\beta_1)_n\cdots(\beta_s)_n}=\frac{(e_1' n)!\cdots (e_{q_1'}' n)!}{(f_1' n)!\cdots (f_{q_2'}' n)!}.
\end{equation}
En divisant le terme de gauche de l'égalité \eqref{eq5} par le terme de droite de l'égalité \eqref{eq6}, on obtient, pour tout $n\in\mathbb{N}$,
\begin{equation}\label{eq7}
\frac{(e_1 n)!\cdots (e_{q_1} n)!(f_1' n)!\cdots (f_{q_2'}' n)!}{(f_1 n)!\cdots (f_{q_2} n)!(e_1' n)!\cdots (e_{q_1'}' n)!}=\left(\frac{C}{C_{(\alpha,\beta)}}\right)^n.
\end{equation}
Raisonnons par l'absurde et supposons que $C$ soit différent de $C_{(\alpha,\beta)}$. Il existe deux suites d'entiers strictement positifs disjointes $(c_1,\dots,c_k)$ et $(d_1,\dots,d_{\ell})$  telles que, pour tout $n\in\mathbb{N}$, on ait 
$$
\frac{(e_1 n)!\cdots (e_{q_1} n)!(f_1' n)!\cdots (f_{q_2'}' n)!}{(f_1 n)!\cdots (f_{q_2} n)!(e_1' n)!\cdots (e_{q_1'}' n)!}=\frac{(c_1 n)!\cdots (c_k n)!}{(d_1 n)!\cdots (d_{\ell} n)!}.
$$
Comme $C/C_{(\alpha,\beta)}\neq 1$ et d'après l'identité \eqref{eq7}, $(c_1,\dots,c_k)$ et $(d_1,\dots,d_{\ell})$ ne peuvent être simultanément vides.
Soit $M:=\max(c_1,\dots,c_k,d_1,\dots,d_{\ell})$. D'après \eqref{eq1} et par unicité de l'écriture sous forme de coefficients hypergéométriques \textit{i.e.} proposition \ref{unicité}, le symbole de Pochhammer $(1/M)_n$ devrait apparaître dans le terme de droite de l'égalité \eqref{eq7}, mais ce n'est pas le cas. D'où la contradiction.
\end{proof}

\subsection{Description des sauts de l'application $\Delta$ de Landau}\label{section delta facto}

Dans cette partie, nous allons montrer que les abscisses et les amplitudes des sauts effectués par $\Delta_{(\textbf{e},\textbf{f}\;)}$ sur $[0,1]$ apparaissent naturellement dans la réécriture de $\mathcal{Q}_{(\textbf{e},\textbf{f}\;)}$ sous la forme $\mathcal{P}_{(\alpha,\beta)}$.

\begin{propo}\label{veri num conj}
Soit $\textbf{e}$ et $\textbf{f}$ deux suites finies d'entiers strictement positifs. On peut écrire de manière unique la suite $\mathcal{Q}_{(\textbf{e},\textbf{f}\;)}$ sous la forme
\begin{equation}\label{écriture 3}
\mathcal{Q}_{(\textbf{e},\textbf{f}\;)}(n)=C^n(\gamma_1)_n^{m_1}\dots (\gamma_t)_n^{m_t},\,n\geq 0,
\end{equation}
où $C$ est une constante strictement positive, $0<\gamma_1<\dots<\gamma_t\leq 1$ sont des rationnels et les $m_1,\dots,m_t$ sont dans $\mathbb{Z}\setminus\{0\}$. Les sauts de $\Delta_{(\textbf{e},\textbf{f}\;)}$ sur $[0,1]$ se font aux abscisses $\gamma_1,\dots,\gamma_t$ avec $m_i$ pour amplitude (positive ou négative) en $\gamma_i$.
\end{propo}

\begin{proof}
On écrit $\textbf{e}:=(e_1,\dots,e_{q_{1}})$ et $\textbf{f}:=(f_1,\dots,f_{q_2})$. L'existence et l'unicité de l'écriture \eqref{écriture 3} découlent respectivement des propositions \ref{propo equiv} et \ref{unicité} de la partie \ref{sectionunicité}.
\'Etudions les sauts de la fonction $\Delta_{(\textbf{e},\textbf{f}\;)}$. 

Si $c\in\mathbb{N}$, $c\geq 1$, alors la fonction $[0,1]\longrightarrow\mathbb{Z},x\longmapsto \lfloor cx\rfloor$ effectue un saut d'amplitude 1 en $\frac{1}{c},\frac{2}{c},\dots,\frac{c-1}{c}$ et 1. D'après \eqref{eq1}, pour tout $n\geq 0$, on a 
$$
\mathcal{Q}_{(\textbf{e},\textbf{f}\;)}(n)=C^n\frac{\left(\frac{1}{e_1}\right)_n\dots\left(\frac{e_1-1}{e_1}\right)_n\dots\left(\frac{1}{e_{q_1}}\right)_n\dots\left(\frac{e_{q_1}-1}{e_{q_1}}\right)_n}{\left(\frac{1}{f_1}\right)_n\dots\left(\frac{f_1-1}{f_1}\right)_n\dots\left(\frac{1}{f_{q_2}}\right)_n\dots\left(\frac{f_{q_2}-1}{f_{q_2}}\right)_n}(1)_n^{{q_1}-{q_2}}.
$$
Ainsi, en simplifiant le quotient et en regroupant les symboles de Pochhammer identiques, on obtient l'écriture $\mathcal{Q}_{(\textbf{e},\textbf{f}\;)}(n)=C^n(\gamma_1)_n^{m_1}\dots (\gamma_t)_n^{m_t}$, $n\geq 0$, où les $\gamma_1<\dots<\gamma_t$ correspondent effectivement aux abscisses des sauts de $\Delta_{(\textbf{e},\textbf{f}\;)}$ et les $m_i$, $1\leq i\leq t$, à leur amplitude.
\end{proof}

\begin{Remarque}
Des résultats analogues à ceux de la proposition \ref{veri num conj} sont mentionnés dans \cite{Villegas} par Rodriguez-Villegas puis précisés dans \cite{Bober} par Bober.
\end{Remarque}

Supposons que, pour tout $x\in[0,1]$, on ait $\Delta_{(\textbf{e},\textbf{f}\,)}(x)\geq 0$. Alors, comme $\Delta_{(\textbf{e},\textbf{f}\,)}(0)=\Delta_{(\textbf{e},\textbf{f}\,)}(\gamma_t)=0$, $\gamma_1$ correspond à un saut d'amplitude positive et $\gamma_t$ correspond à un saut d'amplitude négative. On remarque grâce à l'identité \eqref{eq1} que $\gamma_1=1/M_{(\textbf{e},\textbf{f}\,)}$ et $\gamma_t=(M_{(\textbf{e},\textbf{f}\,)}-1)/M_{(\textbf{e},\textbf{f}\,)}$ ou $1$, où $M_{(\textbf{e},\textbf{f}\,)}=\max(e_1,\dots,e_{q_1},f_1,\dots,f_{q_2})$. Or $\gamma_1$ correspond à un saut d'amplitude positive donc $M_{(\textbf{e},\textbf{f}\,)}=\max(e_1,\dots,e_{q_1}):=M_{\textbf{e}}$ et $\gamma_t=1$ car $(M_{\textbf{e}}-1)/M_{\textbf{e}}$ correspond aussi à l'abscisse d'un saut d'amplitude positive. En résumé, on a $\gamma_1=1/M_{\textbf{e}}$, $\gamma_t=1$ et, pour tout $x\in\left[(M_{\textbf{e}}-1)/M_{\textbf{e}},1\right[$, on a $\Delta_{(\textbf{e},\textbf{f}\,)}(x)\geq 1$.
\medskip

Supposons de plus que, pour tout $x\in[1/M_{\textbf{e}},1[$, on ait $\Delta_{(\textbf{e},\textbf{f}\,)}(x)\geq 1$. On sait alors que $(M_{\textbf{e}}-1)/M_{\textbf{e}}$ correspond à l'abscisse d'un saut d'amplitude positive. 

Si $M_{\textbf{e}}\geq 3$, alors on a $(M_{\textbf{e}}-1)/M_{\textbf{e}}>1/M_{\textbf{e}}$ et, pour tout $x\in\left[1/M_{\textbf{e}};(M_{\textbf{e}}-1)/M_{\textbf{e}}\right[$ on a $\Delta_{(\textbf{e},\textbf{f}\,)}(x)\geq 1$, donc pour tout $x\in\left[(M_{\textbf{e}}-1)/M_{\textbf{e}},1\right[$ on a $\Delta_{(\textbf{e},\textbf{f}\,)}(x)\geq 2$. Ainsi $m_t\leq -2$, ce qui signifie que dans l'écriture des coefficients hypergéométriques sous la forme $C_{(\alpha,\beta)}^n\frac{(\alpha_1)_n\dots(\alpha_r)_n}{(1)_n(\beta_1)_n\dots(\beta_{r-1})_n}$, au moins un des $\beta_i$ est égal à 1. D'après \cite[p. 310]{Dwork 1}, ceci implique que l'équation différentielle fuchsienne associée à $F_{(\textbf{e},\textbf{f}\,)}$ admet une solution de type logarithmique à l'origine $G_{(\textbf{e},\textbf{f}\,)}(z)+\log(z)F_{(\textbf{e},\textbf{f}\,)}(z)$, où $G_{(\textbf{e},\textbf{f}\,)}$ est définie par l'identité \eqref{definition de G facto} donnée dans la partie \ref{section mirror map}.

Si $M_{\textbf{e}}=2$, alors $\mathcal{Q}_{(\textbf{e},\textbf{f}\,)}(n)$ est de la forme $\mathcal{Q}_{(\textbf{e},\textbf{f}\,)}(n)=\frac{(2n)!^j}{n!^{2j}}=\left(2^{2j}\right)^n\frac{\left(1/2\right)_n^j}{(1)_n^j}$, où $j\in\mathbb{N}$, $j\geq 1$. Ainsi le seul couple de suites $(\textbf{e},\textbf{f}\,)$ dont l'équation différentielle hypergéométrique associée n'admet pas de solution de type logarithmique à l'origine correspond au cas $j=1$.
\medskip

Nous allons maintenant montrer que si $\textbf{e}$ et $\textbf{f}$ sont deux suites d'entiers strictement positifs disjointes vérifiant $|\textbf{e}|=|\textbf{f}\,|$, alors les assertions suivantes sont équivalentes.
\medskip
\begin{itemize}
\item[$(i)$] $\Delta_{(\textbf{e},\textbf{f}\,)}$ est croissante sur $[0,1[$.
\item[$(ii)$] Il existe des entiers strictement positifs $N_1,\dots,N_k$ tels que $\textbf{e}$ est constituée des éléments du multi-ensemble $\bigcup_{i=1}^kA_{N_i}$ et $\textbf{f}$ est constituée des éléments du multi-ensemble $\bigcup_{i=1}^kB_{N_i}$, où les multi-ensembles $A_{N_i}$ et $B_{N_i}$ sont définis comme dans la partie \ref{Enonce du crit}.
\item[$(iii)$] L'équation différentielle fuchsienne associée à $F_{(\textbf{e},\textbf{f}\,)}$ a tous ses exposants égaux à $0$ à l'origine.
\end{itemize}
\medskip

Soit $(\alpha_1,\dots,\alpha_r)$ et $(\beta_1,\dots,\beta_s)$ les suites telles que $\mathcal{Q}_{(\textbf{e},\textbf{f}\,)}=\mathcal{P}_{(\alpha,\beta)}$.

Si $\Delta_{(\textbf{e},\textbf{f}\,)}$ est croissante sur $[0,1[$ alors $\beta_1=\dots=\beta_s=1$ et l'équation différentielle fuchsienne associée à $F_{(\textbf{e},\textbf{f}\,)}$ a donc tous ses exposants égaux à $0$ à l'origine (voir \cite{Schwarz}). Ainsi, on a $(i)\Rightarrow (iii)$. De plus, comme $\alpha$ est $R$-partitionnée, il existe des entiers strictement positifs $N_1,\dots,N_k$ tels que $\{\alpha_1,\dots,\alpha_r\}=\bigcup_{i=1}^kR_{N_i}$. Ainsi, d'après le lemme~\ref{propoZudi 1}, on obtient bien $(i)\Rightarrow(ii)$. 

Réciproquement, s'il existe des entiers strictement positifs $N_1,\dots,N_k$ tels que $\textbf{e}$ est constituée des éléments du multi-ensemble $\bigcup_{i=1}^kA_{N_i}$ et $\textbf{f}$ est constituée des éléments du multi-ensemble $\bigcup_{i=1}^kB_{N_i}$ alors le lemme \ref{propoZudi 1} montre que $\beta_1=\dots=\beta_s=1$. On a donc bien $(i)\Leftrightarrow(ii)$.

Si l'équation différentielle fuchsienne associée à $F_{(\textbf{e},\textbf{f}\,)}$ a tous ses exposants égaux à $0$ à l'origine, alors $\beta_1=\dots=\beta_s=1$ et $\Delta_{(\textbf{e},\textbf{f}\,)}$ est croissante sur $[0,1[$. On obtient bien $(i)\Leftrightarrow(iii)$.

\address{E. Delaygue, Institut Fourier, CNRS et Université Grenoble 1, 100 rue des Maths, BP 74, 38402 Saint-Martin-d'Hères cedex, France. Email : Eric.Delaygue@ujf-grenoble.fr}

\end{document}